\documentclass[11pt,a4paper]{article}
\usepackage[reset, a4paper, vmargin=3truecm, hmargin=2truecm]{geometry}
\usepackage{amsmath, amssymb, amsfonts}
\usepackage{tikz}

\usepackage{tikz}
\usepackage{amsthm}
\theoremstyle{plain}
 \newtheorem{thm}{Theorem}[section]
 \newtheorem{theorem}{Theorem}[section]      % new
 
 \newtheorem{lemma}[thm]{Lemma}              % lem
 \newtheorem{proposition}[thm]{Proposition}  % prop
\theoremstyle{definition}
 \newtheorem{definition}[thm]{Definition}    % defn
 \newtheorem{exmp}[thm]{Example}

 \newtheorem{remark}[thm]{Remark}            % rem

\DeclareMathOperator{\ord}{ord}
\DeclareMathOperator{\idx}{idx}

\DeclareMathOperator{\trace}{trace}
\DeclareMathOperator{\mc}{mc}

\DeclareMathOperator{\Hd}{Hd}

\DeclareMathOperator{\Hu}{\mathcal S}
\DeclareMathOperator{\He}{\mathcal T}
\DeclareMathOperator{\Hr}{\mathcal U}
\DeclareMathOperator{\Np}{Np}

\numberwithin{equation}{section}

\title{A Kac--Moody root system and linear ordinary differential equations}
\author{Toshio Oshima\footnote{\texttt{oshima@ms.u-tokyo.ac.jp}\\
\phantom{\quad} 2020 Mathematics Subject Classification. Primary 34M03; Secondary 34M55\\
\phantom{\quad} Key words and phrases. Harnad dual, confluence, middle convolution, Kac-Moody root system, Painlev\'e equation}}
\date{}
\begin{document}
\maketitle
\begin{abstract}
We show that every irreducible linear differential equations on $\mathbb P^1$ with 
regular and/or unramified irregular singularities can be transformed into either 
the trivial equation $u'=0$ or a Fuchsian system of $E_8$-fundamental spectral type 
by a sequence of invertible transformations consisting of confluences, unfoldings, Laplace 
transformations, gauge transformations and M\"obius transformations. 
The $E_8$-spectral type is uniquely determined by the original equation. 
A Fuchsian system of $E_8$-fundamental spectral type has three singular points, 
$0$, $1$, and $\infty$. 
The degrees of the minimal polynomials of the residue matrices at $0$ and $1$ 
are two and three, respectively, 
and the sum of the maximal dimensions of eigenspaces of the residue 
matrices at $0$, $1$, and $\infty$
is not greater than the size of the matrices.
The result follows from the correspondence between spectral types of Fuchsian 
systems and roots of a star-shaped Kac--Moody root system. 
We show that the $E_8$-fundamental spectral types form a complete set of 
representatives for the orbits of the corresponding transformations
on the set of positive imaginary roots.
\end{abstract}

\normalsize
%%%%%%%%%%%%%%%%%%%%%%  1  %%%%%%%%%%%%%%%%%%%%
\section{Introduction}
%%%%%%%%%%%%%%%%%%%%%%%%%%%%%%%%%%%%%%%%%%%%%%%
Consider a linear differential equation
\begin{align}\label{eq:system}
\frac{du}{dx}=A(x)u,\qquad A(x)\in M(n,\mathbb C(x)).
\end{align}
Here, $A(x)$ is a square matrix whose entries are rational functions. We write
\begin{align}\label{eq:irsystem}
A(x)=\sum_{j=1}^{p-1}\sum_{i=0}^{r_j}\frac{A_{j,i}}{(x-c_j)^{i+1}}
+\sum_{i=1}^{r_p}A_{p,i-1}x^{i-1},
\qquad A_{j,i}\in M(n,\mathbb C).
\end{align}
The equation is called a Fuchsian system if $r_1=\cdots=r_p=0$. In this case,
$A_j:=A_{j,0}$ is called the residue matrix at $x=c_j$.
The residue matrix at $x=\infty$ is given by
$
A_p:=-A_1-\cdots-A_{p-1}
$.

Katz~\cite{katz1996rigid} introduced invertible operations, called middle convolutions and middle tensor operations, on local systems and proved that any rigid local system can be transformed into the trivial rank-one system.
Dettweiler and Reiter~\cite{DR} subsequently interpreted this result in terms of Fuchsian systems, introducing two operations on the set of residue matrices
${A_1,\dots,A_p}$: middle convolutions and additions.
These operations correspond, respectively, to fractional derivatives and gauge transformations of solutions of the differential equations.
This result was further interpreted by Oshima~\cite{Ow} for scalar linear differential equations of arbitrary order.
These transformations have led to a number of fundamental results on linear ordinary differential equations; see, for example, \cite{Ow}.

The local structure of a Fuchsian system is determined by its generalized Riemann scheme, which specifies the conjugacy classes of its residue matrices.
It is a table of the eigenvalues of the residue matrices together with their multiplicities.
The table obtained by omitting the eigenvalues is called the spectral type, which is a $p$-tuple of partitions of an integer $n$.
The Riemann scheme is obtained by attaching eigenvalues to a spectral type so that the Fuchs relation
\[
\trace  A_1+\cdots+\trace A_p=0
\]
is satisfied.
A Fuchsian system is said to be rigid if it is uniquely determined by its Riemann scheme.

In general, a Fuchsian system corresponding to a generalized Riemann scheme has accessory parameters.
The number of accessory parameters is determined by the index of rigidity, which in turn depends only on the spectral type.
Katz~\cite{katz1996rigid} proved that an irreducible Fuchsian system is rigid if and only if its index of rigidity equals $2$.
In this case, the system can be transformed into the trivial equation $u'=0$ by a sequence of Katz reductions.
Here, a reduction of the rank of a differential equation by means of an addition and a middle convolution is called a Katz reduction.

We say that a Riemann scheme or a spectral type is irreducibly realizable if it is realized by an irreducible Fuchsian system.
The problem of determining the conditions for such realizability is called the additive Deligne--Simpson problem (cf.~\cite{Kostov}).
Crawley-Boevey~\cite{CB} solved this problem by establishing a correspondence between spectral types and elements of the root lattice of a star-shaped Kac--Moody root system.
He proved that a spectral type is irreducibly realizable if and only if the corresponding element is a positive root.
Under this correspondence, middle convolutions and additions correspond to the actions of the Weyl group, while the index of rigidity corresponds to the norm of the corresponding element of the root lattice.
In particular, an irreducibly realizable rigid spectral type corresponds to a positive real root, whereas an irreducibly realizable non-rigid spectral type corresponds to a positive imaginary root.
The latter can be reduced to a fundamental spectral type corresponding to a positive imaginary root in a fundamental domain of the Weyl group
(cf.~\cite[Theorem~5.4]{KacBook}).
There are only finitely many fundamental spectral types with a given index of rigidity (cf.~\cite{Ow}).

Hiroe~\cite{HiroeDuke} formulated and solved a Deligne--Simpson problem for certain systems of the form \eqref{eq:system} without ramified irregular singularities. In \cite{HiroeDuke,Hiroe2025}, he proved that a tuple of partitions with refinement structure is irreducibly realizable if and only if its unfolding is irreducibly realizable.
In this paper, we consider systems allowing unramified irregular singularities.

Note that there exists a versal unfolding of an irreducible system allowing unramified irregular 
singularities, which realizes both confluences and unfoldings as a family of systems \eqref{eq:system} whose singular points are local holomorphic parameters.

The finiteness of fundamental spectral types with a given index of rigidity also holds for systems allowing unramified irregular singularities (cf.~\cite{HiroeOshima}).

Isomonodromic deformations of Fuchsian systems with accessory parameters give rise to Painlev\'e-type equations, while their degenerations correspond to confluences of the systems.
These equations are invariant under additions and middle convolutions.
Thus, the classification of fundamental spectral types is closely related to the classification of Painlev\'e-type equations, as illustrated by the study of four-dimensional Painlev\'e-type equations in \cite{HKNS}.

In a suitable enlargement of the space of Fuchsian systems, one can consider Laplace transformations, which are more fundamental than middle convolutions.
For example, in \cite{Ow}, the author introduced operators corresponding to middle convolutions in terms of Laplace transformations and additions (cf.~Remark~\ref{remark:Hd} (2)).

In the space of systems allowing unramified irregular singularities, we say that two systems are equivalent if one is transformed into the other by a gauge transformation, a Laplace transformation, a confluence, an unfolding, or a M\"obius transformation. We then consider the equivalence relation generated by these transformations.

It is therefore natural to study the equivalence classes of this relation.
This problem was proposed by Kawakami at a conference on functional equations in a complex domain held at Tohoku University in December 2025.
Our main purpose is to show that the trivial equation $u'=0$ and the Fuchsian systems of $E_8$-fundamental spectral types form a complete set of representatives of these equivalence classes.
These systems are of the form
\begin{align*}
\frac{du}{dx}
&=\frac {A_1}{x}u+\frac{A_2}{x-1}u,
\qquad A_1,A_2\in M(n,\mathbb C),
\end{align*}
where
\begin{align*}
\begin{split}
 &(A_1-\lambda_{11})(A_1-\lambda_{12})
 =(A_2-\lambda_{21})(A_2-\lambda_{22})(A_2-\lambda_{23})=0
 \qquad(\exists\lambda_{j\nu}\in\mathbb C),\\
 &\sum_{j=1}^3\max_{\lambda\in\mathbb C}
 \dim\ker(A_j-\lambda)\le n,
 \qquad A_3:=-A_1-A_2.
\end{split}
\end{align*}

In Section \ref{sec:pre}, we explain the transformations introduced above and define the necessary notation.
Most of the results in this section are known.
We also define $E_8$-, $E_7$-, $E_6$-, and $D_4$-fundamental spectral types and formulate our problem precisely.
In Section \ref{sec:example}, we give examples of the equivalence classes.

In Section \ref{sec:main}, we state our main result, Theorem~\ref{thm:main}.
To prove the theorem, we define a map $\mathcal U$ from the space of irreducibly realizable spectral types to the space consisting of $E_8$-fundamental spectral types and the trivial partition.
The map $\mathcal U$ is a composition of additions, M\"obius transformations, confluences, unfoldings, Laplace transformations, and middle convolutions.
Theorem~\ref{thm:main} is reduced to Lemma~\ref{lem:main}, which describes a key property of the map $\mathcal U$.
We prove the lemma in this section, except for the claim that the representative has the 
highest order among fundamental spectral types in every equivalence class. 

In Section \ref{sec:RSm}, we compute the restriction of the map $\mathcal U$ to the space of 
fundamental spectral types.
The result is given in Theorem~\ref{thm:RSm}, and the proof of Lemma~\ref{lem:main} is completed 
by Proposition~\ref{cor:RSm}, which follows from this theorem.
Moreover, for a $T$-fundamental spectral type $\mathbf m$ with $T=E_7$, $E_6$, or $D_4$, we give 
an explicit expression for $\mathcal U\mathbf m$, which clarifies the $T$-fundamental spectral 
types contained in each equivalence class.

%%%%%%%%%%%%%%%%%%%%%%%%%%%%%%%%%%%%%%%%%%%%%%%%%%%%%%%%%%%%%%%
\section{Preliminary results}\label{sec:pre}
%%%%%%%%%%%%%%  2.1 %%%%%%%%%%%%%%%%%%%%%%%%%%%%%%%%%%%%%%%%%%%
\subsection{Spectral types of Fuchsian differential equations}
%%%%%%%%%%%%%%%%%%%%%%%%%%%%%%%%%%%%%%%%%%%%%%%%%%%%%%%%%%%%%%%
First we consider a linear differential equation
\begin{align}\label{eq:Fsystem}
  \mathcal M: \frac{du}{dx}=\sum_{j=1}^{p-1}\frac{A_j}{x-c_j}u,\qquad A_j\in M(n,\mathbb C),
\end{align}
which we call a {\em Fuchsian system}.
A linear ordinary differential equation whose coefficients are polynomials or rational functions,
or equivalently, a linear differential equation on the Riemann sphere, 
can be presented by a Fuchsian system if it is irreducible and all its singular points are regular.
Here, the irreducibility is understood either in the sense of $\mathcal D$-module or 
in terms of the monodromy representation of the solution space.  
We say that the Fuchsian system is {\em irreducible} if the matrices $A_1,\dots,A_p$.
have no non-trivial common invariant subspace. 
The matrix $A_j$ is called the {\em residue matrix} at the singular point $c_j$.
The residue matrix at the infinite point is denoted by $A_p=A_\infty$ and is given by
\begin{align}\label{eq:sum0}
  A_p=-(A_1+\cdots+A_{p-1}).
\end{align}
The conjugacy class of $A_j$ is determined by generalized eigenvalues $\lambda_{j,\nu}$
with multiplicities $m_{j,\nu}$ written by 
$[A_j]=\{[\lambda_{j,\nu}]_{m_{j,\nu}}\mid \nu=1,\dots,n_j\}$.
If the multiplicities are {\em ordered}, namely,
\begin{align*}%\label{eq:ordered}
 m_{j,1}\ge m_{j,2}\ge \cdots \ge m_{j,n_j}\quad(j=1,\dots,p),
\end{align*}
which can be assumed by arranging the order of indices, 
the condition imposed on $A_j$ is
\begin{align*}
\mathrm{rank}\,\prod\limits_{\nu=1}^k (A_j-\lambda_{j,\nu})=n-m_{j,1}-\cdots- m_{j,k}
\quad(k=1,\dots,n_j).
\end{align*}
The (generalized) eigenvalues of $A_j$ are called {\em characteristic exponents} and 
their multiplicities determine the spectral type
\begin{align}\label{eq:partn}
n=m_{j,1}+\cdots+m_{j,n_j}\quad(j=1,\dots,p).
\end{align}
The (generalized) {\em Riemann scheme} (cf.~\cite{Ow}) 
\begin{align*}
\begin{Bmatrix} x=c_j\ (j=1,\dots,p-1) & x=\infty\\
    [\lambda_{j,1}]_{m_{j,1}} & [\lambda_{p,1}]_{m_{p,1}} \\[-.5mm]
    \vdots & \vdots\\[-.5mm]
   [\lambda_{j,n_j}]_{m_{j,n_j}}&  [\lambda_{p,n_p}]_{m_{p,n_p}}
\end{Bmatrix}
\end{align*}
describes the local structure of the equation at the singular points.
In particular, if the characteristic exponents are generic, the local solutions have 
the asymptotic forms
\begin{equation*}
 \begin{aligned}
 u&\sim C_{j,\nu,i}(x-c_j)^{\lambda_{j,\nu}}
   &&(x\to c_j,\ \ i=1,\dots,m_{j,\nu},\ \nu=1,\dots,n_j),\\
 &\sim C_{p,\nu,i}(\tfrac1x)^{\lambda_{p,\nu}}
   &&(x\to\infty,\ \ i=1,\dots,m_{p,\nu},\ \nu=1,\dots,n_p).
 \end{aligned}
\end{equation*}
Here $C_{j,\nu,1},\ldots,C_{j,\nu,m_{j,\nu}}$ are linearly independent vectors.

\begin{definition} 
A {\em spectral type} is a tuple of partitions \eqref{eq:partn} of an integer $n$.
We express it in several equivalent ways:
\begin{align}\label{eq:m}
\begin{split}
\mathbf m&:=\bigl(m_{j,\nu}\bigr)_{\substack{1\le \nu\le n_j\\ 1\le j\le p}}\\
 &=m_{1,1}\cdots m_{1,n_1},\,\ldots,\,m_{p-1,1}\cdots m_{p-1,n_{p-1}},\,m_{p,1}\cdots m_{p,n_p}\\
 &=\bigl[[m_{1,1},\ldots, m_{1,n_1}],\ldots,[m_{p-1,1},\ldots,m_{p-1,n_{p-1}}],[m_{p,1},
\ldots,m_{p,n_p}]\bigr].
\end{split}\end{align}
The integer $n$ is called the {\em rank} or {\em order} of $\mathbf m$ and is denoted by $\ord\mathbf m$.
The spectral type $\mathbf m$ determined by the residue matrices of a Fuchsian system $\mathcal M$
is called the spectral type of $\mathcal M$. 
A spectral type realized by an irreducible Fuchsian system is called 
{\em irreducibly realizable}.
\end{definition}

\begin{exmp}
The generalized hypergeometric series ${}_nF_{n-1}(x)$ satisfies a linear differential equation of 
order $n$ having an  $(n{-}1)$-dimensional space of local holomorphic solutions at $x=1$.
Its spectral type is
\begin{align*}
{}_nF_{n-1} : n=&\overbrace{1+\cdots+1}^n=(n-1)+1=\overbrace{1+\cdots+1}^n,\\
                & 1\cdots 1,(n-1)1,1\cdots 1=1^n,(n-1)1,1^n\\
                & \bigl[[1,\ldots,1],[n-1,1],[1,\ldots,1]\bigr]. 
\end{align*}
We give some examples of differential equations characterized by their spectral types:

\medskip
Gauss\,:\,$11,11,11\ $\quad($\ord=2$,\ \ singularities : $\{0,1,\infty\}$),

${}_3F_2:\ \ \,111,21,111$\ \ \ ($\ord=3$,\ \ singularities : $\{0,1,\infty\}$,\ \text{2-dimensional holomorphic solutions at $x=1$}),

Heun: \,$11,11,11,11$\ ($\ord=2$,\ \ singularities : $\{0,1,y,\infty\}$)

\quad
Monodromy preserving deformation with respect to $y$ gives Painlev\'e VI equation,

Jordan-Pochhammer : $21,21,21,21$\ \  ($\ord=3$,\ \ singularities : $\{0,1,y,\infty\}$)

\quad
Regarding $y$ as a new variable, we obtain Appell's $F_1$.
\end{exmp}

%%%%%%%%%%%%%%%%%%%%%%%%%%%%%%%%%%%%%%%%%
%%%%%%%%%%%%%%%%%%% 2.2 %%%%%%%%%%%%%%%%%
\subsection{Middle convolutions}
%%%%%%%%%%%%%%%%%%%%%%%%%%%%%%%%%%%%%%%%%%
Riemann-Liouville transformations 
\[u(x)\mapsto \frac1{\Gamma(\mu)}\int_{c_j}^x u(t)(x-t)^{\mu-1}dt\] 
and gauge transformations
\[u(x)\mapsto (x-c)^\lambda u(x)\]
of solutions $u(x)$ of the system $\mathcal M$ induce invertible 
transformations of $\mathcal M$ that preserve irreducibility. They also 
induce transformations of the spectral type $\mathbf m$ of $\mathcal M$,
which we call middle convolutions of $\mathbf m$.
%%%
\begin{definition}[middle convolution, \cite{DR}, \cite{katz1996rigid}]\label{def:mc}
For $\sigma\in\mathbb Z_{\ge0}^p$, the middle convolution
$\mc_\sigma\mathbf m$ of $\mathbf m$ is defined by
\begin{align*}
\mc_\sigma\mathbf m
&:=\Bigl(m_{j,\nu}-\delta_{\nu,\sigma_j}d_{\sigma}(\mathbf m)\Bigr)_{\substack{1\le \nu\le n_j\\ 1\le j\le p}},\\
d_{\sigma}(\mathbf m)&:=m_{1,\sigma_1}+\cdots+m_{p,\sigma_p}-(p-2)\cdot\ord\mathbf m
\quad(m_{j,\nu}=0\text{ if }\nu\le 0\text{ or }\nu>n_j). 
\end{align*}
The relation $\mathbf m'=\mc_\sigma\mathbf m$ with $\sigma=(\sigma_1,\dots,\sigma_p)$
and $D=d_\sigma(\mathbf m)$ is denoted by
\begin{equation}
 \mathbf m\xrightarrow[\sigma_1,\dots,\sigma_p]{D}\mathbf m'.
\end{equation}
\end{definition}

\begin{theorem}[\cite{katz1996rigid}]\label{thm:KatzRed}
A spectral type $\mathbf m$ is irreducibly realizable if and only if 
$\mc_\sigma \mathbf m$ is irreducibly realizable or $\ord \mathbf m=1$.
\end{theorem}

If $\ord\mc_\sigma \mathbf m<\ord\mathbf m$, namely, 
$d_\sigma(\mathbf m)>0$ (cf.~\eqref{eq:KatzRed}),
we call $\mc_\sigma \mathbf m$ a {\em Katz reduction} of $\mathbf m$. 
In particular, if $\sigma$ is chosen so that 
$d_\sigma(\mathbf m)$ is maximal among positive integers,
$\mc_\sigma \mathbf m$ is called a 
{\em maximal Katz reduction} of $\mathbf m$.

\begin{definition}[index of rigidity, \cite{katz1996rigid}] 
The {\em index of rigidity} of $\mathbf m$ is 
the even integer defined by
\[
 \idx\mathbf m:=\sum\limits_{\substack{1\le \nu\le n_j\\1\le j\le p}}m_{j,\nu}^2-(p-2)(\ord \mathbf m)^2.\]
\end{definition}
\begin{remark}
\rm{(1)} \ $m_{j,\sigma_j}\ (1\le j<p)$ corresponds to the multiplicity of the eigenvalue 0
of the resulting residue matrix under a gauge transformation, while
$m_{p,\sigma_p}$ corresponds to the multiplicity of the eigenvalue $\mu$
of the resulting residue matrix at $x=\infty$ (cf.~\cite{DR}). 

\smallskip
Usually we assume $m_{j,\nu}>0$. However, in defining $\mc_\sigma \mathbf m$, we allow 
some multiplicities $m_{j,\nu}=0$.
If a spectral type $\mathbf m$ contains zero multiplicities, we 
identify it with the spectral type defined by omitting  them.

\medskip
\noindent
\rm{(2)} \ We have the following relations.
\begin{align}
\ord\mc_\sigma\mathbf m&=\ord\mathbf m - d_{\sigma}(\mathbf m),\label{eq:KatzRed}\\
 \mc^2_\sigma&=\mathrm{id},\label{eq:mc2eq1}\\
 \idx\mc_\sigma\mathbf m&=\idx \mathbf m.
\end{align}

We show some typical examples of middle convolutions of spectral types.
\end{remark}
\begin{exmp}\label{ex:kacreduction}
 $\underline{1}11,\underline{2}1,\underline{1}11
\xrightarrow[1,1,1]{1{+}2{+}1{-}1\cdot 3=1}011,11,011\simeq
\underline{1}1,\underline{1}1,\underline{1}1\xrightarrow[1,1,1]{3-2}1,1,1\simeq \mathbf 1$

$\underline{2}11,\underline{2}11,\underline{1}111
\xrightarrow[1,1,1]{5-4=1}\underline{1}11,\underline{1}11,0\underline{1}11
\simeq \underline{1}11,\underline{1}11,\underline{1}11
\xrightarrow[1,1,1]{3-3=0}111,111,111$ 

$\underline{6}511,\underline{6}511,\underline{6}1111111\xrightarrow[1,1,1]{18-13=5}
1\underline{5}11,1\underline{5}11,\underline{1}1111111\xrightarrow[2,2,1]{11-8=3}
\times$\\
Here $\mathbf 1$ denotes the trivial partition $1=1$, which correspond to the trivial equation 
$u'=0$.
\end{exmp}

In the last example, $\mc_{(2,2,1)}$ is not defined as a spectral type, 
because it contains a negative multiplicity.  
Thus the spectral type  $6511,6511,61111111$ is not irreducibly realizable.

%%%%%%%%%%%%%%%%%%%%%%%%%%%%%%%%%%%%%%%%%%%%%%%%%%
%%%%%%%%%%%%%%%%%%   2.3   %%%%%%%%%%%%%%%%%%%%%%%
\subsection{A star-shaped Kac-Moody root system}
%%%%%%%%%%%%%%%%%%%%%%%%%%%%%%%%%%%%%%%%%%%%%%%%%%
We explain the correspondence between spectral types and a Kac--Moody root system
introduced by Crawley-Boevey \cite{CB}.
\begin{definition}
We introduce the following Kac--Moody root system $(\Pi,W)$.
\begin{equation*}%
 \begin{split}
 (\alpha|\alpha)
 &= 2\quad(\alpha\in\Pi:=\{\alpha_0,\,\alpha_{j,\nu}\mid j\ge1,\,\nu\ge1\})\\
 (\alpha_0|\alpha_{j,\nu})
 &=-\delta_{\nu,1}\\
 (\alpha_{i,\mu}|\alpha_{j,\nu})
 &=\begin{cases}
    0 &(i\ne j\text{ \ \ or \ \ }|\mu-\nu|>1)\\
    -1&(i=j\text{ \ and \ }|\mu-\nu|=1)
  \end{cases}\\
 s_\alpha\,&:\ x\mapsto
 x-(x|\alpha)\alpha\quad(\alpha\in\Pi),\ \ \  W:=\langle s_\alpha|\alpha\in\Pi\rangle
 \end{split}%\hspace{-8mm}
 \raisebox{-0.8cm}{\scalebox{0.95}
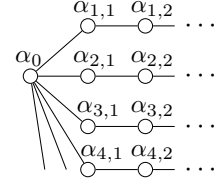
{
\begin{tikzpicture}
 [root/.style={draw,circle,inner sep=0mm,minimum size=2mm}]
\node[root] (O) at (0,0) {};
\node[root] (A) at (0.8,0) {};
\node[root] (B) at (1.6,0) {};
\node (AB) at (2.4,0) {$\cdots$};
\node[root] (C) at (0.8,-0.7) {};
\node[root] (E) at (1.6,-0.7) {};
\node (CE) at (2.4,-0.7) {$\cdots$};
\node[root] (D) at (0.8,0.7) {};
\node[root] (F) at (1.6,0.7) {};
\node (DF) at (2.4,0.7) {$\cdots$};
\node[root] (G) at (0.8,-1.3) {};
\node[root] (H) at (1.6,-1.3) {};
\node (GH) at (2.4,-1.3) {$\cdots$};
\node at (-0.02,0.24) {\small$\alpha_0$};
\node at (0.9,0.22) {\small$\alpha_{2,1}$};
\node at (1.7,0.22) {\small$\alpha_{2,2}$};
\node at (0.9,0.92) {\small$\alpha_{1,1}$};
\node at (1.7,0.92) {\small$\alpha_{1,2}$};
\node at (0.95,-0.48) {\small$\alpha_{3,1}$};
\node at (1.7,-0.48) {\small$\alpha_{3,2}$};
\node at (1,-1.08) {\small$\alpha_{4,1}$};
\node at (1.7,-1.08) {\small$\alpha_{4,2}$};
\draw (O)--(A)--(B)--(AB) (O)--(C)--(E)--(CE) (O)--(D)--(F)--(DF) (O)--(G)--(H)--(GH)
(O)--(0.5,-1.3) (O)--(0.2,-1.3)
;
\end{tikzpicture}}\hspace{-5mm}}
\end{equation*}
\end{definition}
%%%%%%%%%%%%%%%%%%%%%%%%%%%%
\begin{definition}\label{def:starKac}
We associate each spectral type $\mathbf m$ with an element 
$\alpha_{\mathbf m}$ of this root lattice:
\begin{align*}
  \alpha_{\mathbf m}
  := n\alpha_0+\sum_{j\ge1}&\sum_{i\ge1}\Bigl(\sum_{\nu>i}m_{j,\nu}\Bigr)
  \alpha_{j,i}\qquad
 (n=\ord\mathbf m) 
 \\
  =n\alpha_0+\sum_{j\ge1}&\sum_{\nu\ge1}n_{j,\nu}\alpha_{j,\nu},\\
 &\hspace{-2.8cm} m_{j,\nu}\,\,{=n_{j,\nu-1}-n_{j,\nu}}
 \quad(j=1,2,\ldots,\ \nu=1,2,\ldots,\ \  n_{j,0}=n).
\end{align*}
\end{definition}

\begin{exmp} Correspondence $\mathbf m$ and $\alpha_{\mathbf m}$ : 
Here, the number inside a circle is the coefficient $n_{j,\nu}$ of the 
simple root $\alpha_{j,\nu}$ in the Dynkin diagram, while $m_{j,\nu}$ 
is the difference between two consecutive coefficients. 
Simple roots with zero coefficients are usually omitted from the diagram.
\begin{align*}
 \texttt{11,11,11}
  \leftrightarrow\alpha&=2\alpha_0+\alpha_{1,1}+\alpha_{2,1}+\alpha_{3,1}\\
 \texttt{21,21,21,21}\leftrightarrow\alpha&
  =3\alpha_0+\alpha_{1,1}+\alpha_{2,1}+\alpha_{3,1}+\alpha_{4,1}\hspace{6cm}
\end{align*}

\vspace{-1.8cm}\hspace{10.8cm}
\begin{tikzpicture}
 [root/.style={draw,circle,inner sep=0mm,minimum size=2.7mm}]
\node[root] (a0) at (0,0) {}; 
\node[root] (a1) at (-0.5,0) {}; 
\node[root] (a2) at (0.5,0) {}; 
\node[root] (a3) at (0,0.5) {}; 
\draw (a1)--(a0)--(a2) (a3)--(a0);
\node at (0,0.5) {\scriptsize 1};
\node at (-0.5,0) {\scriptsize 1};
\node at (0.5,0) {\scriptsize 1};
\node at (0,0) {\scriptsize 2};
\node at (0.1,0.25) {\tiny  1};
\node at (0.1,0.75) {\tiny 1};
\node at (0.25,0.15) {\tiny 1};
\node at (0.72,0.15) {\tiny 1};
\node at (-0.25,0.15) {\tiny 1};
\node at (-0.75,0.15) {\tiny 1};
\end{tikzpicture}
\quad
\raisebox{-3mm}{
\begin{tikzpicture}
 [root/.style={draw,circle,inner sep=0mm,minimum size=2.7mm}]
\node[root] (a0) at (0,0) {}; 
\node[root] (a1) at (-0.5,0) {}; 
\node[root] (a2) at (0.5,0) {}; 
\node[root] (a3) at (0,0.5) {}; 
\node[root] (a4) at (0,-0.5) {}; 
\draw (a1)--(a0)--(a2) (a3)--(a0)--(a4);
\node at (0,0.5) {\scriptsize 1};
\node at (-0.5,0) {\scriptsize 1};
\node at (0.5,0) {\scriptsize 1};
\node at (0,0) {\scriptsize 3};
\node at (0,-0.5) {\scriptsize 1};
\node at (0.1,0.25) {\tiny 2};
\node at (0.1,0.75) {\tiny 1};
\node at (0.25,0.15) {\tiny 2};
\node at (0.72,0.15) {\tiny 1};
\node at (-0.25,0.15) {\tiny 2};
\node at (-0.75,0.15) {\tiny 1};
\node at (-0.15,-0.25) {\tiny 2};
\node at (-0.1,-0.75) {\tiny 1};
\end{tikzpicture}}
\vspace{-5mm}

\raisebox{1mm}{111,21,111\,:\!}
\begin{tikzpicture}
 [root/.style={draw,circle,inner sep=0mm,minimum size=2.7mm}]
\node[root] (a3) at (0,0) {}; 
\node[root] (a2) at (-0.5,0) {}; 
\node[root] (a1) at (-1,0) {}; 
\node[root] (a4) at (0.5,0) {}; 
\node[root] (a5) at (1.0,0) {}; 
\node[root] (a7) at (0,0.5) {}; 
\draw (a1)--(a2)--(a3)--(a4)--(a5) (a3)--(a7);
\node at (0,0) {\scriptsize 3};
\node at (-0.5,0) {\scriptsize 2}; 
\node at (-1,0) {\scriptsize 1}; 
\node at (0,0.5) {\scriptsize 1}; 
\node at (0.5,0) {\scriptsize 2}; 
\node at (1.0,0) {\scriptsize 1}; 
\node at (0.1,0.25) {\tiny 2};
\node at (0.1,0.75) {\tiny  1};
\node at (0.25,0.15) {\tiny  1};
\node at (0.75,0.15) {\tiny 1};
\node at (1.22,0.15) {\tiny  1};
\node at (-0.25,0.15) {\tiny 1};
\node at (-0.75,0.15) {\tiny 1};
\node at (-1.22,0.15) {\tiny 1};
\end{tikzpicture}
\qquad
\raisebox{1mm}{33,222,111111\,:\!}
\begin{tikzpicture}
 [root/.style={draw,circle,inner sep=0mm,minimum size=2.7mm}]
\node[root] (a3) at (0,0) {}; 
\node[root] (a2) at (-0.5,0) {}; 
\node[root] (a1) at (-1,0) {}; 
\node[root] (a4) at (0.5,0) {}; 
\node[root] (a5) at (1.0,0) {}; 
\node[root] (a6) at (1.5,0) {}; 
\node[root] (a7) at (2.0,0) {}; 
\node[root] (a8) at (2.5,0) {}; 
\node[root] (a0) at (0,0.5) {}; 
\draw (a1)--(a2)--(a3)--(a4)--(a5)--(a6)--(a7)--(a8) (a3)--(a0);
\node at (0,0) {\scriptsize 6}; 
\node at (-0.5,0) {\scriptsize 4}; 
\node at (-1,0) {\scriptsize 2}; 
\node at (0.5,0) {\scriptsize 5}; 
\node at (1.0,0) {\scriptsize 4}; 
\node at (1.5,0) {\scriptsize 3}; 
\node at (2.0,0) {\scriptsize 2}; 
\node at (2.5,0) {\scriptsize 1}; 
\node at (0,0.5) {\scriptsize 3}; 
\node at (0.1,0.25) {\tiny 3};
\node at (0.1,0.75) {\tiny 3};
\node at (0.25,0.15) {\tiny 1};
\node at (0.75,0.15) {\tiny 1};
\node at (1.25,0.15) {\tiny 1};
\node at (1.75,0.15) {\tiny 1};
\node at (2.25,0.15) {\tiny 1};
\node at (2.72,0.15) {\tiny 1};
\node at (-0.25,0.15) {\tiny 2};
\node at (-0.75,0.15) {\tiny 2};
\node at (-1.22,0.15) {\tiny 2};
\end{tikzpicture}
\end{exmp}

\begin{proposition}
Under Definition~\ref{def:starKac}, we have
\begin{align*} 
\hspace{-12mm}\alpha_{\mc_{(1,\dots,1)}\mathbf m}&=s_{\alpha_0}\alpha_{\mathbf m},&
\alpha_{\mathbf m'}&=s_{\alpha_{j,\nu}}\alpha_{\mathbf m}\quad
(\mathbf m'=\cdots %m_{j,\nu-1}
m_{j,\nu+1}m_{j,\nu}\cdots).
\hspace{-2.3cm}\raisebox{2.4mm}{\begin{tikzpicture}
\draw[<->] (0,0) .. controls (0.1,0.15) and (0.7,0.15) .. (0.8,0); 
\end{tikzpicture}}\qquad\qquad\quad
\\
\idx\mathbf m&=(\alpha_{\mathbf m}|\alpha_{\mathbf m}),&
d_{1,\dots,1}\mathbf m&=(\alpha_0|\alpha_{\mathbf m}).
\end{align*}
In particular, the transformations of spectral types induced by 
middle convolutions $\mc_\sigma$ correspond to the transformations 
of the root lattice by the Weyl group. 
\end{proposition}

%%%%%%%%%%%%%%%%%%%%%%%%%%%%%%%%%%%%%%%%%%%%
%%%%%%%%%%%  2.4  %%%%%%%%%%%%%%%%%%%%%%%%%%
\subsection{Deligne-Simpson problem}
%%%%%%%%%%%%%%%%%%%%%%%%%%%%%%%%%%%%%%%%%%%%
The problem of determining the conditions for the existence of an irreducible Fuchsian system
with a given generalized Riemann scheme is called the additive {\em Deligne-Simpson problem}.  
It is solved as follows.

\begin{theorem}[\cite{Kostov, katz1996rigid, CB, DR, Ow}]\label{thm:DS}
The spectral type $\mathbf m$ is irreducibly realizable if and only if
$\alpha_{\mathbf m}$ is a root of the star-shaped Kac--Moody root system
(more precisely, when $\idx\mathbf m=1$, $\mathrm{GCD}\{m_{j,\nu}\}=1$ should be satisfied). 
Moreover, if the spectral type of $(\bar A_1,\dots\bar A_p)\in M(n,\mathbb C)^p$ 
is irreducible realizable, there exists an irreducible tuple $(A_1,\dots,A_p)\in M(n,\mathbb C)^p$
such that
\begin{align*}
 A_1+\cdots+A_p\in\mathbb C\text{ \ and \ }A_j\sim \bar A_j. 
\end{align*}
Here $(A_1,\dots,A_p)$ is called irreducible if the common 
invariant subspace of $A_1,\dots,A_p$ is trivial.
\end{theorem}

\begin{remark}
\rm{(1)} \ Kostov \cite{Kostov} studied the problem in some cases using implicit 
function theorem.
Katz \cite{katz1996rigid} introduced middle convolutions and solved the problem
for rigid spectral types. 
Dettweiler and Reiter \cite{DR} interpreted the results of \cite{katz1996rigid} 
in a language of linear algebra. 
Crawley-Boevey \cite{CB} solved the problem by establishing a relation with
a Kac--Moody root system.
The corresponding problem for single differential equations of arbitrary order
was solved by \cite{Ow}.

\smallskip
\noindent
\rm{(2)} \ It follows from \eqref{eq:sum0} that the residue matrices $A_j$ of a Fuchsian system 
$\mathcal M$ satisfy $\trace A_1+\cdots+\trace A_p=0$, which corresponds to the Fuchs relation
for the Riemann scheme. 
If the sum of the traces of $\bar A_j$ is zero in Theorem~\ref{thm:DS}, 
$A_1,\dots,A_p$ can be realized as the residue matrices of a Fuchsian system.

\smallskip
\noindent
\rm{(3)} \ When $\mathbf m$ is rigid, a realized system $\mathcal M$ is uniquely determined by
the Riemann scheme (cf.~Theorem~\ref{thm:moduli}) but it may be reducible if the eigenvalues of 
$\bar A_j$ take special values (cf.~\cite{CB, Ow}).
\end{remark}

\begin{theorem}[\cite{katz1996rigid}]\label{thm:moduli}
Let $\mathbf m$ be the spectral type of an irreducible Fuchsian system given by 
\eqref{eq:Fsystem}.  
Then
\begin{align*}
  \dim \{(\tilde A_1,\dots,\tilde A_p)\in M(n,\mathbb C)^p\mid \tilde A_j\sim  A_j,\,
 \tilde A_1+\cdots+\tilde A_p=0\}\,/GL(n,\mathbb C) =2-\idx\mathbf m
\end{align*}
\end{theorem}

This theorem says that the system has $2-\idx\mathbf m$ parameters that are not determined 
by the Riemann scheme.  They are called {\em accessory parameters}.

\begin{definition}
A spectral type $\mathbf m$ is {\em rigid} if $\mathbf m$ is irreducibly realizable 
and $\idx\mathbf m =2$.
This is equivalent to the condition that  $\alpha_{\mathbf m}$ is a real root 
satisfying $\ord\mathbf m>0$.
\end{definition}

\begin{remark} {\rm(1)} \ The definition of $\idx \mathbf m$ and $d_\sigma(\mathbf m)$ gives
\begin{equation*}
\sum\limits_{j=1}^p\sum\limits_{\nu=1}^{n_j}m_{j,\nu}(m_{j,1}-m_{j,\nu})
 = d_{1,\cdots,1}(\mathbf m)\cdot\ord\mathbf m-\idx\mathbf m.
\end{equation*}
{\rm(2)} \ Suppose $\idx\mathbf m>0$.
We may assume that $\mathbf m$ is ordered. 
Then, $d_{1,\cdots,1}(\mathbf m)>0$, and 
if $\ord\mathbf m\ne1$, a Katz reduction of $\mathbf m$ is possible.
Hence, Theorem~\ref{thm:DS} follows from Theorem~\ref{thm:KatzRed}.

An irreducible system $\mathcal M$ with a spectral type $\mathbf m$ satisfying $\idx \mathbf m>0$
is called a rigid system.  Then $\idx\mathbf m=2$ and 
$\alpha_{\mathbf m}$ is a positive real root.
In this case, the system $\mathcal M$ is constructed by a finite sequence
of suitable additions and middle convolutions starting from 
the trivial equation $u'=0$.
This result is obtained in \cite{katz1996rigid} and corresponds to the fact that
the real positive root $\alpha_{\mathbf m}$ belongs to the orbit $W\alpha_0$.
From this point of view, equations with rigid spectral types can be analyzed;
for example, integral representations and series expressions of solutions,
conditions for irreducibility, contiguous relations, and connection problems
are studied in \cite{Ow}.

\smallskip
\noindent
{\rm(3)} \ Suppose $\idx\mathbf m\le0$, which will be the main case considered in this paper.
We may moreover assume $\mathbf m$ is ordered.
Note that if $\mathbf m':=\mc_{(1,\dots,1)}\mathbf m$ is not defined, 
$\mathbf m$ is not irreducibly realizable (cf.~the last example in 
Example~\ref{ex:kacreduction}). 
If $d_{1,\cdots,1}(\mathbf m)>0$, the condition that $\mathbf m$ is irreducibly realizable
is reduced to corresponding condition for $\mathbf m'$, which  satisfies 
$\ord\mathbf m'<\ord\mathbf m$.  
Hence it suffices to consider the case $d_{1,\cdots,1}(\mathbf m)\le0$.
Crawley-Boevey \cite{CB} proved that every $\mathbf m$ satisfying 
$d_{1,\cdots,1}(\mathbf m)\le0$ is irreducibly realizable by applying a result
on representations of quivers, thereby proving Theorem~\ref{thm:DS}.
In the case $\idx\mathbf m=0$, the only exception occurs when
$\mathrm{GCD}\{m_{j,\nu}\}>1$.

By an iteration of additions and middle convolutions, 
any irreducible system $\mathcal M$ with accessory parameters, namely, a non-rigid irreducible 
system $\mathcal M$ is transformed into a system with a {\em fundamental spectral type} whose order cannot be reduced by a Katz reduction.

A spectral type $\mathbf m=(m_{j,\nu})_{\substack{j=1,\dots,p\\\nu=1,\dots,n_j}}$ 
is called {\em fundamental} if and only if
\begin{equation}
 d_\sigma(\mathbf m)\le 0\ \ (\forall\sigma\in\mathbb Z_{\ge0}^p)\label{eq:fund0}
\end{equation}
and moreover $\mathrm{GCD}\{m_{j,\nu}\}=1$ when $\idx\mathbf m=0$.
If $\mathbf m$ is ordered, the condition \eqref{eq:fund0} is equivalent to
\begin{align*}
  m_{1,1}+\cdots+m_{p,1}\le (p-2)\ord\mathbf m. 
\end{align*}
This condition is equivalent to
\begin{equation*}%\label{eq:fdomain}
  (\alpha_{\mathbf m}|\alpha)\le 0\quad(\forall\alpha\in\Pi)
\end{equation*}
and also to the condition that
$\alpha_{\mathbf m}$ is a positive imaginary root in a fundamental domain of the 
Weyl group (cf.~\cite[Theorem~5.4]{KacBook}).
\end{remark}

\begin{remark}\label{remark:id}
We will study the classification of irreducibly realizable spectral types. 
Two spectral types will be identified if one can be obtained from the other by
permuting the tuple of partitions, adding or deleting trivial partitions,
or adding or deleting terms equal to $0$.
For example,
\[
 21,111,111 \simeq 12,111,111 \simeq 111,21,111 \simeq 021,111,111 \simeq 3,21,111,111. 
\]
A spectral type $\mathbf m=(m_{j,\nu})_{\substack{j=1,\dots,p\\\nu=1,\dots,n_j}}$ 
is uniquely expressed in a {\em normal form}\\
\begin{align*}%\label{sp:normalform}
 \begin{cases}
  m_{j,1}\ge m_{j,2}\ge\cdots\ge m_{j,n_j}>0,\ \ n_j>1&(1\le j\le p),\\
  m_{j,\nu}=m_{j+1,\nu}\ (1\le \nu <k_j)\text{, and }m_{j,k_j}>m_{j+1,k_j} \text{ or }k_j=n_j{+}1&(\exists k_j,\,1\le j<p).
\end{cases}
\end{align*}
\end{remark}
\begin{remark}\label{rem:d0}
If $d_\sigma(\mathbf m)=0$, then $\mc_\sigma(\mathbf m)=\mathbf m$,
although the middle convolution $\mc_\sigma$ induces a nontrivial
transformation of Fuchsian systems of spectral type $\mathbf m$.
When $\mathbf m$ is fundamental and ordered and $d_{1,\dots,1}(\mathbf m)=0$, 
the Fuchsian systems of spectral type
$\mathbf m$ admit an automorphism group.
This group is isomorphic to the Weyl group of the root system whose
simple roots are
\[
 \{\alpha_0\}\cup\{\alpha_{j,\nu}\mid m_{j,\nu+1}=m_{j,1},\ j=1,\dots,p\}.
\]
When $\idx\mathbf m<0$, this root system is of classical type (cf.~\cite{Ow}).
For example, for
$\mathbf m=6^2,4^3,2^511$, denoted by $E_8^{(2)}$ in \cite{Ow},
the root system is of type $E_8$:

\quad
\raisebox{1mm}{$E_8^{(2)}$ :}
\begin{tikzpicture}
 [root/.style={draw,circle,inner sep=0mm,minimum size=2.7mm}]
\node[root] (a3) at (0,0) {}; 
\node[root] (a2) at (-0.5,0) {}; 
\node[root] (a1) at (-1,0) {}; 
\node[root] (a4) at (0.5,0) {}; 
\node[root] (a5) at (1.0,0) {}; 
\node[root] (a6) at (1.5,0) {}; 
\node[root] (a7) at (2.0,0) {}; 
\node[root] (a8) at (2.5,0) {}; 
\node[root] (a9) at (3,0) {}; 
\node[root] (a0) at (0,0.5) {}; 
\draw (a1)--(a2)--(a3)--(a4)--(a5)--(a6)--(a7)--(a8)--(a9) (a3)--(a0);
\node at (2.5,0) {$\cdot$}; 
\node at (0.1,0.25) {\tiny 6};
\node at (0.1,0.75) {\tiny 6};
\node at (0.25,0.15) {\tiny 2};
\node at (0.75,0.15) {\tiny 2};
\node at (1.25,0.15) {\tiny 2};
\node at (1.75,0.15) {\tiny 2};
\node at (2.25,0.15) {\tiny 2};
\node at (2.72,0.15) {\tiny 1};
\node at (3.22,0.15) {\tiny 1};
\node at (-0.25,0.15) {\tiny 4};
\node at (-0.75,0.15) {\tiny 4};
\node at (-1.22,0.15) {\tiny 4};
\end{tikzpicture}\ \raisebox{1mm}{(cf.~\cite[\S13.1.5]{Ow})\ \ $\Rightarrow$ \ \ $E_8$ :}
\begin{tikzpicture}
 [root/.style={draw,circle,inner sep=0mm,minimum size=2.7mm}]
\node[root] (a3) at (0,0) {}; 
\node[root] (a2) at (-0.5,0) {}; 
\node[root] (a1) at (-1,0) {}; 
\node[root] (a4) at (0.5,0) {}; 
\node[root] (a5) at (1.0,0) {}; 
\node[root] (a6) at (1.5,0) {}; 
\node[root] (a7) at (2.0,0) {}; 
\node[root] (a0) at (0,0.5) {}; 
\draw (a1)--(a2)--(a3)--(a4)--(a5)--(a6)--(a7) (a3)--(a0);
\node at (0.1,0.25) {\tiny 6};
\node at (0.1,0.75) {\tiny 6};
\node at (0.25,0.15) {\tiny 2};
\node at (0.75,0.15) {\tiny 2};
\node at (1.25,0.15) {\tiny 2};
\node at (1.75,0.15) {\tiny 2};
\node at (2.25,0.15) {\tiny 2};
\node at (-0.25,0.15) {\tiny 4};
\node at (-0.75,0.15) {\tiny 4};
\node at (-1.22,0.15) {\tiny 4};
\end{tikzpicture}
\end{remark}
\begin{definition}
For a  spectral type $\mathbf m=\bigl(m_{j,\nu}\bigr)_{\substack{j=1,\dots,p\\\nu=1,\dots,n_j}}$, 
set
\[
 \mathrm{N_p}(\mathbf m):=|\{j\in\{1,\dots,p\}\mid m_{j,\nu}\ne n\text{ for }\nu=1,\dots,n_j\}|. 
\]

We say that $\mathbf m$ is of $(n_1,\ldots,n_p)$-type if
$m_{j,\nu}\ne0,n$ and $n_j>1$ in the expression \eqref{eq:partn}.

We say that $\mathbf m$ is of $(n_1,n_2,*)$-type if
$\mathbf m$ is of $(n_1,n_2,n_3)$-type for some integer $n_3$.

We say that $\mathbf m$ is $E_8$-fundamental
$E_7$-fundamental, $E_6$-fundamental or $D_4$-fundamental if
$\mathbf m$ is fundamental and 
$(2,3,*)$-type, $(2,4,*)$-type, $(3,3,*)$-type or 
$(2,2,2,*)$-type, respectively.
\end{definition}

\medskip
Dynkin diagrams of roots corresponding to some spectral types are as follows:

\smallskip
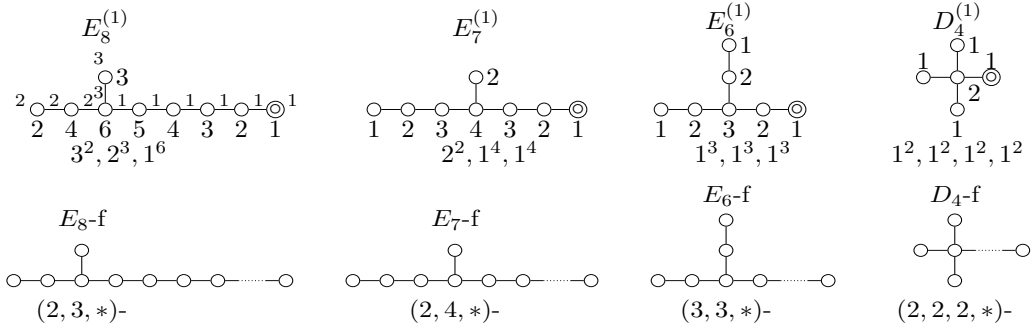
\scalebox{0.9}[0.85]{\begin{tikzpicture}
 [root/.style={draw,circle,inner sep=0mm,minimum size=2mm}]
\node[root] (a1) at (0,0) {}; 
\node[root] (a2) at (0.5,0) {}; 
\node[root] (a3) at (1,0) {}; 
\node[root] (a4) at (1.5,0) {}; 
\node[root] (a5) at (2,0) {}; 
\node[root] (a6) at (2.5,0) {}; 
\node[root] (a7) at (3,0) {}; 
\node[root] (a8) at (3.5,0) {$\circ$}; 
\node[root] (a0) at (1,0.5) {}; 
\node at (0,-0.3) {2};
\node at (0.5,-0.3) {4};
\node at (1,-0.3) {6};
\node at (1.5,-0.3) {5};
\node at (2,-0.3) {4};
\node at (2.5,-0.3) {3};
\node at (3,-0.3) {2};
\node at (3.5,-0.3) {1};
\node at (1.25,0.5) {3};
\node at (1,1.3) {$E_8^{(1)}$}; 
\draw (a1)--(a2)--(a3)--(a4)--(a5)--(a6)--(a7)--(a8) (a0)--(a3);
\node at (1.2,-0.7) {$3^2,2^3,1^6$};
\node at (0.9,0.27) {\scriptsize 3};
\node at (0.9,0.8) {\scriptsize 3};
\node at (0.75,0.15) {\scriptsize 2};
\node at (0.25,0.15) {\scriptsize 2};
\node at (-0.25,0.15) {\scriptsize 2};
\node at (1.25,0.15) {\scriptsize 1};
\node at (1.75,0.15) {\scriptsize 1};
\node at (2.25,0.15) {\scriptsize 1};
\node at (2.75,0.15) {\scriptsize 1};
\node at (3.25,0.15) {\scriptsize 1};
\node at (3.75,0.15) {\scriptsize 1};
\end{tikzpicture}\qquad
\begin{tikzpicture}
 [root/.style={draw,circle,inner sep=0mm,minimum size=2mm}]
\node[root] (a) at (-0.5,0) {}; 
\node[root] (a1) at (0,0) {}; 
\node[root] (a2) at (0.5,0) {}; 
\node[root] (a3) at (1,0) {}; 
\node[root] (a4) at (1.5,0) {}; 
\node[root] (a5) at (2,0) {}; 
\node[root] (a6) at (2.5,0) {$\circ$}; 
\node[root] (a0) at (1,0.5) {}; 
\node at (1,1.3) {$E_7^{(1)}$}; 
\node at (-0.5,-0.3) {1};
\node at (0,-0.3) {2};
\node at (0.5,-0.3) {3};
\node at (1,-0.3) {4};
\node at (1.5,-0.3) {3};
\node at (2,-0.3) {2};
\node at (2.5,-0.3) {1};
\node at (1.25,0.5) {2};
\draw (a)--(a1)--(a2)--(a3)--(a4)--(a5)--(a6) (a0)--(a3);
\node at (1.2,-0.7) {$2^2,1^4,1^4$};
\end{tikzpicture}\qquad
\begin{tikzpicture}
 [root/.style={draw,circle,inner sep=0mm,minimum size=2mm}]
\node[root] (aa) at (1,1) {}; 
\node[root] (a1) at (0,0) {}; 
\node[root] (a2) at (0.5,0) {}; 
\node[root] (a3) at (1,0) {}; 
\node[root] (a4) at (1.5,0) {}; 
\node[root] (a5) at (2,0) {$\circ$}; 
%\node[root] (a6) at (2.5,0) {}; 
\node[root] (a0) at (1,0.5) {}; 
\node at (0,-0.3) {1};
\node at (0.5,-0.3) {2};
\node at (1,-0.3) {3};
\node at (1.5,-0.3) {2};
\node at (2,-0.3) {1};
\node at (1,1.4) {$ E_6^{(1)}$}; 
\node at (1.25,0.5) {2};
\node at (1.25,1) {1};
\draw (a1)--(a2)--(a3)--(a4)--(a5) (aa)--(a0)--(a3);
\node at (1.2,-0.7) {$1^3,1^3,1^3$};
\end{tikzpicture}\qquad\ \ 
\begin{tikzpicture}
 [root/.style={draw,circle,inner sep=0mm,minimum size=2mm}]
\node[root] (a0) at (0,0) {}; 
\node[root] (a1) at (-0.5,0) {}; 
\node[root] (a2) at (0.5,0) {$\circ$}; 
\node[root] (a3) at (0,0.5) {}; 
\node[root] (a4) at (0,-0.5) {}; 
\draw (a1)--(a0)--(a2) (a3)--(a0)--(a4);
\node at (0.25,0.5) {1};
\node at (-0.5,0.25) {1};
\node at (0.5,0.25) {1};
\node at (0,-0.8) {1};
\node at (0.25,-0.25) {2};
\node at (0,0.9) {$D_4^{(1)}$};
\node at (0,-1.2) {$1^2,1^2,1^2,1^2$};
\end{tikzpicture}
}

\scalebox{0.9}[0.8]{\begin{tikzpicture}
 [root/.style={draw,circle,inner sep=0mm,minimum size=2mm}]
\node[root] (a1) at (0,0) {}; 
\node[root] (a2) at (0.5,0) {}; 
\node[root] (a3) at (1,0) {}; 
\node[root] (a4) at (1.5,0) {}; 
\node[root] (a5) at (2,0) {}; 
\node[root] (a6) at (2.5,0) {}; 
\node[root] (a7) at (3,0) {}; 
\node[root] (a8) at (4,0) {}; 
\node[root] (a0) at (1,0.5) {}; 
\node at (1,1) {$E_8$-f}; 
\node at (1,-0.5) {$(2,3,*)$-};
\draw (a1)--(a2)--(a3)--(a4)--(a5)--(a6)--(a7)--(3.3,0) (3.7,0)--(a8) (a0)--(a3);
\draw[densely dotted] (3.3,0)--(3.7,0);
\end{tikzpicture}\qquad
\begin{tikzpicture}
 [root/.style={draw,circle,inner sep=0mm,minimum size=2mm}]
\node[root] (a) at (-0.5,0) {}; 
\node[root] (a1) at (0,0) {}; 
\node[root] (a2) at (0.5,0) {}; 
\node[root] (a3) at (1,0) {}; 
\node[root] (a4) at (1.5,0) {}; 
\node[root] (a5) at (2,0) {}; 
\node[root] (a6) at (3,0) {}; 
\node[root] (a0) at (1,0.5) {}; 
\node at (1,1) {$E_7$-f}; 
\node at (1,-0.5) {$(2,4,*)$-};
\draw (a)--(a1)--(a2)--(a3)--(a4)--(a5)--(2.3,0) (2.7,0)--(a6) (a0)--(a3);
\draw[densely dotted] (2.3,0)--(2.7,0);
\end{tikzpicture}\qquad
\begin{tikzpicture}
 [root/.style={draw,circle,inner sep=0mm,minimum size=2mm}]
\node[root] (aa) at (1,1) {}; 
\node[root] (a1) at (0,0) {}; 
\node[root] (a2) at (0.5,0) {}; 
\node[root] (a3) at (1,0) {}; 
\node[root] (a4) at (1.5,0) {}; 
\node[root] (a5) at (2.5,0) {}; 
%\node[root] (a6) at (2.7,0) {}; 
\node[root] (a0) at (1,0.5) {}; 
\node at (1,1.4) {$E_6$-f}; 
\node at (1,-0.5) {$(3,3,*)$-};
\draw (a1)--(a2)--(a3)--(a4)--(1.8,0)  (2.2,0)--(a5) (aa)--(a0)--(a3);
\draw[densely dotted] (1.8,0)--(2.2,0);
\end{tikzpicture}\qquad
\raisebox{0mm}{\begin{tikzpicture}
 [root/.style={draw,circle,inner sep=0mm,minimum size=2mm}]
\node[root] (a0) at (0,0) {}; 
\node[root] (a1) at (-0.5,0) {}; 
\node[root] (a2) at (1.0,0) {}; 
\node[root] (a3) at (0,0.5) {}; 
\node[root] (a4) at (0,-0.5) {}; 
\draw (a1)--(a0)--(0.3,0)  (0.7,0)--(a2) (a3)--(a0)--(a4);
\draw[densely dotted] (0.3,0)--(0.7,0);
\node at (0,0.9) {$D_4$-f};
\node at (0,-1) {$(2,2,2,*)$-};
\end{tikzpicture}}}

\begin{definition}\label{def:mcsim}
For the classification of irreducibly realizable spectral types, we use the
following notation under the identification given in Remark~\ref{remark:id}:
\begin{align*}
 \mathcal P& : \ \text{the set of spectral types of irreducible Fuchsian systems,}\\
 \mathcal F& : \ \text{the set of fundamental spectral types,}\\
 \mathcal F_\ell&:=\{\mathbf m\in\mathcal F\mid \idx\mathbf m=-\ell\},\\
 \mathcal F^T&:=\{\mathbf m\in\mathcal F\mid \mathbf m:\,\text{$T$-fundamental}\}
 \quad(T=E_8,\, E_7,\, E_6\ \text{or}\ D_4).
\end{align*}

For spectral types in $\mathcal P$, we define an equivalence relation \ $\underset{\mc}\sim$ \ 
by declaring that $\mathbf m\underset{\mc}\sim\mathbf m'$ 
if $\mathbf m'$ is obtained from $\mathbf m$
by a finite sequence of middle convolutions $\mc_\sigma$.
\end{definition}
Note that
\begin{equation}\mathcal P/\!\underset{\mc}\sim\ \simeq\ \mathcal F\cup\{\mathbf 1\},\quad
\mathcal F=
 \bigsqcup_{k=0}^\infty\mathcal F_{2k}.\label{eq:mcsim}
\end{equation}

%%%%%%%%%%%%%%%%%%%%%%%%%%%%%%%%%%%%%%%%%%%%%%%%%%%%
%%%%%%%%%%%%%%%%%%%%%%  T-fundamental %%%%%%%%%%%%%%
%%%%%%%%%%%%%%%%%%%%%%%%%%%%%%%%%%%%%%%%%%%%%%%%%%%%
The spectral types in $\mathcal F_\ell$ for small $|\ell|$
are classified and listed in
\cite{Ow, Ospect10}.  
We have the following result.
\begin{theorem}[\cite{Kostov}]
$\mathcal F_0=\{33,222,111111\ \ 22,1111,1111\ \ 111,111,111\ \ 11,11,11,11\}$. 
\end{theorem}
\begin{theorem}[\cite{Ow}]\label{thm:Fuchsfinite}
$|\mathcal F_\ell|<\infty\quad(\ell=0,2,4,6,\ldots)$. 
\end{theorem}

We describe the elements of $\mathcal F^{E_8}$ that play an important role in this paper.
\begin{proposition}
Let \[\mathbf m=m_{11}m_{12},\,m_{21}m_{22}m_{23},\,m_{31}m_{32}\cdots m_{3,n_3}\] be an ordered 
$(2,3,*)$-spectral type.  Then, 
\begin{align}\label{eq:E8fund}
 m_{12}\ge m_{21}+m_{31}
\end{align}
is a necessary and sufficient condition so that $\mathbf m$ is fundamental.

If this condition is satisfied, then $n_3\ge 6$ and 
\begin{equation*}
 m_{11}\ge m_{12}>m_{21}\ge m_{22}\ge m_{23}>m_{31}\ge m_{32}\ge \cdots\ge m_{3,n_3}. 
\end{equation*}

Conversely, any $\mathbf m\in\mathcal F^{E_8}$ with $\ord\mathbf m=n$ is obtained as follows.
First, choose a pair of integers $(m_{21},m_{12})$ satisfying
\begin{equation}
  \tfrac n3\le m_{21}<m_{12}\le\tfrac n2
\label{eq:E8cond}
\end{equation}
and a positive integer $m_{31}$ satisfying  
\eqref{eq:E8fund}. 
Next, choose $m_{22}$ satisfying $\frac {n-m_{21}}2\le m_{22}\le m_{21}$.
Then set $m_{11}=n-m_{12}$ and $m_{23}=n-m_{21}-m_{22}$.
Finally, choose a  non-increasing sequence $m_{31},m_{32},\ldots$ whose sum is $n$.

Note that $\mathbf m\in\mathcal F^{E_8}$ exists if $n=6$ or $n>7$.
Moreover when $n=6$, $\mathbf m=E_8^{(1)}=3^2,2^3,1^6\in\mathcal F_0$.
\end{proposition}
\begin{proof}
The ordered $(2,3,*)$-spectral type $\mathbf m$ satisfies
\begin{align}\begin{split}
  n&=m_{11}+m_{12}=m_{21}+m_{22}+m_{23}=m_{31}+\cdots+m_{3,n_3},\\
  m_{11}&\ge m_{12},\ m_{21}\ge m_{22}\ge m_{23},\ m_{31}\ge m_{32}\ge\cdots\ge m_{3,n_3}.
\end{split}\label{eq:E8ordered}
\end{align}
The condition that $\mathbf m$ is fundamental is equivalent to
\begin{align}\label{eq:E8fundamental}
 d_{1,1,1}(\mathbf m)=m_{11}+m_{21}+m_{31}-n=m_{21}+m_{31}-m_{12}\le 0.
\end{align}
Thus, the first assertion follows.

It follows from \eqref{eq:E8ordered} that $m_{12}\le\tfrac n2$ and $m_{21}\ge\frac n3$
and \eqref{eq:E8fundamental} proves $m_{21}\le m_{12}-m_{31}<m_{12}$.  
Hence \eqref{eq:E8cond} follows.

There exist integers $m_{21}$ and $m_{12}$ satisfying \eqref{eq:E8cond} when $n=6$ or $n>7$.
If $n=6$, $(m_{21},m_{12})=(2,3)$.  
Finally, note that
\begin{align*}
  m_{31}&\le n-m_{11}-m_{21}\le n-\tfrac n2-\tfrac n3\le \tfrac n6,\\
  m_{23}-m_{31}&\ge (n-m_{21}-m_{22})-(n-m_{11}-m_{21})=m_{11}-m_{22}\ge m_{12}-m_{21}>0.
\end{align*}
The proposition now follows.
\end{proof}
%%%%%%%%%%%%%%%%%%%%%%%%%%%%%%%%%%%%%%%%%%%%
%%%%%%%%%%%%%%%%  2.5  %%%%%%%%%%%%%%%%%%%%%
\subsection{Confluences and unfoldings}
%%%%%%%%%%%%%%%%%%%%%%%%%%%%%%%%%%%%%%%%%%%%
The Riemann scheme of a differential equation on $\mathbb P^1$
with regular singularities at finite points and
an irregular singularity of Poincar\'e rank 1 at infinity is 
\begin{align}
\begin{Bmatrix} x=\infty & x=c_j\ (j=2,\dots,p{-}1)\\
    [\lambda^{(1)}_1x+\lambda^{(0)}_{1,1}]_{m^{(0)}_{1,1}} & [\lambda_{j,1}]_{m_{j,1}}\\[-1.5mm]
    \vdots & \vdots\\[-.5mm]
    [\lambda^{(1)}_{n_1}x+\lambda^{(0)}_{n_1,R_{n_1}}]_{m^{(0)}_{n_0,R_{n_0}}} & [\lambda_{j,n_j}]_{m_{j,n_j}}
\end{Bmatrix}\label{eq:RSirr}
\end{align}
This implies the existence of local solutions with asymptotic expansions
\[ u\sim C_{\nu,k,\ell}(\tfrac 1x)^{\lambda^{(0)}_{\nu,k}}e^{-\lambda^{(1)}_\nu x}
\quad(x\to\infty,\ 1\le \ell\le m^{(0)}_{\nu,k},\ 1\le k\le R_\nu).
\]
In this case, the spectral type at $\infty$ is defined by
\[\bigl\{[\lambda^{(1)}_\nu x+\lambda^{(0)}_{\nu,k}]_{m^{(0)}_{\nu,k}}\mid
   1\le k\le R_\nu,\ 1\le \nu \le n_1\}.
\]
We consider the multiplicity 
\[m^{(1)}_\nu:=%\sum\limits_{k=1}^{R_\nu} 
m^{(0)}_{\nu,1}+\cdots+m^{(0)}_{\nu,R_\nu}\]
of the exponential term $e^{-\lambda^{(1)}_\nu x}$.
Then the partition 
$\{m^{(0)}_{\nu,k}\}$ of $n$ is a refinement of the partition 
$n=m^{(1)}_1+\cdots+m^{(1)}_{n_1}$.
In general, characteristic exponents at an unramified singular point with Poincar\'e rank $r$
are polynomials of degree $r$ and the spectral type at the point
is a refinement sequence of $r{+}1$ tuples of partitions of $n$ (cf.~\cite{Ov}).

By inserting parentheses or separating the entries by $|$,
this spectral type can be expressed as follows
(cf.~\cite{HKNS, Ov}).
\begin{align*}
\hat{\mathbf m}&=
(m^{(0)}_{1,1}\cdots m^{(0)}_{1,R_1})\cdots(m^{(0)}_{n_1,1}\cdots m^{(0)}_{n_1,R_{n_1}}),m_{2,1}\cdots,\cdots m_{p{-}1,n_{p{-}1}}\\
 &=m^{(0)}_{1,1}m^{(0)}_{1,2}\cdots m^{(0)}_{n_1,R_{n_1}}|\,m^{(1)}_1
\cdots m^{(1)}_{n_1},m_{2,1}\cdots,\cdots,m_{p{-}1,1}\cdots m_{p{-}1,n_{p{-}1}}
\end{align*}

\noindent
Unfolding of $\hat{\mathbf m}$ : 
$m^{(0)}_{1,1}m^{(0)}_{1,2}\cdots m^{(0)}_{n_1,R_{n_1}},\,m_{1,1}\cdots m_{1,n_1},m_{2,1}\cdots,
\cdots,m_{p-1,1}\cdots m_{p-1,n_{p-1}}$

\begin{definition} A {\em confluence} is an operation that adds a refinement structure to a
spectral type, while an {\em unfolding} is its inverse,
namely, an operation that removes a refinement structure from a spectral type.
When a spectral type with a refinement structure is expressed using
$\verb/|/$, unfolding is the operation of replacing ``$\verb/|/$'' by ``$\verb/,/$''.
\end{definition}

\begin{exmp}
In general, there are several refinement structures of a spectral type: 

$33,222,111111$ $\xrightarrow{\text{cf}}$

$111111|222,33=(11)(11)(11),33=11\cdot 11\cdot 11|222,33
\qquad
111111|33,222=(111)(111),222
$

\medskip
$42,21111,21111$ $\xrightarrow{\text{cf}}$

$11112|42,21111=(1111)(2),21111$
\qquad
$21111|42,21111=(211)(11),21111$

$21111|21111,42=(2)(1)(1)(1),42$
\qquad\ \ 
$11112|11112|42=((1)(1)(1)(1))((2))$

$21111|21111|42\,=((2)(1)(1))((1)(1))=2\cdot1\cdot1\cdot1\cdot1|211\cdot11|42$
\end{exmp}

\begin{remark}
Middle convolutions and indices of rigidity can be defined similarly for 
spectral types of systems allowing unramified irregular singularities, 
and they are compatible with operations such as confluences and unfoldings.
For a fixed index of rigidity, there are only finitely many such spectral
types that cannot be reduced in order by middle convolutions.
This finiteness extends Theorem~\ref{thm:Fuchsfinite} and follows from the
finiteness of the ``types'' of Weyl group orbits in a general symmetric
Kac--Moody root space (cf.~\cite{HiroeOshima}).
Here, if the Dynkin diagram of a root contains a subdiagram of the form
\scalebox{0.8}{\begin{tikzpicture}
 [root/.style={draw,circle,inner sep=0mm,minimum size=2mm}]
\node[root] (a0) at (0,0) [label=above:$m$]{};
\node[root] (a1) at (0.7,0) [label=above:$m$]{};
\node[root] (a2) at (2.2,0) [label=above:$m$]{};
\draw (-0.5,0)--(a0)--(a1)--(1.2,0) (1.7,0)--(a2)--(2.7,0);
\draw[densely dotted] (1.2,0)--(1.7,0);
\end{tikzpicture}}\ \ 
we may extend its length without changing the corresponding spectral type.
Hence, we regard such roots as belonging to the same ``type.''
\end{remark}

The Deligne-Simpson problem on systems allowing unramified irregular singularities 
was formulated in \cite{HiroeDuke} and solved as follows.
\begin{theorem}[\cite{HiroeDuke, Hiroe2025}]\label{thm:cf}
Let $\mathbf m$ be a spectral type with a refinement structure.
Then $\mathbf m$ is irreducibly realizable by systems \eqref{eq:system} 
satisfying a suitable condition on \eqref{eq:irsystem} 
if and only if the unfolding of $\mathbf m$ is irreducibly realizable.
\end{theorem}

There is also an analytic realization of unfolding and confluence as follows.

\begin{definition}[\cite{Okzv, Ov}]
Let $\mathcal N$ be a differential equation with unramified irregular singularities. 
A versal unfolding $\widetilde{\mathcal N}$ of $\mathcal N$ is a family of 
differential equations which holomorphically depend on some singular points 
as local holomorphic parameters, such that the index of rigidity is independent 
of the parameters and 
$\mathcal N$ is contained in the family as a confluence of singular points of Fuchsian
systems.
\end{definition}

For example, a versal unfolding of a system $\mathcal N$ having an
irregular singularity of Poincar\'e rank $1$ at $\infty$ is of the form
\[
\mathcal M_b: \frac{du}{dx}=\frac{A_1}{1-bx}u+\sum_{j=2}^{p-1}\frac{A_j}{x-c_j}u
\quad(|b|\ll1).
\]
Here, the matrices $A_j$ depend holomorphically on the parameter $b$,
$\idx\mathcal M_b=\idx\mathcal N$, and $\mathcal N=\mathcal M_0$.

\begin{theorem}
[\cite{Okzv, Kawakami2025, Hiroe2025}]  
An irreducible system $\mathcal N$  of spectral type $\mathbf m$ with a
refinement structure admits a versal unfolding.
\end{theorem}
Defining  middle convolutions and additions of versal unfoldings, 
this theorem is proved in the rigid case and conjectured in general by 
\cite{Okzv}.
When the Poincar\'e rank of every irregular singularity of $\mathcal N$
is equal to one, a versal unfolding of $\mathcal N$ was explicitly
constructed by Kawakami \cite{Kawakami2025}. 
The general existence theorem above is proved by Hiroe \cite{Hiroe2025}.

\begin{definition}\label{def:hatP}
We denote by $\hat{\mathcal P}$ the set of spectral types of irreducible
systems allowing unramified irregular singularities; namely, the set of
irreducibly realizable spectral types with refinement structures.
\end{definition}
%%%%%%%%%%%%%%%%%%%%%%%%%%%%%%%%%%%%%%%%%%%%%%%%%%%%%%%%%%%%%%% 
%%%%%%%%%%%%%%%%%%%%%%%%  2.6  %%%%%%%%%%%%%%%%%%%%%%%%%%%%%%%%
\subsection{Harnad duality (Laplace transformations)}
%%%%%%%%%%%%%%%%%%%%%%%%%%%%%%%%%%%%%%%%%%%%%%%%%%%%%%%%%%%%%%%
The Laplace transformation 
\[
 u(x)\mapsto\int_c^\infty u(t)e^{-xt}dt
\]
of $u(x)$ induces the transformation $(x,\tfrac{d}{dx})\mapsto(-\tfrac{d}{dx},x)$ 
of differential operators and hence Harnad duality (cf.~\cite{Harnad}) for differential equations.
Combining this transformation with gauge transformations gives a bijection between irreducible systems $\mathcal N$ having an irregular singularity of Poincar\'e rank $1$ at $\infty$ and regular singularities at finite points.
Consequently, it induces a transformation $\Hd_\sigma$ between their spectral types.
Here, $\sigma$ corresponds to the characteristic exponents at regular singular points, as in the case of $\mc_\sigma$.

Each regular singular point is transformed into an irregular singularity at $\infty$,
and its position determines the exponential factor, namely, the coefficient of the degree-$1$ term of the characteristic exponents, while the constant terms are the original characteristic exponents other than $0$.
The characteristic exponent $0$ corresponds to local holomorphic solutions at finite singular points and is encoded by $\sigma$.
Conversely, the irregular singularity at $\infty$ is transformed into singularities at finite points.
The multiplicities $m^{(0)}_{k,0}$ of the characteristic exponent $0$ in Definition~\ref{def:Hd}
are determined by the order of the resulting spectral type, as specified 
in Definition~\ref{def:Hd}.

Suppose that the Riemann scheme of the equation $\mathcal N$ is given by \eqref{eq:RSirr} and,
moreover, that $\lambda_{j,1}=0$ for $j=1,\dots,n_j$.
The Laplace transformation of $\mathcal N$, which is the Harnad dual of $\mathcal N$, 
has the following Riemann scheme:
\begin{align*}
&
\begin{Bmatrix}
 x=\infty\phantom{AAAAAAAAAAAAAAAA}  & x=-\lambda_i^{(0)}\phantom{AAAAAAAA}\\
 [c_jx+\lambda_{j,\nu}]_{m_{j,\nu}}\ (1\le j\le p,\ 2\le \nu\le n_j)\phantom{AA} & [0]_{\hat n-m_i^{(1)}}
 \phantom{AAAAAAAA}\\
 & [\lambda^{(1)}_{i,\nu}]_{m_{i,\nu}^{(0)}}\ \ (1\le\nu\le R_i)
\end{Bmatrix},\\
&\qquad \hat n:=pn-(m_{1,1}+\cdots+m_{p,1}).\notag
\end{align*}

When considering the Harnad duality, we choose $\infty$ as the first singularity in the Riemann scheme and the corresponding spectral type.

\begin{definition}\label{def:Hd}
The Harnad dual $\Hd\mathbf m$ of a system with an irregular singularity 
of Poincar\'e rank 1 at infinity and regular singularities at finite points is defined 
by
\begin{align*}
\mathbf m&=\bigl(m_{1,1}^{(0)}\cdots m_{1,R_1}^{(0)}\bigr)\cdots\bigl(m_{n_1,1}^{(0)}\cdots 
 m_{n_1,R_{n_1}}^{(0)}\bigr)
{,\,}\underline{m_{2,1}}m_{2,2}\cdots m_{2,n_2}
{,\,}\ldots{\,,}\underline{m_{q,1}}m_{q,2}\cdots m_{q,n_q}\\
&=m_{1,1}^{(0)}\cdots m_{1,R_1}^{(0)}\cdots m_{n_1,1}^{(0)}\cdots m_{n_1,R_{n_1}}^{(0)}
|\,
m_1^{(1)}\cdots m_{n_1}^{(1)},\,
\underline{m_{2,1}}\cdots m_{2,n_2},\,\ldots\,,\underline{m_{q,1}}\cdots m_{q,n_q},\\
\Hd\mathbf m&=
\bigl(m_{2,2}\cdots m_{2,n_2}\bigr)
\cdots{\bigl(}m_{q,2}\cdots m_{q,n_q}{\bigr)},\,
\underline{m_{1,0}^{(0)}}m_{1,1}^{(0)}\cdots m_{1,R_1}^{(0)},\,
\ldots,\, \underline{m_{n_1,0}^{(0)}}m_{n_1,1}^{(0)}\cdots m_{n_1,R_{n_1}}^{(0)},\\
&\qquad m^{(0)}_{k,0} \text{ are determined by } 
\sum_{j=2}^q\sum_{\nu=2}^{n_j}m_{j,\nu}=\sum_{\nu=0}^{R_k}m^{(0)}_{k,\nu}
\ \ (k=1,\dots,n_1). 
\end{align*}
Here, $\Hd_{1,\dots,1}$ is simply denoted by $\Hd$.
\end{definition}

\begin{remark}\label{remark:Hd}
(1) \ 
$\Hd_\sigma$ defines a map between irreducibly realizable spectral types.

\smallskip
\noindent 
(2) \ $\idx \Hd_\sigma\mathbf m=\idx\mathbf m$ and  $\Hd^2=\mathrm{id}$.

\smallskip
\noindent
(3) \ Let $\mathcal M$ be a Fuchsian system with spectral type $\mathbf m$.
Applying a generic linear fractional transformation $\mathrm{L}$ to $\mathcal M$,
we obtain a system that is non-singular at infinity.
The spectral type obtained by applying $\Hd_2$ to the Harnad dual
$\Hd_{1,\dots,1}$ of this system gives the middle convolution
$\mc_{(1,\dots,1)}\mathbf m$:
\begin{align*}
 \mathbf m&=m_{1,1}\cdots m_{1,n_1},\,\ldots\,,m_{p,1}\cdots m_{p,n_p}
 \ \overset{L}{\simeq}\ 
 n,\,\underline{m_{1,1}}m_{1,2}\cdots m_{1,n_1},\,\ldots\,,
    \underline{m_{p,1}}m_{1,2}\cdots m_{p,n_p}\\
 &\xrightarrow{\Hd} \ (m_{1,2}\cdots m_{1,n_1})\cdots(m_{p,2}\cdots m_{p,n_p}),\,
 n'\underline{n}\qquad(n'=(p-1)n-(m_{1,1}+\cdots+m_{p,1}))\\
 &\xrightarrow{\Hd_2} \ 
  n',m'_{1,1}m_{1,2}\cdots m_{1,n_1},\,\ldots\,,m'_{p,1}m_{p,2}\cdots m_{p,n_p}\quad
  (m'_{j,1}=n'-(n-m_{j,1})=m_{j,1}-d_{1,\dots,1}(\mathbf m))\\
 &\simeq m'_{1,1}m_{1,2}\cdots m_{1,n_1},\,\ldots\,,m'_{p,1}m_{p,2}\cdots m_{p,n_p}
 =\mc_{(1,\dots,1)}\mathbf m.
\end{align*}

\noindent
(4) \  
Let $\mathbf m^{(0)}$ be a Fuchsian spectral type.
Consider a confluence $\mathbf m^{(1)}$ of $\mathbf m^{(0)}$ at infinity,
the Harnad dual $\mathbf m^{(2)}$ of $\mathbf m^{(1)}$,
and the unfolding $\mathbf m^{(3)}$ of $\mathbf m^{(2)}$.
Then we obtain the following transformation between Fuchsian spectral types:
\[
 \mathbf m^{(0)} \xrightarrow{\text{cf}} \mathbf m^{(1)}
 \xrightarrow{\text{Hd}}\mathbf m^{(2)}\xrightarrow{\text{uf}}\mathbf m^{(3)}.
\]
\end{remark}
%%%
\begin{exmp}\label{ex:Hd}
We give some examples of Remark~\ref{remark:Hd} (4) with abbreviated expressions.

\noindent
$33,222,111111\xrightarrow{\text{cf}}(11)(11)(11),\underline{3}3\xrightarrow{\Hd}
(3),\underline{1}11,\underline{1}11,\underline{1}11\xrightarrow{\text{uf}}111,111,111$

: $33,222,111111\xrightarrow{(1^2)(1^2)(1^2)}111,111,111$

\noindent
The above second line is a short expression of the first line showing the 
spectral type at $\infty$.

\medskip
\noindent
$33,222,111111\xrightarrow{\text{cf}}(111)(111),\underline{2}22\xrightarrow{\Hd}(22),\underline{1}111,\underline{1}111\xrightarrow{\text{uf}}22,1111,1111$

: $33,222,111111\xrightarrow{(1^3)(1^3)}22,1111,1111$

\noindent
$21,21,111,111\xrightarrow{\text{cf}}(1)(1)(1),\underline{2}1,\underline{2}1
\xrightarrow{\Hd}(1)(1),\underline{1}1,\underline{1}1,\underline{1}1
\xrightarrow{\text{uf}}11,11,11,11,11$

: $21,21,111,111\xrightarrow{(1)(1)(1)}11,11,11,11,11$
\end{exmp}

%%%%%%%%%%%%%%%%%%%%%%%%%%%%%%%%%%%%%%%%%%%%%
%%%%%%%%%%%%%%%%%%   2.7   %%%%%%%%%%%%%%%%%%
\subsection{Examples}\label{sec:example}
%%%%%%%%%%%%%%%%%%%%%%%%%%%%%%%%%%%%%%%%%%%%%
As shown in Example~\ref{ex:Hd}, the transformations considered above
give rise to correspondences between different fundamental spectral types.
These correspondences induce isometric transformations of the star-shaped
Kac--Moody root space and map roots to roots.
Hence, these correspondences can be described within a fixed index of rigidity.
We describe these correspondences for the cases $\mathcal F_0$, $\mathcal F_2$, and $\mathcal F_4$.
The cases $\mathcal F_6$ and $\mathcal F_8$ are given in \cite{Ospect10}.

\subsubsection{$\mathcal F_0$ contains 4 fundamental spectral types}

\noindent
\scalebox{0.98}{\tt\begin{tabular}{|llll|}
\multicolumn{4}{c}{idx=$0$ : \large 4 fundamental spectral types,  1 class}\\ \hline
 \,\ 6:\,33,222,111111\ $E_8^{(1)}$\!\!\!&\,\ 4:\,22,1111,1111\ $E_7^{(1)}$\!\!\!
& \,\ 3:\,111,111,111\ $E_6^{(1)}$\!\!\!
&\,\ {2}:\,11,11,11,11\ $D^{(1)}_4$
\\\hline
\end{tabular}}
%%%%%%%%%%%%%%%%%%%%%%%%%%%%%%%%%%%%%%%%%%%%%%%%%%%%%%%%%%%%%%%%%%%%%%%%%%%%%%%%
\bigskip

Here, the order of each spectral type is indicated at the head.

In the following figure, some correspondences between spectral types 
(cf.~Example~\ref{ex:Hd}) and 
unramified confluences of fundamental spectral types are shown.  

\bigskip
\noindent
%%%%%%%%%%%%%%%%%%%%%%%%%%%%%%%  Graph : Idx 0  %%%%%%%%%%%%%%%%%%%%%%%%%%%%%
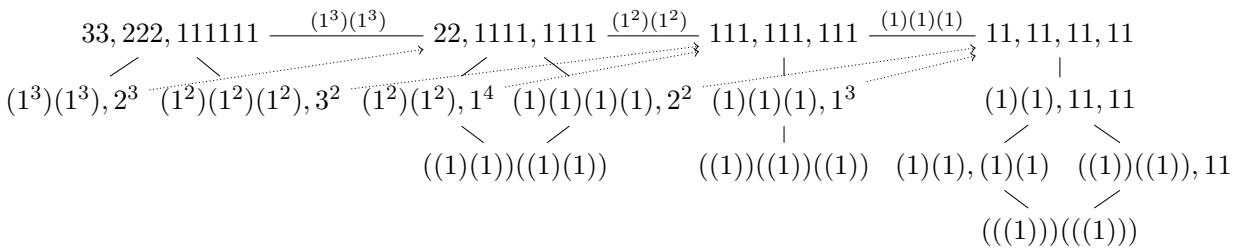
\scalebox{0.96}{\begin{tikzpicture}
\node(0) at(-0.85,0){$33,222,111111$};
\node(1) at(-2.2,-0.9){$(1^3)(1^3),2^3$};
\node(2) at(0.25,-0.9){$(1^2)(1^2)(1^2),3^2$};
\node(3) at(3.9,0){$22,1111,1111$};
\node(4) at(5.1,-0.9){$(1)(1)(1)(1),2^2$};
\node(5) at(2.7,-0.9){$(1^2)(1^2),1^4$};
\node(6) at(3.9,-1.8){$((1)(1))((1)(1))$}; %((1^2))((1^2))$};
\node(7) at(7.6,0){$111,111,111$};
\node(8) at(7.6,-0.9){$(1)(1)(1),1^3$};
\node(9) at(7.6,-1.8){$((1))((1))((1))$};
\node(10) at(11.4,0){$11,11,11,11$};
\node(11) at(11.4,-0.9){$(1)(1),11,11$};
\node(12) at(10.2,-1.8){$(1)(1),(1)(1)$};
\node(13) at(12.7,-1.8){$((1))((1)),11$};
\node(14) at(11.4,-2.7){$(((1)))(((1)))$};
\node at (1.6,0.2) {\scriptsize$(1^3)(1^3)$};
\node at (5.75,0.2) {\scriptsize$(1^2)(1^2)$};
\node at (9.5,0.2) {\scriptsize$(1)(1)(1)$};
\draw (0)--(3) (3)--(7) (7)--(10) (0)--(1) 
(0)--(2) (3)--(4) (3)--(5) (4)--(6) (5)--(6) 
(7)--(8) (8)--(9) (10)--(11) (11)--(12) (11)--(13) 
(12)--(14) (13)--(14) ;
\draw[densely dotted,->]  (1)--(3);
\draw[densely dotted,->]  (2)--(7);
\draw[densely dotted,->]  (5)--(7);
\draw[densely dotted,->]  (4)--(10);
\draw[densely dotted,->]  (8)--(10);
\end{tikzpicture}}

\subsubsection{Examples of fundamental spectral types with $\idx=2-2m\ (m\ge 1)$}
\begin{align*}
Gar_m &:  \overbrace{11,11,\ldots,11}^{m+3}=[11]_{m+3}
&
[11]_{m+3}^T &: (m+2)^2,(m+1)^22,1^{2m+4}
\\
FS_m &: m1,m1,1^{m+1},1^{m+1}
&
Sa_m &: m^2,m^2,(2m-1)1,1^{2m}
\\
D_4^{(m)} &: m^2,m^2,m^2,m(m-1)1
&
E_6^{(m)} &: m^3,m^3,m^2(m-1)1
\\
E_7^{(m)} &: (2m)^2,m^4,m^3(m-1)1
&
E_8^{(m)} &: (3m)^2,(2m)^3,m^5(m-1)1
\end{align*}
Here, we use the notation such as $[11]_3=11,11,11$ and $[2^2]_2,[21^2]_2=22,22,211,211$ etc. 

We remark that $D_4^{(1)}=[11]_4=Gar_1=FS_1=Sa_1$ when $m=1$.

The spectral types $D_4^{(m)}$, $E_6^{(m)}$, $E_7^{(m)}$ and $E_8^{(m)}$ are 
introduced by \cite{Ow}.
The spectral types 
$Gar_m$, $FS_m$, $Sa_m$ and $D_4^{(m)}$ are 
related to higher order Painlev\'e-type equations called
Garnier systems, Fuji-Suzuki systems \cite{FujiSuzuki2, Tsuda2014UC}, Sasano systems \cite{Sasano} 
and matrix Painlev\'e systems,  respectively.

Note that the  inequalities
\[ \ord\mathbf m\le \ord E_8^{(m)}=6m\text{  \ and  \ }
 \Np(\mathbf m)\le \Np([11]_{m+3})=m+3\quad(\forall\mathbf m\in\mathcal F_{2m-2})
\]
are proved by \cite{Ow}, 
which assures $|\mathcal F_{2m-2}|<\infty$ for $m\in\mathbb Z_{>0}$.

%%%%%%%%%%%%%%%%%%%%%%%%%%%%%%%%%%%%%%%%%
%%%%%%%  Example  Idx=-2 %%%%%%%%%%%%%%%%
%%%%%%%%%%%%%%%%%%%%%%%%%%%%%%%%%%%%%%%%%
%%%%%%%%%%%%%%%%%%%%%% Table : Idx=-2 %%%%%%%%%%%%%%%%%%%%%%%%%%%%%%%%%%%%%%
\subsubsection{$\mathcal F_2$ contains 13 fundamental spectral types}

\noindent
\scalebox{0.92}{\tt\begin{tabular}{|llll|}
\multicolumn{4}{c}{\large $\idx=-2$ : 13 fundamental spectral types, 2 classes (cf.~\cite{Ow, HKNS, Kawakami2025})}\\ \hline
\,\ 8:\,44,332,11111111\ $[11]_5^T$\!\!\!\!&\,\ 6:\,33,2211,111111&\,\ 5:\,32,11111,11111
&\,\ 5:\,221,221,11111\\
\,\ 4:\,211,1111,1111&\,\ {4}:\,31,22,22,1111\ $Sa$\!\!\!\!
& \,\ {3}:\,21,21,111,111\ $FS$\!\!\!\!&\,\ {2}:\,11,11,11,11,11\ $[11]_5$\\ \hline
%&\!& \,\ 6:\,222,222,2211\ $E_6^{(2)}$\!\!\\
12:\,66,444,2222211\ $E_8^{(2)}$\!\!\!\!&10:\,55,3331,22222 &\,\ 8:\,44,2222,22211\ $E_7^{(2)}$\!\!\!\!
& \,\ 6:\,222,222,2211\ $E_6^{(2)}$\!\!\\
\,\ {4}:\,22,22,22,211\ $D_4^{(2)}$&&&\\\hline
%
%&&\\\hline
\end{tabular}}
%%%%%%%%%%%%%%%%%%%%%%%%%%%%%%%%%%%%%%%%%%%%%%%%%%%%%%%%%%%%%%%%%%%%%%%%%%%%%

\medskip
\noindent
%%%%%%%%%%%%%%%%%% Graph : Idx=-2 (1/2) %%%%%%%%%%%%%%%%%%%%%%%%%%%
\scalebox{0.95}{\begin{tikzpicture}
%\node at (8,2.8){\small$Gar$};
\node(0) at(-1.1,0){$44,332,11111111$}; % {$4^2,3^22,1^6$};
\node at (-1.1,0.5){$[1^2]_5^T$};
\node(2) at(1.4,0.8){$33,2211,111111$}; %{$3^2,2^21^2,1^6$};
\node(3) at(2.9,1.6){$221,221,11111$}; %{$[221]_2,1^5$};
\node(5) at(3,0){$32,11111,11111$}; %{$32,[1^5]_2$};
\node(6) at(6.6,1.6){${\underline{31,22,22,1111}}$}; %{$31,[2^2]_2,1^4$}
\node  at(6.6,2){\small$Sa$};
\node(9) at(6.6,0.8){$211,1111,1111$}; %{$21^2,1^4,1^4$};
\node(11) at(10.3,1.6){$21,21,111,111$}; %{$[21]_2,[1^3]_2$};
\node  at(10.3,2){\small$FS$};
\node(16) at(14,1.6){$11,11,11,11,11$}; %{$[11]_5$};
\node  at(14,2){\small$Gar$};
\draw (0)--(5) (2)--(3) 
(2)--(9) (3)--(6) (5)--(9) 
(6)--(11) (9)--(11) (11)--(16);
\draw[densely dotted] (0)--(2); 
\node at (0.5,0.45) {\scriptsize$+1$};
\node at (12.13, 1.8) {\scriptsize$(1)(1)(1)$};
\node at (8.45, 1.8) {\scriptsize$(1^2)(1^2)$};
\node at (4.75, 1.8) {\scriptsize$(2)(2)(1)$};
\node at (9.1,1.05) {\scriptsize$(1^2)(1)(1)$};
\node at (4.1,1) {\scriptsize$(1^2)(1^2)(1)(1)$};
\node at (1.55,1.25) {\scriptsize$(21)(21)$};
\node at (1.05,0.2) {\scriptsize$(1^4)(1^4)$};
\node at (5.4,0.3) {\scriptsize$(1^3)(1^2)$};
\end{tikzpicture}}
%%%%%%%%%%%%%%%%%%%%%%%%%%%%%%%%%%%%%%%%%%%%%%%%%%%%%%%%%%%%%%

\medskip
\noindent
%%%%%%%%%%%%%%%%% Graph : Idx=-2  (2/2) %%%%%%%%%%%%%%%%%%
\scalebox{0.95}{\begin{tikzpicture}
\node(0) at (-0.1,0) {$66,444,2222211$};
\node at (-0.2,0.5) {\small$E_8^{(2)}$};
\node(1) at (4,0) {$44,2222,22211$};
\node at (4,0.5) {\small$E_7^{(2)}$};
\node(2) at (8,0) {$222,222,2211$};
\node at (8,0.5) {\small$E_6^{(2)}$};
\node(3) at (11.6,0) {${\underline{22,22,22,211}}$};
\node at (11.6,0.5) {\small$D_4^{(2)}$};
\node (4) at (2,0.8) {$55,3331,22222$};
\draw (0)--(1)--(2)--(3);
\draw[densely dotted] (0)--(4);
\node at (1.1,0.45) {\scriptsize$+3$};
%\node at (6,0.9) {$D_4^{(m)}$-$E_8^{(m)}$-class}; 
\node at (2.05,0.2) {\scriptsize$(2^3)(21^2)$};
\node at (9.8,0.2) {\scriptsize$(2^3)$};
\node at (6.1,0.2) {\scriptsize$(2^2)(2^2)$};
\end{tikzpicture}}
%%%%%%%%%%%%%%%%%%%%%%%%%%%%%%%%%%%%%%%%%%%%%%%%%%%%%%

\medskip
The above dotted line with $+1$ or $+3$ represents the following relations.

\smallskip
$
44,332,11111111\xleftarrow[1]{\mc}\underline{5}4,33\underline{3},\underline{2}1111111\xrightarrow{\text{cf}}
%\xrightarrow[+1]{\mathrm{mc}}
11111\cdot211|54,\underline{3}33\xrightarrow{\mathrm{Hd}} 
33,\underline{1}11111,\underline{2}211
$

$
 66,444,2222211\xleftarrow[1]{\mathrm{mc}}
 \underline{7}6,\underline{5}44,22222\underline{2}1\xleftarrow[1]{\mathrm{mc}}
 7\underline{7},5\underline{5}4,3\underline{2}22221\xleftarrow[1]{\mathrm{mc}}
 \underline{8}7,55\underline{5},3\underline{3}22221\xrightarrow{\text{cf}}
 2^4\cdot3^21|87,5^3$

\quad$\xrightarrow{\mathrm{Hd}} 5^2,\underline{2}2^4,\underline{3}3^21\simeq 55,3331,22222
$

\medskip
Here, $D_4$-fundamental spectral types are indicated by underlines.

Note that $(1^5)(21^2),3^3$ is fundamental (cf.~\cite{HiroeOshima}) but its unfolding $21^7,54,3^3$ is not.

The systems of the spectral types in $\mathcal F_2$ and their confluences are studied by 
\cite{HKNS, Kawakami2025} related to 4-dimensional Painlev\'e-type equations.

\subsubsection{$\mathcal F_4$ contains 37 fundamental spectral types}
\scalebox{0.9}{\begin{tabular}{|llll|}
\multicolumn{4}{c}{\large $\idx=-4$ : 37 fundamental spectral types, 5 classes}\\ \hline
10:\,55,442,1111111111\ $[11]_6^T$\!\!\!\!&\,\ 8:\,44,3311,11111111&\,\ 7:\,331,331,1111111&\,\ 6:\,42,111111,111111\\
\,\ {6}:\,51,33,33,111111\ $Sa$& \,\ 5:\,311,11111,11111\!\!& \,\ {4}:\,31,31,1111,1111\ 
$FS$\!\!\!\!&
\,\ {2}:\,11,11,11,11,11,11\ $[11]_6$
\\ \hline
18:\,99,666,3333321\ $E_8^{(3)}$& 12:\,66,3333,33321\ $E_7^{(3)}$\!\!\!&\,\ 9:\,333,333,3321\ $E_6^{(3)}$
&\,\ {6}:\,33,33,33,321\ $D_4^{(3)}$
\\ \hline
\,\ 9:\,54,333,111111111&\,\ 7:\,43,2221,11111111&\,\ 6:\,321,222,111111& \,\ 6:\,33,21111,11111\\
\,\ 5:\,221,2111,11111&\,\ 4:\,1111,1111,1111&\,\ {4}:\,31,22,211,1111&\,\ {3}:\,21,111,111,111\\
\,\ {3}:\,21,21,21,21,111&&&
\\ \hline
 12:\,66,444,22221111\ $E_8^{(2,2)}$\!\!&10:\,55,3331,222211&\,\ 9:\,441,333,22221 &\,\ 8:\,44,22211,22211\\ 
 \,\ 8:\,44,2222,221111\ $E_7^{(2,2)}$\!\!\!& \,\ 7:\,331,2221,2221&\,\ 6:\,222,2211,2211&\,\ 6:\,222,222,21111\ $E_6^{(2,2)}$\\
\,\ {6}:\,51,33,222,222& \,\ {4}:\,22,22,211,211&\,\ {4}:\,31,31,22,22,22&\,\ {4}:\,22,22,22,1111\ $D_4^{(2,2)}$\\
\hline
14:\,77,554,2222222&10:\,55,3322,22222&\,\ 8:\,332,332,2222&\ \,{6}:\,42,33,33,222\\ \hline
\end{tabular}}
%%%%%%%%%%%%%%%%%%%%%%%%%%%%%%%%%%%%%%%%%%%%%%%%%%%%%%%%%%%%%%%%%%%%%%%%

\bigskip
\noindent
\if0
%%%%%%%%%%%%%% Grap : Idx=-4  1-2/5 %%%%%%%%%%%%%%%%%%%%%%%%%%%%%%%%%%%%
\scalebox{0.95}[0.95]{\begin{tikzpicture}
\node at (7,2.6){$[1^2]_6$-class};
\node(0) at(0.3,0){$5^2,4^22,1^{10}$};
\node at (0.4,0.4){$[1^2]_6^T$};
\node(2) at(1.9,1){$4^2,3^21^2,1^8$};
\node(3) at(2.8,2){$[3^21]_2,1^7$};
\node(5) at(5,0){$42,[1^6]_2$};
\node(6) at(5.1,2){${\underline{51,[3^2]_2,1^6}}$};
\node  at(5,2.4){\small$Sa$};
\node(9) at(7.5,1){$31^2,1^5,1^5$};
\node(11) at(9.5,2){$[31]_2,[1^4]_2$};
\node  at(9.5,2.4){\small$FS$};
\node(16) at(12.7,2){$[1^2]_6$};
\node at (12.7,2.4) {\small$Gar$};
%\end{tikzpicture}}
%
%
%\scalebox{0.95}[0.95]{\begin{tikzpicture}
\node at (7.35,2.2) {\scriptsize$(1^3)(1^3)$};
\node at (11.35,2.2) {\scriptsize$(1)(1)(1)(1)$};
\node at (9.2,1.4) {\scriptsize$(1^3)(1)(1)$};
\node at (3.8,2.4) {\scriptsize$(3)(3)(1)$};
\node at (3,1.4) {\scriptsize$(31)(31)$};
\node at (4.9,1.2) {\scriptsize$(1^3)(1^3)(1)(1)$};
\node at (6.85,0.4) {\scriptsize$(1^4)(1^2)$};
\node at (2.8,0.2) {\scriptsize$(1^5)(1^5)$};
\node at (11.7,0.6) {$[3^21]_2=331,331$ \ etc.}; 
%%%%%%%%%%%%%%%%%%%%%%%%%%%%%%%%%%%%%%%%%%%%%%%%
\node(1) at(1,-1.5){$54,3^3,1^9$};
\node(4) at(2.7,-2.5){$43,2^31,1^7$};
\node(7) at(5,-1.5){$3^2,21^4,1^6$};
\node(8) at(5,-2.5){$321,2^3,1^6$};
\node(10) at(7.5,-1.5){$2^21,21^3,1^5$};
\node(12) at(9.5,-0.5){$[1^4]_3$};
\node(13) at(9.5,-2.5){$31,2^2,21^2,1^4$};
\node(14) at(11.7,-0.5){$[21]_4,1^3$};
\node(15) at(11.7,-1.5){$21,[1^3]_3$};
\draw (0)--(5)  (2)--(3) 
(2)--(9) (3)--(6) (5)--(9) 
(6)--(11)(9)--(11) 
(11)--(16)
(1)--(7) (4)--(8) (4)--(10)  (7)--(10)
(7)--(12) (8)--(13) (10)--(15) 
(12)--(14) (13)--(14) 
(13)--(15) (14)--(15) (10)-- (13) (10)--(14)
;
\draw[densely dotted] (0)--(2) (1)--(4);
\node at (1.2,0.5) {\scriptsize$+2$};
%%\node at (4,-1.9) {\scriptsize$a$};
%\node at (6.2,-0.88) {\scriptsize$a$};
%%\node at (2,-0.5) {\scriptsize$a:(2^2)(21)$};
%\node at (5.6,-0.5) {\scriptsize$a:(21)(1^3)$};
%%\node at (6.2,-0.88) {\scriptsize$b$};
\node at (9.8,-1.3) {\scriptsize$(1^2)(1^2)(1)$};
\node at (7.7,-2) {\scriptsize$(2)(1^2)(1)$};
\node at (4.25,-1.95) {\scriptsize$(1^4)(1^3)$};
\node at (7.1,-2.7) {\scriptsize$(1^3)(1^2)(1)$};
\node at (3.8,-2.8) {\scriptsize$(2^2)(21)$};
\node at (9.8,-1.7) {\scriptsize$(1^2)(1)$};
\node at (9.4,-1.95) {\scriptsize$(1^2)(1)(1)$};
\node at (12.3,-0.85) {\scriptsize$(1)(1)(1)$};
\node at (8.9,-0.9) {\tiny$(1^2)(1)(1)(1)$};
\node at (2,-2) {\scriptsize$+2$};
\node at (3,-1.3) {\scriptsize$(1^5)(1^4)$};
\node at (5.9,-2.05) {\scriptsize$(21)(1^3)$};
\node at (7.3,-0.7) {\scriptsize$(1^3)(1^3)$};
\node at (10.45,-0.35) {\tiny$(1)(1)(1)(1)$};
\node at (11.2,-2.2) {\scriptsize$(1^2)(1^2)$};
\node at (11.5,-1.92) {\scriptsize$(1^2)(1)$};
\node at (12.3,-1.1) {\scriptsize$(1^2)(1)$};
\end{tikzpicture}}
\else

\scalebox{0.95}[0.95]{\begin{tikzpicture}
\node at (7,2.6){$[1^2]_6$-class};
\node(0) at(0.3,0){$5^2,4^22,1^{10}$};
\node at (0.4,0.4){$[1^2]_6^T$};
\node(2) at(1.9,1){$4^2,3^21^2,1^8$};
\node(3) at(2.8,2){$[3^21]_2,1^7$};
\node(5) at(5,0){$42,[1^6]_2$};
\node(6) at(5.1,2){${\underline{51,[3^2]_2,1^6}}$};
\node  at(5,2.4){\small$Sa$};
\node(9) at(7.5,1){$31^2,1^5,1^5$};
\node(11) at(9.5,2){$[31]_2,[1^4]_2$};
\node  at(9.5,2.4){\small$FS$};
\node(16) at(12.7,2){$[1^2]_6$};
\node at (12.7,2.4) {\small$Gar$};
\draw (0)--(5)  (2)--(3) 
(2)--(9) (3)--(6) (5)--(9) 
(6)--(11)(9)--(11) 
(11)--(16);
\node at (7.35,2.2) {\scriptsize$(1^3)(1^3)$};
\node at (11.35,2.2) {\scriptsize$(1)(1)(1)(1)$};
\node at (11.7,0.6) {$[3^21]_2=331,331$ \ etc.}; 
\draw[densely dotted] (0)--(2);
\node at (1.2,0.5) {\scriptsize$+2$};
\node at (9.2,1.4) {\scriptsize$(1^3)(1)(1)$};
\node at (3.8,2.4) {\scriptsize$(3)(3)(1)$};
\node at (3,1.4) {\scriptsize$(31)(31)$};
\node at (4.9,1.2) {\scriptsize$(1^3)(1^3)(1)(1)$};
\node at (6.85,0.4) {\scriptsize$(1^4)(1^2)$};
\node at (2.8,0.2) {\scriptsize$(1^5)(1^5)$};
%%%%%%%%%%%%%%%%%%%%%%%%%%%%%%%%%%%%%%%%%%%%%%%%
\end{tikzpicture}}
\medskip

\quad$55,442,1111111111\xleftarrow[2]{\mc}\cdot\xrightarrow{\text{cf}}
 1111111\cdot311|75,\underline{4}44\xrightarrow{\textrm{Hd}}44,\underline{1}1111111,\underline{3}311$

Here, $\xleftarrow[d]{\mc}$ \ represents the map $\mathcal R$ and $d$ is 
the difference of the orders thorough the map.

\scalebox{0.95}[0.95]{\begin{tikzpicture}
\node(1) at(1,-1.5){$54,3^3,1^9$};
\node(4) at(2.7,-2.5){$43,2^31,1^7$};
\node(7) at(5,-1.5){$3^2,21^4,1^6$};
\node(8) at(5,-2.5){$321,2^3,1^6$};
\node(10) at(7.5,-1.5){$2^21,21^3,1^5$};
\node(12) at(9.5,-0.5){$[1^4]_3$};
\node(13) at(9.5,-2.5){$31,2^2,21^2,1^4$};
\node(14) at(11.7,-0.5){$[21]_4,1^3$};
\node(15) at(11.7,-1.5){$21,[1^3]_3$};
\draw
(1)--(7) (4)--(8) (4)--(10)  (7)--(10)
(7)--(12) (8)--(13) (10)--(15) 
(12)--(14) (13)--(14) 
(13)--(15) (14)--(15) (10)-- (13) (10)--(14)
;
\draw[densely dotted] (1)--(4);
%%\node at (4,-1.9) {\scriptsize$a$};
%\node at (6.2,-0.88) {\scriptsize$a$};
%%\node at (2,-0.5) {\scriptsize$a:(2^2)(21)$};
%\node at (5.6,-0.5) {\scriptsize$a:(21)(1^3)$};
%%\node at (6.2,-0.88) {\scriptsize$b$};
\node at (9.8,-1.3) {\scriptsize$(1^2)(1^2)(1)$};
\node at (7.7,-2) {\scriptsize$(2)(1^2)(1)$};
\node at (4.25,-1.95) {\scriptsize$(1^4)(1^3)$};
\node at (7.1,-2.7) {\scriptsize$(1^3)(1^2)(1)$};
\node at (3.8,-2.8) {\scriptsize$(2^2)(21)$};
\node at (9.8,-1.7) {\scriptsize$(1^2)(1)$};
\node at (9.4,-1.95) {\scriptsize$(1^2)(1)(1)$};
\node at (12.3,-0.85) {\scriptsize$(1)(1)(1)$};
\node at (8.9,-0.9) {\tiny$(1^2)(1)(1)(1)$};
\node at (2,-2) {\scriptsize$+2$};
\node at (3,-1.3) {\scriptsize$(1^5)(1^4)$};
\node at (5.9,-2.05) {\scriptsize$(21)(1^3)$};
\node at (7.3,-0.7) {\scriptsize$(1^3)(1^3)$};
\node at (10.45,-0.35) {\tiny$(1)(1)(1)(1)$};
\node at (11.2,-2.2) {\scriptsize$(1^2)(1^2)$};
\node at (11.5,-1.92) {\scriptsize$(1^2)(1)$};
\node at (12.3,-1.1) {\scriptsize$(1^2)(1)$};
\end{tikzpicture}}
\fi
%%%%%%%%%%%%%%%%%%%%%%%%%%%%%%%%%%%%%%%%%%%%%%%%%%%%

\medskip
\quad$54,333,111111111\xleftarrow[2]{\mc}\cdot\xrightarrow{\text{cf}}
 1^6\cdot221|65,\underline{4}43
\xrightarrow{\mathrm{Hd}}43,\underline{1}111111,\underline{2}221$
%\medskip

\bigskip
\scalebox{0.95}[0.95]{\begin{tikzpicture}
\node at (8.3,1.75) {$E_8^{(2,2)}$-class};
\node(0) at(3,1){$6^2,4^3,2^41^4$};
\node at (3,1.45) {\small $E_8^{(2,2)}$};
\node(1) at(4.5,-2){$5^2,3^31,2^41^2$};
\node(2) at(5.5,-1){$4^21,3^3,2^41$};
\node(3) at(6.5,1){$4^2,2^4,2^21^4$};
\node at (6.5,1.45) {\small $E_7^{(2,2)}$};
\node(4) at(6.5,0){$4^2,[2^31^2]_2$};
\node(5) at(8.5,-2){$3^21,[2^31]_2$};
\node(6) at(10,1){$[2^3]_2,21^4$};
\node at (10,1.45) {\small $E_6^{(2,2)}$};
\node(7) at(10,-1){$51,3^2,[2^3]_2$};
\node(8) at(10,0){$2^3,[2^21^2]_2$};
\node(9) at(13.2,1){${\underline{[2^2]_3,1^4}}$};
\node(10) at(13.2,0){$[2^2]_2,[21^2]_2$};
\node(11) at(13.2,-1){$[31]_2,[22]_3$};
\node at (13.2,1.45) {\small $D_4^{(2,2)}$};
\draw (0)--(3) (1)--(5) (2)--(7) (3)--(6)
(4)--(8) (5)--(11) (6)--(9) (6)--(10) (7)--(11) 
(8)--(10) (8)--(11) (9)--(10) (10)--(11) ;
\draw[densely dotted] (0)--(1) (0)--(2) (0)--(4)
;
\node at (3.8,-0.5) {\scriptsize$+3$};
\node at (4.3,0) {\scriptsize$+6$};
\node at (4.8,0.5) {\scriptsize$+2$}
;
\node at (4.8,1.2) {\scriptsize $(2^3)(21^4)$};
\node at (8.3,1.2) {\scriptsize $(2^2)(2^2)$};
\node at (8.3,0.2) {\scriptsize $(2^2)(21^2)$};
\node at (11.5,1.2) {\scriptsize $(2)(2)(2)$};
\node at (12.2,0.6) {\scriptsize$(2)(1^2)(1^2)$};
\node at (11.5,0.2) {\scriptsize$(2)(2)(1^2)$};
\node at (12.2,-0.4) {\scriptsize$(2)(2)(1)(1)$};
\node at (11.5,-0.8) {\scriptsize$(2)(2)(2)$};
\node at (7.8,-0.8) {\scriptsize$(2^2)(2^2)(1)$};
\node at (11.1,-1.7) {\scriptsize$(2)(2)(2)(1)$};
\node at (6.6,-1.8) {\scriptsize$(2^21)(2^21)$};
\node at (13.8,0.62) {\scriptsize$(1^2)(1^2)$};
\node at (13.8,0.38) {\scriptsize$(2)(2)$};
\node at (13.8,-.38) {\scriptsize$(2)(1)(1)$};
\node at (13.8,-.62) {\scriptsize$(2)(2)$};
\end{tikzpicture}}

\smallskip
\quad$ 66,444,22221111\xleftarrow[2]{\mc}\cdot\xrightarrow{\text{cf}}2221\cdot2221|77,\underline{6}44
\xrightarrow{\mathrm{Hd}}44,\underline{1}2221,\underline{1}2221
$

\quad$ 66,444,22221111\xleftarrow[6]{\mc}\cdot\xrightarrow{\text{cf}}
2221\cdot33\cdot41|765,\underline{9}9
\xrightarrow{\mathrm{Hd}}\underline{2}2221,\underline{3}33,\underline{4}41
$

%\quad $99,765,43^22^31^2\xrightarrow{\mathrm{mc}}97,655,3^22^41^2
%\xrightarrow{\mathrm{mc}}77,554,32^41^3
%\xrightarrow{\mathrm{mc}}76,544,2^51^3
%\xrightarrow{\mathrm{mc}}66,444,2^41^4$

\quad$ 66,444,22221111\xleftarrow[3]{\mc}\cdot\xrightarrow{\text{cf}}
22211\cdot331|87,\underline{5}55
\xrightarrow{\mathrm{Hd}}55,\underline{2}22211,\underline{3}331$

\medskip
We have another class:
\smallskip

$ {77,554,2222222}
  \xleftarrow[1]{\mc}\cdot\xrightarrow{\text{cf}}
  2^4\cdot32^2|87,5^3\xrightarrow{\mathrm{Hd}}
  5^2,\underline{2}2^4,\underline{3}32^2={55,3322,22222}$

$\overset{\text{cf}}{\to} 32\cdot32|55,\underline{2}2222 \xrightarrow{\mathrm{Hd}}{2222,\underline{3}32,\underline{3}32}
\overset{\text{cf}}\to 3\cdot3\cdot2|332,\underline{2}222\xrightarrow{\textrm{Hd}}{222,\underline{3}3,\underline{3}3,\underline{4}2}$
%%%%%%%%%%%%%%%%%%%%%%%%%%%%%%%%%%%%%%%%

\bigskip
Putting $m=3$ in the following, we have the case when $\idx\mathbf m=2-2m=-4$.

\noindent
\scalebox{0.95}[0.95]{\begin{tikzpicture}
\node at (7.4,2.9){$[1^2]_{m+3}$-class\ $(m\ge2)$};
\node(0) at(-0.2,0){$(m+2)^2,(m+1)^22,1^{2m+4}$};
\node at (-0.4,0.4){$[1^2]_{m+3}^T$};
\node(2) at(1.1,1){$(m+1)^2,m^21^2,1^{2m+2}$};
\node(3) at(2.3,2){$[m^21]_2,1^{2m+1}$};
\node(5) at(8.1,0){$(m+1)2,[1^{m+3}]_2$};
\node(6) at(6.2,2){${\underline{(2m-1)1,[m^2]_2,1^{2m}}}$};
\node  at(6.2,2.4){\small$Sa$};
\node(9) at(9.2,1){$m1^2,1^{m+2},1^{m+2}$};
\node(11) at(10.2,2){$[m1]_2,[1^{m+1}]_2$};
\node  at(10.2,2.4){\small$FS$};
\node(16) at(13.2,2){$[1^2]_{m+3}$};
\node  at(13,2.4){\small$Gar$};
\draw (0)--(5) (2)--(3) 
(2)--(9) (3)--(6) (5)--(9)
(6)--(11) (9)--(11) (11)--(16);
\draw[densely dotted] (0)--(2); 
\node at (0.9,0.57) {\scriptsize$+(m{-}1)$};
\node at (11.9,2.4) {\scriptsize$(1)\cdots(1)$};
\node at (8.5,2.4) {\scriptsize$(1^m)(1^m)$};
\node at (4,2.4) {\scriptsize$(m)(m)(1)$};
\node at (10.5,1.5) {\scriptsize$(1^m)(1)(1)$};
\node at (9.5,0.5) {\scriptsize$(1^{m+1})(1^2)$};
\node at (5.4,1.2) {\scriptsize$(1^m)(1^m)(1)(1)$};
\node at (4.4,0.2) {\scriptsize$(1^{m+2})(1^{m+2})$};
\node at (2.5,1.5) {\scriptsize$(m1)(m1)$};
\end{tikzpicture}}

\noindent
\scalebox{0.95}[0.95]{$1^{2m+4},(m{+}2)^2,(m{+}1)^22\xleftarrow[m{-}1]{\mathrm{mc}}
 1^{2m+1}\cdot m1^2|(2m{+}1)(m{+}2),(m{+}1)^3\xrightarrow{\mathrm{Hd}}
 (m{+}1)^2,\underline{1}1^{2m+1},\underline{m}m1^2
$}

\bigskip
\noindent
\scalebox{0.95}[0.95]{\begin{tikzpicture}
\node(0) at (-1.2,0) {$(3m)^2,(2m)^3,m^5(m{-}1)1$};
\node at (-1.2,0.5) {$E_8^{(m)}$};
\node(1) at (3.9,0) {$(2m)^2,m^4,m^3(m{-}1)1$};
\node at (3.9,0.5) {$E_7^{(m)}$};
\node(2) at (8.8,0) {$[m^3]_2,m^2(m{-}1)1 $};
\node at (8.8,0.5) {$E_6^{(m)}$};
\node(3) at (12.8,0) {${\underline{[m^2]_3,m(m{-}1)1}}$};
\node at (12.8,0.5) {$D_4^{(m)}$};
\draw (0)--(1)--(2)--(3);
\node at (6.5,1) {$D_4^{(m)}$-$E_8^{(m)}$-class\ $(m\ne2)$}; 
\node at (1.5,0.35) {\tiny$(m^3)(m^2m-11)$};
\node at (10.85,0.2) {\scriptsize$(m^3)$};
\node at (6.7,0.2) {\scriptsize$(m^2)(m^2)$};
\end{tikzpicture}}

Suzuki \cite{Suzuki} studied Fuchsian systems of the spectral types in $\mathcal F_4$ 
to get 6-dimensional Painlev\'e systems.
%%%%%%%%%%%%%%%%%%%%%%%%%%%%%%%%%%%%%%%%%%%%%%%%%%%
%%%%%%%%%%%%%%%%   3   %%%%%%%%%%%%%%%%%%%%%%%%%%%%
\section{Main results}
%%%%%%%%%%%%%%%%   3.1  %%%%%%%%%%%%%%%%%%%%%%%%%%%
\label{sec:main}
%%%%%%%%%%%%%%%%%%%%%%%%%%%%%%%%%%%%%%%%%%%%%%%%%%%
\subsection{A main theorem}
In this section, we state our main theorem on the following equivalence classes. 
\begin{definition} [\cite{Kawakami2025}] If a spectral type $\mathbf m$
of an irreducible differential equation $\mathcal M$ is transformed to a spectral type 
$\mathbf m'$ by a finite sequence of gauge transformations, 
confluences, unfoldings, Harnad duality and fractional linear transformations 
of $\mathcal M$, we say that $\mathbf m$ and $\mathbf m'$ are equivalent
and write $\mathbf m\sim\mathbf m'$. Moreover, we set
\begin{equation*}
 	\mathcal F_{[\mathbf n]}:=\{\mathbf m\in\mathcal F\mid \mathbf m\sim\mathbf n\}.
\end{equation*}
\end{definition}
%%%
\begin{remark}
Since $\idx\mathbf m=\idx\mathbf n$ for $\mathbf m\in\mathcal F_{[\mathbf n]}$, we have
$|\mathcal F_{[\mathbf n]}|<\infty$.
\end{remark}

Under the notation in Definition~\ref{def:mcsim}, 
set 
\begin{equation}
 \mathcal F^T_\ell:=\mathcal F^T\cap\mathcal F_\ell
 =\{\mathbf m\in\mathcal F^{T}\mid \idx\mathbf m=-\ell\}\quad(\ell=0,2,4,\ldots).
\end{equation}
The following table shows the numbers of certain elements of $\mathcal F_\ell$.
\medskip

\noindent\ \ 
\scalebox{1}{
\begin{tabular}{|r|r|r|r|r||r|r|r|r|r|r|}
\multicolumn{11}{c}{\scalebox{0.9}[1]{\bf Equivalence classes and $E_8,\,E_7,\,E_6,\,D_4$-fundamental spectral types in $\mathcal F_\ell$}}
\\ \hline
$\ell$ & $|\mathcal F_\ell|\ \ $ &  $\mathrm{N_p}{=}3$&$\mathrm{N_p}{=}4$&\!\!$\ord\le$\!\!& 
$\!\!|\mathcal F_{[\mathbf n]}|\le\!\!$ & $\!\!|\mathcal F_{[\mathbf n]}|\ge\!\!$ &
$\!|\mathcal F_\ell^{E_8}|\!$
&\!$|\mathcal F_\ell^{E_7}|$\!&\!$|\mathcal F_\ell^{E_6}|$\!&$\!|\mathcal F_\ell^{D_4}|$\!\\ \hline
0  &    4  &    3 &    1 &  6&  4& 4&   1& 1  & 1 &  1\\
2  &    13 &    9 &    3 & 12&  8& 5&   2&  3& 2 &  2\\
4  &    37 &   25 &    9 & 18&  9& 4&   5&  6&  6&  4\\
6  &    69 &   46 &   17 & 24& 12& 4&   7& 12& 10&  6\\
8  &   113 &   73 &   29 & 30& 23& 4&  10& 15& 12&  8\\
10 &   198 &  127 &   50 & 36& 36& 3&  15& 25& 20&  9\\ \hline
20 &  1186 &  680 &   345& 66&104& 4&  40& 80& 67& 25\\ \hline
50 & 40617 &18878 & 12953&156&757& 4& 351&931&701&137\\ \hline
80 & 419160&165023&136352&246&3173&3&1456&4386&3251&414\\ \hline
100&1461132&522509&471378&306&7613&3&3181&10389&7500&743\\ \hline
\end{tabular}}
\medskip

For example, the column headed ``$|\mathcal F_{[\mathbf n]}|\le$'' gives
the maximum cardinality of $\mathcal F_{[\mathbf n]}$ among
$\mathbf n\in\mathcal F_\ell$.
In particular, $
\max_{\mathbf n\in\mathcal F_{100}}|\mathcal F_{[\mathbf n]}|=7613$.
Moreover, $|\{\mathbf m\in\mathcal F_{100}\mid \mathrm{N_p}(\mathbf m)=4\}|=471378$ etc.

\medskip
\begin{theorem}\label{thm:main}
$\mathcal F^{E_8}$ is a set of complete representatives of $\mathcal F/\!\sim$.  Namely,
\begin{equation*}
\mathcal F=\bigsqcup_{\mathbf n\in\mathcal F^{E_8}}\mathcal F_{[\mathbf n]}.
\end{equation*}
Moreover, every $\mathbf n\in\mathcal F^{E_8}$ satisfies
\[
\ord\mathbf n>\ord\mathbf m\quad(\forall\mathbf m\in\mathcal F_{[\mathbf n]}\setminus\{\mathbf n\}).
\]
\end{theorem}
\begin{remark}
Since $\mathcal P/\!\underset{\mc}\sim\ \simeq \mathcal F\cup\{\bf 1\}$ and 
$\hat{\mathcal P}/\!\sim\ \simeq \mathcal P/\!\sim$, we have
\[
  \hat{\mathcal P}/\!\sim\ \simeq \mathcal P/\!\sim\ \simeq \mathcal F^{E_8}\cup\{\bf 1\}.
\] 
Here, recall that 
 $\hat{\mathcal P}$ and $\mathcal P$ are the set of irreducibly realizable spectral 
types of systems allowing unramified irregular singularities 
(cf.~Definition~\ref{def:hatP}) and that of Fuchsian systems, respectively. 
\end{remark}
%%%%%%%%%%%%%%%%%%%%%%%%%%%%%%%%%%%%%%%%%%%%%%%%%%%
%%%%%%%%%%%%%%%%   3.2  %%%%%%%%%%%%%%%%%%%%%%%%%%%
\subsection{Proof of the main theorem}
%%%%%%%%%%%%%%%%%%%%%%%%%%%%%%%%%%%%%%%%%%%%%%%%%%%
By defining a map $\mathcal U$ of $\mathcal P$ to $\mathcal F^{E_8}\cup\{\mathbf 1\}$,
the proof of Theorem~\ref{thm:main} is reduced to Lemma~\ref{lem:main}.
\begin{definition} Let 
$\mathbf m=m_{11}\cdots m_{1n_1},\,m_{21}\cdots m_{2n_2},\,\ldots,\,m_{p1}\cdots m_{pn_p} \in \mathcal P$.

\smallskip
\noindent
(1) \ 
We define $\mathcal R\mathbf m\in\mathcal F\cup\{\mathbf 1\}$ by the condition 
$\mathcal R\mathbf m\underset{\mc}\sim\mathbf m$ (cf.~Definition~\ref{def:mcsim} and \eqref{eq:mcsim}).

\smallskip
\noindent
(2) \ Let $\mathcal M$ be a Fuchsian system of the spectral type $\mathbf m$
and let $L\mathbf m$ be the spectral type of the system obtained by
applying a generic linear fractional transformation on $\mathbb P^1$ to $\mathcal M$.
We denote by $\mathcal S\mathbf m$ the unfolding of the spectral type obtained by 
applying the Harnad dual  $\Hd_{0,\ldots,0}$ to $L\mathbf m$:
\begin{align}\label{def:Sm}
\begin{split}
\mathcal S\mathbf m&:=\bigl[[(p-1)n,n],\,[\overbrace{n,\dots,n}^p],\,
 [\{m_{j\nu}\mid j=1,\dots,p,\ \nu=1,\dots,n_j\}]\bigr],\\
 n&:=\ord\mathbf m.
\end{split}
\end{align}
Then, we define
\begin{equation*}
\mathcal U:\mathcal P\to\mathcal F^{E_8}\cup\{\mathbf 1\},\ \ \mathcal U\mathbf m:=\mathcal R\mathcal S^2\mathbf m.
\end{equation*}
\end{definition}

\begin{remark}\label{remark:defS}
{\rm(1)} Note that $\mathcal R\mathbf m$
is obtained from $\mathbf m$ by successively applying middle convolutions 
according to the Katz reduction procedure, until no further reduction is possible. 

Moreover,
\eqref{def:Sm} is obtained by
\begin{align*}
 \mathbf m&\simeq n,\,\underline{0}m_{11}\cdots m_{1n_1},\,\cdots,\,\underline{0}m_{p1}
  \cdots m_{p,n_p}\\
 &\xrightarrow{\Hd}\ (m_{11}\cdots m_{1n_1})\cdots(m_{p1}\cdots m_{p,n_p}),\,((p-1)n)n\\
 &\xrightarrow{\text{uf}}\ 
  \mathcal S\mathbf m=m_{11}\cdots m_{1n_1}\cdots m_{p1}\cdots m_{p,n_p},\,\overbrace{n\cdots n}^p,\,((p-1)n)n
\end{align*}
and $\mathcal S\mathbf m$ is $(2,p,*)$-type. Then,
\begin{equation*}
\mathcal S^2\mathbf m=\bigl[[2pn,pn],\,[pn,pn,pn],\,[(p-1)n,\overbrace{n,\dots,n}^{p+1},\{m_{j,\nu}\}]\bigr].
\end{equation*}

\smallskip
\noindent
{\rm (2)} \ 
It follows from the above definition that if $\mathbf m=(m_{j,\nu}),\,
\mathbf m'=(m'_{j,\nu})\in\mathcal P$ satisfies 
$\{m_{j,\nu}\}=\{m'_{j,\nu}\}$ and $\Np(\mathbf m)=\Np(\mathbf m')$, 
then $\mathcal S\mathbf m\simeq\mathcal S\mathbf m'$ and $\mathbf m\sim\mathbf m'$.
  For example, 
\begin{equation*}
 44,2211,2211\sim 44,222,21111 \in\mathcal F_4.
\end{equation*}
\end{remark}

The root corresponding to $\mathcal S\mathbf m$ for
$\alpha_{\mathbf m}$ in Definition~\ref{def:starKac} is as follows:

\medskip
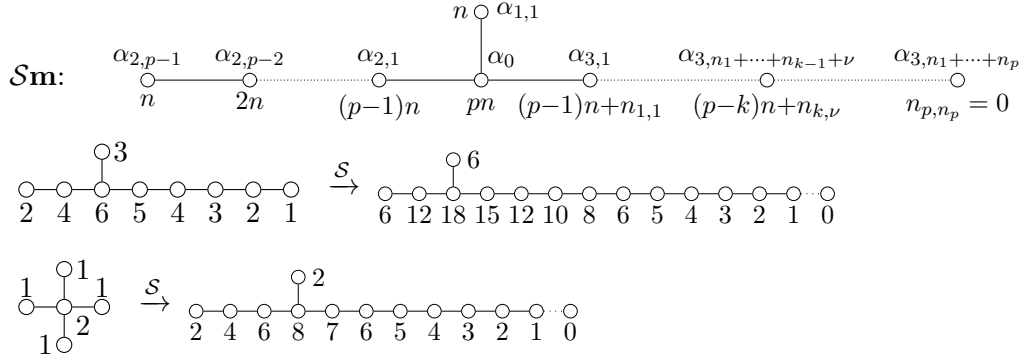
\raisebox{6mm}{$\mathcal S\mathbf m$:\quad}
\scalebox{0.9}{\begin{tikzpicture}
 [root/.style={draw,circle,inner sep=0mm,minimum size=2mm}];
\node[root] (b1) at (-1.9,0) {}; 
\node[root] (a1) at (-0.4,0) {}; 
\node[root] (a2) at (1.5,0) {}; 
\node[root] (a3) at (3,0) {}; 
\node[root] (a4) at (4.6,0) {}; 
\node[root] (a5) at (7.2,0) {}; 
\node[root] (a6) at (10,0) {}; 
\node[root] (a0) at (3,1) {}; 
\node at (-1.9,-0.3) {$n$};
\node at (-1.9,0.3) {$\alpha_{2,p{-}1}$};
\node at (-0.4,-0.3) {$2n$};
\node at (-0.4,0.3) {$\alpha_{2,p{-}2}$};
\node at (1.5,-0.4) {$(p{-}1)n$};
\node at (1.5,0.3) {$\alpha_{2,1}$};
\node at (3,-0.4) {$pn$};
\node at (3.3,0.3) {$\alpha_0$};
\node at (4.6,-0.4) {$(p{-}1)n{+}n_{1,1}$};
\node at (4.6,0.3) {$\alpha_{3,1}$};
\node at (7.2,-0.4) {$(p{-}k)n{+}n_{k,\nu}$};
\node at (7.2,0.3) {$\alpha_{3,n_1+\cdots+n_{k-1}+\nu}$};
\node at (10,0.3) {$\alpha_{3,n_1+\cdots+n_p}$};
\node at (10,-0.4) {$n_{p,n_p}=0$};
\node at (2.7,1) {$n$};
\node at (3.55,1) {$\alpha_{1,1}$};
\draw (b1)--(a1) (a2)--(a3)--(a4)  (a0)--(a3);
\draw[densely dotted] (a1)--(a2) (a4)--(a5)--(a6);
\end{tikzpicture}}

\scalebox{1}{\begin{tikzpicture}
 [root/.style={draw,circle,inner sep=0mm,minimum size=2mm}]
\node[root] (a1) at (0,0) {}; 
\node[root] (a2) at (0.5,0) {}; 
\node[root] (a3) at (1,0) {}; 
\node[root] (a4) at (1.5,0) {}; 
\node[root] (a5) at (2,0) {}; 
\node[root] (a6) at (2.5,0) {}; 
\node[root] (a7) at (3,0) {}; 
\node[root] (a8) at (3.5,0) {}; 
\node[root] (a0) at (1,0.5) {}; 
\node at (0,-0.3) {2};
\node at (0.5,-0.3) {4};
\node at (1,-0.3) {6};
\node at (1.5,-0.3) {5};
\node at (2,-0.3) {4};
\node at (2.5,-0.3) {3};
\node at (3,-0.3) {2};
\node at (3.5,-0.3) {1};
\node at (1.25,0.5) {3};
%\node at (1,1.3) {$\tilde E_8$}; 
\draw (a1)--(a2)--(a3)--(a4)--(a5)--(a6)--(a7)--(a8) (a0)--(a3);
%\node at (1.2,-0.7) {$3^2,2^3,1^6$};
%\node at (0.9,0.27) {\scriptsize 3};
%\node at (0.9,0.8) {\scriptsize 3};
%\node at (0.75,0.15) {\scriptsize 2};
%\node at (0.25,0.15) {\scriptsize 2};
%\node at (-0.25,0.15) {\scriptsize 2};
% (\node at (1.25,0.15) {\scriptsize 1};
%\node at (1.75,0.15) {\scriptsize 1};
%\node at (2.25,0.15) {\scriptsize 1};
%\node at (2.75,0.15) {\scriptsize 1};
%\node at (3.25,0.15) {\scriptsize 1};
%\node at (3.75,0.15) {\scriptsize 1};
\end{tikzpicture}}
\ \raisebox{5mm}{$\xrightarrow{\mathcal S}$}\ 
\scalebox{0.9}{\begin{tikzpicture}
 [root/.style={draw,circle,inner sep=0mm,minimum size=2mm}]
\node[root] (a1) at (0,0) {}; 
\node[root] (a2) at (0.5,0) {}; 
\node[root] (a3) at (1,0) {}; 
\node[root] (a4) at (1.5,0) {}; 
\node[root] (a5) at (2,0) {}; 
\node[root] (a6) at (2.5,0) {}; 
\node[root] (a7) at (3,0) {}; 
\node[root] (a8) at (3.5,0) {}; 
\node[root] (a0) at (1,0.5) {}; 
\node[root] (aa) at (4,0) {}; 
\node[root] (ab) at (4.5,0) {}; 
\node[root] (ac) at (5,0) {}; 
\node[root] (ad) at (5.5,0) {}; 
\node[root] (ae) at (6,0) {}; 
\node[root] (af) at (6.5,0) {}; 
\node at (0,-0.3) {6};
\node at (0.5,-0.3) {12};
\node at (1,-0.3) {18};
\node at (1.5,-0.3) {15};
\node at (2,-0.3) {12};
\node at (2.5,-0.3) {10};
\node at (3,-0.3) {8};
\node at (3.5,-0.3) {6};
\node at (1.3,0.5) {6};
\node at (4,-0.3) {5};
\node at (4.5,-0.3) {4};
\node at (5,-0.3) {3};
\node at (5.5,-0.3) {2};
\node at (6,-0.3) {1};
\node at (6.5,-0.3) {0};
\draw (a1)--(a2)--(a3)--(a4)--(a5)--(a6)--(a7)--(a8)--(aa)--(ab)--(ac)--(ad)--(ae)
(a0)--(a3);
\draw[dotted] (ae)--(af);
\end{tikzpicture}}

\medskip
\raisebox{-2mm}{\begin{tikzpicture}
 [root/.style={draw,circle,inner sep=0mm,minimum size=2mm}]
\node[root] (a0) at (0,0) {}; 
\node[root] (a1) at (-0.5,0) {}; 
\node[root] (a2) at (0.5,0) {}; 
\node[root] (a3) at (0,0.5) {}; 
\node[root] (a4) at (0,-0.5) {}; 
\draw (a1)--(a0)--(a2) (a3)--(a0)--(a4);
\node at (0.25,0.5) {1};
\node at (-0.5,0.25) {1};
\node at (0.5,0.25) {1};
\node at (-0.25,-0.5) {1};
\node at (0.25,-0.25) {2};
%\node at (0,0.9) {$\tilde D_4$};
%\node at (0,-1.2) {$1^2,1^2,1^2,1^2$};
\end{tikzpicture}}
\ \raisebox{5mm}{$\xrightarrow{\mathcal S}$}\ 
\scalebox{0.9}{\begin{tikzpicture}
 [root/.style={draw,circle,inner sep=0mm,minimum size=2mm}]
\node[root] (b1) at (-0.5,0) {}; 
\node[root] (a1) at (0,0) {}; 
\node[root] (a2) at (0.5,0) {}; 
\node[root] (a3) at (1,0) {}; 
\node[root] (a4) at (1.5,0) {}; 
\node[root] (a5) at (2,0) {}; 
\node[root] (a6) at (2.5,0) {}; 
\node[root] (a7) at (3,0) {}; 
\node[root] (a8) at (3.5,0) {}; 
\node[root] (a0) at (1,0.5) {}; 
\node[root] (aa) at (4,0) {}; 
\node[root] (ab) at (4.5,0) {}; 
\node[root] (ac) at (5,0) {}; 
%\node[root] (ad) at (5.5,0) {}; 
%\node[root] (ae) at (6,0) {}; 
%\node[root] (af) at (6.5,0) {}; 
\node at (-0.5,-0.3) {2};
\node at (0,-0.3) {4};
\node at (0.5,-0.3) {6};
\node at (1,-0.3) {8};
\node at (1.5,-0.3) {7};
\node at (2,-0.3) {6};
\node at (2.5,-0.3) {5};
\node at (3,-0.3) {4};
\node at (3.5,-0.3) {3};
\node at (1.3,0.5) {2};
\node at (4,-0.3) {2};
\node at (4.5,-0.3) {1};
\node at (5,-0.3) {0};
\draw (b1)--(a1)--(a2)--(a3)--(a4)--(a5)--(a6)--(a7)--(a8)--(aa)--(ab)
(a0)--(a3);
\draw[dotted] (ab)--(ac);
\end{tikzpicture}}

\bigskip
\begin{lemma}\label{lem:main}
{\rm(1)} \ If $\mathbf m$ is of $(2,3,*)$-type, then $\mathcal S\mathbf m\underset{\mathrm{mc}}\sim\mathbf m$.

\smallskip
\noindent
{\rm(2)} \ 
If $N_p(\mathbf m)=3$, 
then $\mathcal S^2\mathbf m\underset{\mathrm{mc}}\sim\mathcal S\mathbf m$. 

\smallskip
\noindent
{\rm(3)} \ If $\mathbf m'$ is the unfolding of a Harnad dual of a confluence of $\mathbf m$,
then 
\ $\mathcal S^2\mathbf m'\underset{\mathrm{mc}}\sim\mathcal S^2\mathbf m$. 

\smallskip
\noindent
{\rm(4)} \ $\ord\mathcal R\mathcal S\mathbf m>\ord\mathbf m\quad
(\forall\mathbf m\in\mathcal F\setminus\mathcal F^{E_8})$
\end{lemma}

The proof of Lemma~\ref{lem:main} (1), (2) and (3) is given in this section.
In the next section, we show a second main theorem which calculates 
$\mathcal R\mathcal S\mathbf m$ for $\mathbf m\in\mathcal F$ and Proposition~\ref{cor:RSm}
says Lemma~\ref{lem:main} (4).

\medskip
Let \[\mathbf m=m_{11}m_{12},\,m_{21}m_{22}m_{23},\,m_{31}m_{32}\cdots m_{3n_3}\]
be an ordered $(2,3,*)$ spectral type.
Then Katz reductions of $\mathcal S\mathbf m$ are given by
\begin{align*}
\mathcal S\mathbf m&=
\underline{(2n)}n,\,\underline{n}nn,\,
 \underline{m_{11}}m_{12}m_{21}m_{22}m_{23}m_{31}m_{32}\cdots\\
&\qquad\qquad(d_{1,1,1}(\mathbf m)=n+m_{11}-n=m_{11})
\allowdisplaybreaks\\
&\xrightarrow[1,1,1]{m_{11}} \underline{(n+m_{12})}n,\,m_{12}\underline{n}n,\,
 0\underline{m_{12}}m_{21}m_{22}m_{23}m_{31}m_{32}\cdots
\allowdisplaybreaks\\
&\xrightarrow[1,2,2]{m_{12}} \underline{n}n,\,m_{12}m_{11}\underline{n},\,
 00\underline{m_{21}}m_{22}m_{23}m_{31}m_{32}\cdots
\allowdisplaybreaks\\
&\xrightarrow[1,3,3]{m_{21}}(m_{22}+m_{23})\underline{n},\,m_{12}m_{11}\underline{(m_{22}+m_{23})},\,
 000\underline{m_{22}}m_{23}m_{31}\cdots
\allowdisplaybreaks\\
&\xrightarrow[2,3,4]{m_{22}}(m_{22}+m_{23})\underline{(m_{21}+m_{23})},\,
    m_{12}\underline{m_{11}}m_{23},\,0000\underline{m_{23}}m_{31}m_{32}\cdots
\allowdisplaybreaks\\
&\xrightarrow[2,2,5]{m_{11}-m_{22}}\underline{(m_{22}+m_{23})}m_{12},\,
    \underline{m_{12}}m_{22}m_{23},\,0000\underline{(m_{12}-m_{21})}m_{31}m_{32}\cdots\\
&\xrightarrow[1,1,5]{m_{12}-m_{21}}m_{11}m_{12},\,m_{21}m_{22}m_{23},00000\,m_{31}m_{32}\cdots
m_{3,n_3}
\simeq\mathbf m,
\end{align*}
which proves Lemma~\ref{lem:main} (1).
Here, the underlined multiplicities such as $\underline{(2n)}$ indicate the $\sigma$ in 
$\mc_\sigma$ to be considered.  
They are mostly maximal multiplicities in partitions.

\bigskip
Considering the following Katz reductions
\begin{align*}
&\mathcal S^2\mathbf m=\underline{(2pn)}(pn),(pn)(pn)\underline{(pn)},\underline{((p-1)n)}
 \overbrace{n\cdots n}^{p+1}m_{11}\cdots m_{p1}\cdots
\allowdisplaybreaks\\
%&\qquad (2pn+pn+(p-1)n)-3pn=(p-1)n\\
&\xrightarrow[1,3,1]{(p-1)n}
(\underline{(p+1)n})(pn),(pn)(\underline{pn})n,
 0\overbrace{\underline{n}\cdots n}^{p+1}m_{11}\cdots m_{p1}\cdots
\allowdisplaybreaks\\
%&\qquad ((p+1)n+pn+n)-(2p+1)n=n\\
&\xrightarrow[1,2,2]{n}
 (pn)\underline{(pn)},\underline{(pn)}((p-1)n)n,
 00\overbrace{\underline{n}\cdots n}^{p}m_{11}\cdots m_{p1}\cdots
\allowdisplaybreaks\\
&\cdots\cdots\quad(\nu=1,2,\ldots)\\%[-3mm]
%&\overset{\mc}\to
&\xrightarrow[2,1,2\nu+1]{n}
\underline{((p-\nu+1)n)}((p-\nu)n),((p-\nu)n)\underline{((p-\nu)n)}n,
\overbrace{0\cdots0}^{2\nu+1}\overbrace{\underline{n}\cdots n}^{p-2\nu+1}
 m_{11}\cdots\\% m_{p1}\cdots\\
%&\overset{\mc}\to((p-\nu)n)
&\xrightarrow[1,2,2\nu+2]{n}
((p-\nu)n)\underline{((p-\nu)n)},\underline{((p-\nu)n)}((p-\nu-1)n)n,
 \overbrace{0\cdots0}^{2\nu+2}\overbrace{\underline{n}\cdots n}^{p-2\nu}
 m_{11}\cdots,\\% m_{p1}\cdots,\\
\end{align*}
we make the following definition.
\begin{definition}\label{def:He}
Set
\begin{align*}
 \He\mathbf m:=
\begin{cases}
 \bigl[[kn,(k-1)n],[(k-1)n,(k-1)n,n],[m_{11},\cdots,m_{pn_p}]\bigr] &(p=2k-1),\\
 \bigl[[kn,kn],[kn,(k-1)n,n],[m_{11},\cdots,m_{pn_p}]\bigr] &(p=2k).
\end{cases}
\end{align*}
\end{definition}
By the above reductions, we have
\begin{equation*}
\He\mathbf m\underset{\mc}\sim\Hu^2\mathbf m,
\end{equation*}
which proves Lemma~\ref{lem:main} (2).

\begin{proof}[Proof of Lemma~\ref{lem:main} (3)]
Let $\mathbf m\in\mathcal P$ be written as in \eqref{eq:m}.
Suppose that $m_{p-1,1}\cdots m_{p-1,n_{p-1}}$ is a refinement of $m_{p,1}\cdots m_{p,n_p}$.
Write  
\begin{align*}\begin{split}
&\{m_{p-1,\nu}\mid 1\le\nu\le n_{p-1}\}=\{x_{j,\nu}\mid 1\le j\le n_p,\ 1\le\nu\le R_j\},\\
&n_{p-1}=R_1+\cdots+R_{n_p},\ m_{p,j}=x_{j,1}+\cdots+x_{j,R_j}\quad(j=1,\dots,n_p).
\end{split}\end{align*}
Moreover, set
\begin{align*}\begin{split}
 k&:=n_p,\qquad a_\nu:=m_{p,\nu}\quad(1\le\nu\le n_p),\\
 \ell&:=p-2,\quad b_j:=m_{j,1},\quad b_{j,\nu}:=m_{j,\nu}\quad(j=1,\dots,\ell,\ 2\le\nu\le n_j),\\
 \bar a&=a_1+\cdots+a_k=\ord\mathbf m,\quad \bar b=b_1+\cdots+b_\ell.
\end{split}\end{align*}

We consider the confluence of the points $c_{p-1}$ and $c_p=\infty$ corresponding to the refinement
and denote by $\mathrm{Hu}\,\mathbf m$ the unfolding of the Harnad dual of this confluence.
We have
\begin{align*}
\mathbf m&=m_{1,1}\cdots m_{1,n_1},\ \ldots\,,\ m_{p-2,1}\cdots,
\,m_{p-1,1}\cdots m_{p-1,n_{p-1}},\,m_{p,1}\cdots m_{p,n_p}\\
&=b_1b_{1,2}\cdots b_{1,n_1},\,\ldots\,,\ b_\ell b_{\ell,2}\cdots b_{\ell,n_\ell},\,\overbrace{x_{1,1}\cdots x_{1,R_1}}^{a_1}\,\cdots\,\overbrace{x_{k,1}\cdots x_{k,R_k}}^{a_k},\ a_1\cdots a_k\\
\overset{\text{cf}}\to
&\ \ (x_{1,1}\cdots x_{1,R_1})\cdots (x_{k,1}\cdots x_{k,R_k}),\ \underline{b_1}b_{1,2}\cdots b_{1,n_1},\ \ldots\,,\ 
\underline{b_\ell} b_{\ell,2}\cdots b_{\ell,n_\ell}\\
\overset{\mathrm{Hd}}\to&
\ \ (b_{1,2}\cdots b_{1,n_1})\cdots(b_{\ell,2}\dots b_{\ell,n_\ell}),\,
(\ell\bar a-\bar b-a_1)x_{1,1}\cdots x_{1,R_1},\ \ldots,\ 
(\ell\bar a-\bar b-a_k)x_{k,1}\cdots x_{k,R_k}\\
\overset{\mathrm{uf}}\to
&\ \ b_{1,2}\cdots b_{1,n_1}\cdots b_{\ell,2}\cdots b_{\ell,n_\ell},\,
(\bar a-b_1)\cdots(\bar a - b_\ell),\ 
(\ell\bar a-\bar b-a_1)x_{1,1}\cdots x_{1,R_1},\\ 
&\qquad \ldots,\ 
(\ell\bar a-\bar b-a_k)x_{k,1}\cdots x_{k,R_k}.
\end{align*}
Here we have
\begin{align*}
 \ord\mathrm{Hu}\,\mathbf m&=
 (b_{1,2}+\cdots+b_{1,n_1})+\cdots+(b_{\ell,2}+\cdots+b_{\ell,n_\ell})=\ell\bar a-\bar b.
\end{align*}

Putting
\begin{align*}\begin{split}
&x_i=\{x_{i,1},\dots,x_{i.R_i}\},\  y_j=\{b_{j,2},\dots y_{j,n_j}\},\ x=\{x_1\dots x_k, y_1\dots y_\ell\},
\end{split}\end{align*}
we have 
\begin{align*}\begin{split}
&\mathbf m=\bigl[[b_1,y_1],\ldots,[b_\ell,y_\ell],[x_1,\dots,x_k],[a_1\ldots,a_k]\bigr]\\
&\overset{\mathrm{Hu}}\to\bigl[\bar a-b_1,\ldots,\bar a-b_\ell],
  [\ell\bar a-\bar b-a_1,x_1],\ldots,[\ell\bar a-\bar b-a_k,x_k],[y_1,\dots,y_\ell]\bigr].
\end{split}\end{align*}

%{\color{gray}$\ell+2$ points, ord : $\bar a$ $\overset{\mathrm{Hu}}\to$
%$k+2$ points, ord : $\ell\bar a-\bar b,\ $}

\medskip
Suppose $k$ is divisible by 4. Set $k+2=2q$. Then
\begin{align*}
 &\mathcal T \,\mathrm{Hu}\,\mathbf m\\
 &=\bigl[[\underline{q(\ell\bar a-\bar b)},\ q(\ell\bar a-\bar b)],\\
 &\qquad[\underline{q(\ell\bar a-\bar b)},\ (q-1)(\ell\bar a-\bar b),\ \ell\bar a-\bar b],\\
 &\qquad\quad[\bar a-b_1,\dots,\bar a-b_\ell,\ell\bar a-\bar b-a_1,\cdots,\underline{\ell\bar a-\bar b-a_k},x]\bigr]
\allowdisplaybreaks\\
 &\xrightarrow[1,1,k+\ell]{\ell\bar a-\bar b-a_k}\\
 &\bigl[[(q-1)(\ell\bar a-\bar b)+a_k,\ \underline{q(\ell\bar a-\bar b)}],\\
 &\quad[\underline{(q-1)(\ell\bar a-\bar b)+a_k},\ (q-1)(\ell\bar a-\bar b),\ \ell\bar a-\bar b],\\
 &\qquad[\bar a-b_1,\dots,\bar a-b_\ell,\ell\bar a-\bar b-a_1,\cdots,\underline{\ell\bar a-\bar b-a_  {k-1}},0,x]
\allowdisplaybreaks\\
 &\xrightarrow[2,1,k+\ell-1]{\ell\bar a-\bar b-a_{k-1}}\\
 &\bigl[[(q-1)(\ell\bar a-\bar b)+a_k,\ \underline{(q-1)(\ell\bar a-\bar b)+a_{k-1}}],\\
 &\quad[(q-2)(\ell\bar a-\bar b)+a_{k-1}+a_k,\ \underline{(q-1)(\ell\bar a-\bar b)},
  \ \ell\bar a-\bar b],\\
  &\qquad[[\bar a-b_1,\dots,\bar a-b_\ell,\ell\bar a-\bar b-a_1,\ldots,
  \underline{\ell\bar a-\bar b-a_{k-2}},0,0,x]]
\allowdisplaybreaks\\
  &\xrightarrow[2,2,k+\ell-2]{\ell\bar a-\bar b-a_{k-2}-a_{k}}\\
 &\bigl[[\underline{(q-1)(\ell\bar a-\bar b)+a_k},(q-2)(\ell\bar a-\bar b)+a_{k-2}+a_{k-1}+a_k],\\
 &\quad[(q-2)(\ell\bar a-\bar b)+a_{k-1}+a_k,\,\underline{(q-2)(\ell\bar a-\bar b)+a_{k-2}+a_k},
  \,\ell\bar a-\bar b],\\
  &\qquad\bigl[[\bar a-b_1,\dots,\bar a-b_\ell,\ell\bar a-\bar b-a_1,\ldots,
  \underline{\ell\bar a-\bar b-a_{k-3}},a_k,0,0,x]\bigr]
\allowdisplaybreaks\\
  &\xrightarrow[1,2,k+\ell-3]{\ell\bar a-\bar b-a_{k-3}-a_{k-1}}\\
%%%%%%%%%%%%%%%
  &\bigl[[\underline{(q-2)(\ell\bar a-\bar b)+a_{k-3}+a_{k-1}+a_k},\ 
  (q-2)(\ell\bar a-\bar b)+a_{k-2}+a_{k-1}+a_k],\\
 &\quad[\underline{(q-2)(\ell\bar a-\bar b)+a_{k-1}+a_k},\ 
   (q-3)(\ell\bar a-\bar b)+a_{k-3}+\cdots+a_k,\ 
  \,\ell\bar a-\bar b],\\
  &\qquad\bigl[[\bar a-b_1,\dots,\bar a-b_\ell,\ell\bar a-\bar b-a_1,\ldots,
  \underline{\ell\bar a-\bar b-a_{k-4}},a_{k-1},a_k,0,0,x]\bigr]\\
  &\xrightarrow[1,1,k+\ell-4]{\ell\bar a-\bar b-a_{k-4}-a_{k-2}}\\
%%%%
  &\cdots\\
  &\xrightarrow[1,2,\ell+1]{\ell\bar a-\bar b-a_1-a_3}\\
  &\bigl[[\ell\bar a-\bar b+a_1+a_3+\cdots+a_k,\ 
   \ell\bar a-\bar b+a_2+a_3+\cdots+a_k],\\
  &\quad[\ell\bar a-\bar b+a_3+\cdots+a_k,\ 
        a_1+a_2+\cdots+a_k,\,
        \ell\bar a-\bar b],\\
  &\qquad[\bar a-b_1,\dots,\bar a-b_\ell,a_3,a_4,\ldots,a_k,0,0,x]
  \bigr]\allowdisplaybreaks\\
%%%
  &=\bigl[[\underline{(\ell+1)\bar a-\bar b-a_2}, 
   (\ell+1)\bar a-\bar b-a_1],\\
  &\qquad[\underline{(\ell+1)\bar a-\bar b-a_1-a_2},\ 
        \bar a,\,
        \ell\bar a-\bar b],\\
  &\qquad\quad[\bar a-b_1,\dots,\underline{\bar a-b_\ell},a_3,a_4,\ldots,a_k,0,0,x]
  \bigr]\allowdisplaybreaks\\
%%%
  &\xrightarrow[1,1,\ell]{\bar a-a_2-b_\ell}\\
  &\bigl[[\ell\bar a-\bar b+b_\ell,\ 
  \underline{(\ell+1)\bar a-\bar b-a_1}],\\
  &\quad[\ell\bar a-\bar b+b_\ell-a_1,\ 
        \bar a,\ 
        \underline{\ell\bar a-\bar b}],\\
  &\qquad[\bar a-b_1,\dots,\underline{\bar a-b_{\ell-1}},a_2,a_3,a_4,\ldots,a_k,0,0,x]
  \bigr]\allowdisplaybreaks\\
%%%
  &\xrightarrow[2,3,\ell-1]{\bar a-b_{\ell-1}-b_\ell}\\
  &\bigl[[\underline{\ell\bar a-\bar b+b_\ell},\ 
    \ell\bar a-\bar b-a_1+b_{\ell-1}+b_{\ell}],\\
  &\quad[\underline{\ell\bar a-\bar b+b_\ell-a_1},\ 
        \bar a,\ 
        (\ell-1)\bar a-\bar b+b_{\ell-1}+b_\ell],\\
  &\qquad[\bar a-b_1,\dots,\underline{\bar a-b_{\ell-2}},b_\ell,
   a_2,a_3,a_4,\ldots,a_k,0,0,x]
  \bigr]\allowdisplaybreaks\\
%%%  
  &\xrightarrow[1,1,\ell-2]{\bar a-b_{\ell-2}-b_{\ell-1}}\\
  &\bigl[[(\ell-1)\bar a-\bar b+b_{\ell-2}+b_{\ell-1}+b_\ell,\ 
    \underline{\ell\bar a-\bar b-a_1+b_{\ell-1}+b_{\ell}}],\\
  &\quad[(\ell-1)\bar a-\bar b+b_{\ell-2}+b_{\ell-1}+b_\ell-a_1,\ 
        \bar a,\ 
        \underline{(\ell-1)\bar a-\bar b+b_{\ell-1}+b_\ell}],\\
  &\qquad[\bar a-b_1,\dots,\underline{\bar a-b_{\ell-3}},
   b_{\ell-1},b_\ell,a_2,a_3,a_4,\ldots,a_k,0,0,x]
  \bigr]\allowdisplaybreaks\\
%%%
  &\xrightarrow[2,3,\ell-3]{\bar a-b_{\ell-3}-b_{\ell-2}}\\
  &\bigl[[\underline{(\ell-1)\bar a-\bar b+b_{\ell-2}+b_{\ell-1}+b_\ell},\ 
    (\ell-1)\bar a-\bar b-a_1+b_{\ell-3}+\cdots+b_{\ell}],\\
  &\quad[\underline{(\ell-1)\bar a-\bar b+b_{\ell-2}+b_{\ell-1}+b_\ell-a_1},\ 
        \bar a,\ 
        (\ell-2)\bar a-\bar b+b_{\ell-3}+\cdots+b_\ell],\\
  &\qquad[\bar a-b_1,\dots,\underline{\bar a-b_{\ell-4}},b_{\ell-2},
   b_{\ell-1},b_\ell,a_2,a_3,a_4,\ldots,a_k,0,0,x]
  \bigr]\\
  &\cdots\allowdisplaybreaks
\intertext{When $\ell=2r-2$ with a positive integer $r$, we have $p=2r$ and eventually}
  &\xrightarrow[2,3,1]{\bar a-b_1-b_2}\\
  &\bigl[[r\bar a-\bar b+b_2+\cdots+b_\ell,\ 
          r\bar a-\bar b-a_1+b_1+\cdots+b_\ell],\\
  &\quad[r\bar a-\bar b+b_2+\cdots+b_\ell-a_1,\ 
        \bar a,\ 
        (r-1)\bar a-\bar b+b_1+\cdots+b_\ell],\\
  &\qquad[b_2,b_3,\ldots,b_\ell,a_2,a_3,a_4,\ldots,a_k,0,0,x]
  \bigr]\allowdisplaybreaks\\
  &=\bigl[[r\bar a-b_1,r\bar a-a_1],\,[r\bar a-a_1-b_1,\bar a,(r-1)\bar a],\,
   [b_2,b_3,\ldots,b_\ell,a_2,a_3,a_4,\ldots,a_k,0,0,x]
  \bigr]\allowdisplaybreaks\\
  &\xrightarrow[1,1,k+\ell-1]{-b_1}\ \ 
  \xrightarrow[2,1,k+\ell]{-a_1}\\
  &\bigl[[r\bar a,\ r\bar a],\ 
  [r\bar a,\ 
         \bar a,\ 
        (r-1)\bar a],\ 
  [b_2,b_3,\ldots,b_\ell,a_2,a_3,a_4,\ldots,a_k,b_1,a_1,x]\bigr].
%  &\xrightarrow[2,3,\ell+1]{a_2-b_1}\\
\allowdisplaybreaks
\intertext{When $\ell=2r-1$ with a positive integer $r$, we have $p=2r+1$ and}
  &\xrightarrow[1,1,1]{\bar a - b_1-b_2}\\
\if 0
  &\bigl[[r\bar a,\ 
    (r+1)\bar a-a_1-b_1],\\
  &\quad[r\bar a - a_1,\ 
        \bar a,\ 
        r\bar a - b_1],\\
  &\qquad[b_2,b_3,\ldots,b_\ell,a_2,a_3,\ldots,a_k,0,0,x]
\fi
  &\bigl[[r\bar a,\,
    (r+1)\bar a-a_1-b_1],[r\bar a - a_1,\bar a,r\bar a - b_1],\,
  [b_2,b_3,\ldots,b_\ell,a_2,a_3,\ldots,a_k,0,0,x]
  \bigr]\allowdisplaybreaks\\
%%%
  &\xrightarrow[2,3,k+\ell-1]{-b_1}\ \ \xrightarrow[2,1,k+\ell]{-a_1}\\
  &\bigl[[r\bar a,\ 
    (r+1)\bar a],\ 
  [r\bar a,\ 
        \bar a,\ 
        r\bar a],\ 
  [b_2,b_3,\ldots,b_\ell,a_2,a_3,\ldots,a_k,b_1,a_1,x]
  \bigr].
\end{align*}
Hence 
\begin{equation}\label{eq:SHd}
\He\mathrm{Hu}\,\mathbf m\underset{\mc}\sim \He \mathbf m.
\end{equation}

For a general $k$, we replace $m_{p-1,1}\cdots m_{p-1,n_{p-1}}$ and 
$a_1\cdots a_k$ by $m_{p-1,1}\cdots m_{p-1,n_{p-1}}\overbrace{0\cdots0}^{k'}$
and $a_1\cdots a_k\overbrace{0\cdots0}^{k'}$,
respectively, where $k'$ is chosen so that $k+k'$ is divisible by $4$.
Thus, \eqref{eq:SHd} is proved for arbitrary $k$.
This proves Lemma~\ref{lem:main}~(3).
\end{proof}
%%%%%%%%%%%%%%%%%%%%%%%%%%%%%%%%%%%%%%%%%%%%%%%%%%%
%%%%%%%%%%%%%%%%   3.3  %%%%%%%%%%%%%%%%%%%%%%%%%%%
\subsection{Relation to middle convolutions}
%%%%%%%%%%%%%%%%%%%%%%%%%%%%%%%%%%%%%%%%%%%%%%%%%%%
Since $\mc_\sigma$ is a composition of two Harnad duals, as is shown in 
Remark~\ref{remark:Hd}~(3), 
$\mathcal T\mc_\sigma\mathbf m\underset{\mc}\sim \mathcal T\mathbf m$.
We investigate the direct relation between
$\mathcal T\mc_\sigma\mathbf m$ and $\mathcal T\mathbf m$.

For a positive integer $p$, we define the $(2,3,*)$-spectral type
\begin{equation*}
 \mathbf n_p:=\begin{cases}
               \bigl[[k,k],[k,k-1,1],[\overbrace{1,\dots,1}^p]\bigr]&(p=2k),\\
               \bigl[[k,k-1],[k-1,k-1,1],[\overbrace{1,\dots,1}^p]\bigr]&(p=2k-1).
              \end{cases}
\end{equation*}
Applying successive the middle convolutions to $\mathbf n_p$, we obtain
\begin{align*}
\mathbf n_{2k}&=k\underline{k},\,\underline{k}(k{-}1)1,\,\overbrace{1\cdots 1\underline{1}}^{2k}
\xrightarrow[2,1,2k]{1}
  \underline{k}(k{-}1),\,(k{-}1)(\underline{k{-}1})1,\,
  \overbrace{1\cdots 1\underline{1}}^{2k-1}0\simeq \mathbf n_{2k-1}\\
&\xrightarrow[1,2,2k-1]{1}
  (k{-}1)(\underline{k{-}1}),\,(\underline{k{-}1})(k{-}2)1,\,\overbrace{1\cdots 1\underline{1}}^{2k-2}00\simeq\mathbf n_{2k-2}\xrightarrow[2,1,2k-3]{1}\cdots.
\end{align*}
Hence, $\mathbf n_p\underset{\mc}\sim\mathbf n_1$, and $\mathbf n_p$ is a rigid spectral type.

The corresponding real root is
\begin{align*}
\beta_p&:=p\alpha_0+[\tfrac {p{+}1}2]\alpha_{11}+[\tfrac p2]\alpha_{21}+\alpha_{22}
    +\sum\limits_{\nu=1}^{p-1}(p-\nu)\alpha_{3\nu}, 
\ \ \raisebox{-6mm}{\scalebox{0.9}{\begin{tikzpicture}
 [root/.style={draw,circle,inner sep=0mm,minimum size=2mm}]
\node[root] (a1) at (0,0) {}; 
\node[root] (a2) at (0.75,0) {}; 
\node[root] (a3) at (1.5,0) {}; 
\node[root] (a4) at (2.25,0) {}; 
\node[root] (a5) at (3.35,0) {}; 
\node[root] (a6) at (4.1,0) {}; 
\node[root] (a0) at (1.5,0.75) {}; 
\node at (0,-0.3) {1};
\node at (0.75,-0.4) {$[\frac {p{+}1}2]$};
\node at (1.5,-0.4) {$p$};
\node at (2.25,-0.4) {$p{-}1$};
\node at (3.35,-0.4) {$2$};
\node at (4.1,-0.4) {1};
\node at (1.05,0.75) {$[\frac p2]$};
\draw (a1)--(a2)--(a3)--(a4)  (a5)--(a6) (a0)--(a3);
\draw[densely dotted] (a4)--(a5); % (a7)--(a8);
\end{tikzpicture}}}
\end{align*}
Here, we note that
\begin{align*}
(\beta_{2k}|\alpha_0)&=1,& (\beta_{2k-1}|\alpha_0)&=1,
\\
(\beta_{2k}|\alpha_{11})&=0,\quad(\beta_{2k}|\alpha_{12})=-k,
 & (\beta_{2k-1}|\alpha_{11})&=-1,\quad(\beta_{2k-1}|\alpha_{12})=1-k,
\\
(\beta_{2k}|\alpha_{21})&=-1,\quad(\beta_{2k}|\alpha_{22})=2-k,
& (\beta_{2k-1}|\alpha_{21})&=0,\quad(\beta_{2k-1}|\alpha_{22})=2-k,
\\
%(\beta_{2k}|\alpha_{22})&=2-k,& (\beta_{2k-1}|\alpha_{22})&=2-k,
%\\
(\beta_p|\alpha_{23})&=-1,&% (\beta_{2k-1}|\alpha_{23})&=-1,
%\\
(\beta_{p}|\alpha_{3j})&=
  \begin{cases}
    0&(j\ge 1,\ j\ne p),\\
    -1&(j=p).
  \end{cases}
\end{align*}
Set
\begin{align*}
\mathbf m_3&=[m_{11},m_{21},\ldots,m_{p1},m_{12},\dots,m_{1n_1}%,m_{22},\dots.m_{2,n_2}
 ,\,\dots\,,m_{p2},\dots,m_{pn_p}],\\
\mathcal\He'\mathbf m&:=
 \begin{cases}
 \bigl[[kn,(k-1)n],[(k-1)n,(k-1)n,n],\mathbf m_3\bigr]&(p=2k-1),\\
 \bigl[[kn,kn],[kn,(k-1)n,n],\mathbf m_3\bigr]&(p=2k),
 \end{cases}\\
\alpha_{\mathcal\He'\mathbf m}&=n'_0
\alpha_0+n'_{11}\alpha_{11}+n'_{21}\alpha_{21}+n'_{22}\alpha_{22}+\sum_{\nu\ge1}n'_{3\nu}
\alpha_{3\nu}
\intertext{for $\mathbf m$ given by \eqref{eq:m}. We have 
$\mathcal T\mathbf m\simeq\mathcal T'\mathbf m$ and}
 n'_0&=pn,\ \ n'_{11}=[\tfrac p2]n,\ \ n'_{21}=[\tfrac{p+1}2]n,\ \ n'_{22}=n,\\
 n'_{3\ell}&=pn-\sum_{j=1}^\ell m_{j1}=pn-(m_{11}+\cdots+m_{\ell1})\qquad(1\le\ell\le p).
\end{align*}

When $p=2k$, we have
\begin{align*}
(\beta_p|\alpha_{\mathcal\He'\mathbf m})&=
 (\beta_p|pn\alpha_0+kn\alpha_{11}+kn\alpha_{21}+n\alpha_{22}+\sum_{\nu\ge1}n'_{3\nu}\alpha_{3\nu})\\
 &=pn+0-kn+(2-k)n-(pn-(m_{11}+\cdots+m_{p1}))\\
 &=m_{11}+\cdots+m_{p1}-(p-2)n\\
 &=d_{1,\dots,1}(\mathbf m). 
\end{align*}
Since $\mc_{(1,\dots,1)}$ maps $m_{j\nu}$ to $m_{j\nu}-d_{1,\dots,1}(\mathbf m)\delta_{\nu,1}$
with $D:=d_{1,\dots,1}(\mathbf m)$, 
we have
\begin{align*}
 \alpha_{\mathcal\He'\mc_{(1,\dots,1)}\mathbf m}
  &=(n'_0-pD)\alpha_0+(n'_{11}-kD)\alpha_{11}+(n'_{21}-D)\alpha_{21}\\
  &\quad{}+
  \sum_{\nu=1}^{p-1} \bigl((p-\nu)D+n'_{3\nu}\bigr)\alpha_{3\nu}+\sum_{\nu\ge p}n'_{3\nu}\alpha_{3\nu}\\
  &=\alpha_{\mathcal\He'\mathbf m}-D\beta_p=s_{\beta_p}\alpha_{\mathcal\He'\mathbf m}.
\end{align*}

When $p=2k-1$, owing to the relation
\begin{align*}
(\beta_p|\alpha_{\mathcal\He'\mathbf m})&=
 (\beta_p|pn\alpha_0+(k-1)n\alpha_{11}+kn\alpha_{21}+n\alpha_{22}+\sum_{\nu\ge1}n'_{3\nu}\alpha_{3\nu})\\
 &=pn+(1-k)n+0+(2-k)n-(pn-(m_{11}+\cdots+m_{p1}))\\
% &=m_{11}+\cdots+m_{p1}-(p-2)n\\
 &=d_{1,\dots,1}(\mathbf m), 
\end{align*}
we have the same result as in the case $p=2k$.  Thus, the following theorem is proved. 
\begin{theorem}  Under the notation above, we have
\begin{align*}
 \alpha_{\He'\mc_{(1,\dots,1)}\mathbf m}=s_{\beta_p}\alpha_{\He'\mathbf m}.
\end{align*}
\end{theorem}

We note that the reflection $s_{\beta_p}$ with  respect to the real root $\beta_p$ is 
realized by a sequence of reflections with respect to simple roots.

%%%%%%%%%%%%%%%%%%%%%%%%%%%%%%%%%%%%%%%%%%%%%%%%%%%%%%%%
%%%%%%%%%%%%%%%%%%%%%%   4   %%%%%%%%%%%%%%%%%%%%%%%%%%%
\section{Correspondences to complete representatives}
%%%%%%%%%%%%%%%%%%%%%%   4.1  %%%%%%%%%%%%%%%%%%%%%%%%%%
\subsection{Another main theorem}
\label{sec:RSm}
%%%%%%%%%%%%%%%%%%%%%%%%%%%%%%%%%%%%%%%%%%%%%%%%%%%%%%%%
For a given $\mathbf m=m_{11}\cdots m_{1,n_1},\ldots,m_{p,1}\cdots m_{p,n_p}\in\mathcal F$,
we calculate $\mathcal R\mathcal S\mathbf m\in\mathcal F$. 
Note that $\mathbf n\in\mathcal F^{E_8}$ satisfying $\mathbf n\sim\mathbf m$ 
is given by $\mathcal R\mathcal S(\mathcal R\mathcal S\mathbf m)$.
As mentioned in Remark~\ref{remark:defS} (2), $\mathcal S\mathbf m$ is determined by the set 
$\{m_{j,\nu}\}$. 
Hence, we arrange the elements $m_{j,\nu}$ in nonincreasing order and set
\begin{align}\label{eq:an}
 \begin{split}
 \{m_{j,\nu}\mid j=1,\dots,p,\ \nu=1,\dots,n_j\}=\{a_1,\dots a_{\tilde n}\},\\
 a_1\ge a_2\ge\cdots\ge a_{\tilde n},\ \tilde n=n_1+\cdots+n_p. 
 \end{split}
\end{align}
We introduce the following notation in order to state the second main
theorem of this paper.
%%%%%%%%%%%%%%%%%%%%%%%%%%%%%%%%%%
%%%  definition of a_1, a_2,....
%%%%%%%%%%%%%%%%%%%%%%%%%%%%%%%%%%
\begin{definition}\label{def:DpN}
Let $a_1,a_2,\ldots,a_{\tilde n}$ be the sequence of integers \eqref{eq:an} for a given 
$\mathbf m\in\mathcal P$.
Set
\begin{align*}
 a_{[k,\ell]}&:=\begin{cases}
   a_k+a_{k+1}+\cdots+a_{\ell}&(k\le\ell),\\
   0&(k>\ell),
 \end{cases}\\
 D_p&:=\sum_{\nu=1}^{p-2}(n-a_\nu)=(p-2)n-a_{[1,p-2]},\\
 n&:=\ord\mathbf m.\notag
\end{align*}
Then we denote $N_{\mathbf m}$ the maximal positive integer satisfying
\begin{equation}
 a_{p+N-2}+a_{p+N}>D_p.
\tag*{$(D_p:N)$}
\end{equation}
If $(D_p:1)$ does not hold, we set $N_{\mathbf m}=0$.
\end{definition}
%%%%%%%%%%%%%%%%%%%%%%%%%%%%%
\begin{remark}
If $(D_p:N)$ holds and $0<N'<N$, then $(D_p:N')$ also holds. 
\end{remark}
%%%%%%%%%%%%%%%%%%%%%%%%%%  Main Theorem %%%%%%%%%%%%%%%%%%
\begin{theorem}\label{thm:RSm}
We write a relation $m''=\mc_{(\sigma_1,\sigma_2,\ldots)}\mathbf m'$ with
$c=d_{\sigma_1,\sigma_2,\ldots}(\mathbf m')$ in the form
\begin{align*}
 \mathbf m'\,\xrightarrow[\sigma_1,\sigma_2,\ldots]{c}\mathbf m'' \text{ \ or \ }
 \mathbf m'\,\xrightarrow[\sigma_1,\sigma_2,\ldots]{f>g}\mathbf m''\ 
\text{\ \ with }c=f-g, 
\end{align*}
which represents a Katz reduction if $c>0$ or $f>g$. 
Then a sequence of Katz reductions of  $\mathcal S\mathbf m$ is 
given by 
\begin{align*}
\mathcal S\mathbf m&
 \xrightarrow[1,1,1]{a_1}\bullet\xrightarrow[1,2,2]{a_2}% \bullet \xrightarrow[1,3,3]{a_3}
\bullet\to\cdots\to\bullet \xrightarrow[1,p,p]{a_p}\mathbf m^{(0)}
\xrightarrow[2,p,p+1]{a_{[1,p-1]}+a_{p+1}>(p-2)n}\mathbf m^{(1)}
\allowdisplaybreaks\\
&\xrightarrow[2,p-1,p+2]{a_{[1,p-2]}+a_p+a_{p+2}>(p-2)n}\mathbf m^{(2)}
\begin{cases}
\xrightarrow[1,p-2,p+2]{a_{[1,p-3]}+a_{p-1}+a_{p+1}+a_{p+2}>(p-2)n}\mathbf m^{(3,1)},\\
\xrightarrow[2,p-2,p+3]{a_{[1,p-3]}+a_{p-1}+a_p+a_{p+3}>(p-2)n}\mathbf m^{(3,2)},\\
\xrightarrow[1,p-1,p+3]{a_{[1,p-2]}+a_{p+1}+a_{p+3}>(p-2)n}\mathbf m^{(3)}\to
\end{cases}\\
&\to\ \overset{k=1,2,\ldots}\cdots\to \mathbf m^{(4k-1)}
\allowdisplaybreaks\\
&\xrightarrow[1,p,p+4k]{a_{[1,p-2]}+a_{p+4k-2}+a_{p+4k}>(p-2)n}\mathbf m^{(4k)}
\xrightarrow[2,p,p+4k+1]{a_{[1,p-2]}+a_{p+4k+1}>(p-2)n}\mathbf m^{(4k+1)}
\allowdisplaybreaks\\
&\xrightarrow[2,p-1,p+4k+2]{a_{[1,p-2]}+a_{p+4k}+a_{p+4k+2}>(p-2)n}\mathbf m^{(4k+2)}
\xrightarrow[1,p-1,p+4k+3]{a_{[1,p-2]}+a_{p+4k+1}+a_{p+4k+3}>(p-2)n}\mathbf m^{(4k+3)}
\xrightarrow{k\mapsto k+1}%\raisebox{2mm}{$\uparrow$}
\end{align*}
and
\[
\mathcal R\mathcal S\mathbf m=
\begin{cases}
\mathbf m^{(N_{\mathbf m})}&(N_{\mathbf m}\ne 2),\\
\mathbf m^{(2)}&(N_{\mathbf m}=2\text{ and  \eqref{cond:2}}),\\
\mathbf m^{(3,1)}&(N_{\mathbf m}=2\text{ and \eqref{cond:3-1}}),\\
\mathbf m^{(3,2)}&(N_{\mathbf m}=2\text{ and \eqref{cond:3-2}}).
\end{cases}
\]
Here, 
\begin{align}
  a_{[1,p-3]}+a_{p-1}+\max\{a_{p+1}+a_{p+2}, a_p+a_{p+3}\}
  &\le (p-2)n \label{cond:2},\\
 a_{[1,p-3]}+a_{p-1}+a_{p+1}+a_{p+2}&>(p-2)n\tag{C:3.1}\label{cond:3-1},\\
 a_{[1,p-3]}+a_{p-1}+a_p+a_{p+3}&>(p-2)n\tag{C:3.2}\label{cond:3-2}
\end{align}
and for $k\in\mathbb Z_{\ge0}$,
\begin{align*}
\mathbf m^{(4k)}&=\bigl[
 [(2k+1)D_p+n-a_{[p-1,p+4k]},2kD_p+n-a_{[p-1,p+4k-2]}],\\
&\ \ [n-a_1,\ldots,n-a_{p-2},2kD_p+n-a_{[p-1,p+4k-1]},
%\\&\ \ 
 2kD_p+n-a_{[p-1,p+4k-2]}-a_{p+4k}],\\
 &\quad[0^p,D_p-a_{p-1},\ldots,D_p-a_{p+4k-2},
  a_{p+4k+1},a_{p+4k+2},\ldots,a_{\tilde n}]
\bigr],
\allowdisplaybreaks\\
%%%%%%%%%
\mathbf m^{(4k+1)}&=
\bigl[
 [(2k+1)D_p+n-a_{[p-1,p+4k]},{(2k+1)D_p+n-a_{[p-1,p+4k-1]}-a_{p+4k+1}}],\\
 &\qquad[n-a_1,\ldots,n-a_{p-2},{2kD_p+n-a_{[p-1,p+4k-1]}},%
%\\&\ \ 
(2k+1)D_p+n-a_{[p-1,p+4k+1]}],\\%[-1mm]
 &\qquad\quad[0^p,D_p-a_{p-1},\ldots,D_p-a_{p+4k-1},{a_{p+4k+2}},a_{p+4k+3},\ldots,a_{\tilde n}]
\bigr],
%%%%%%%%
\allowdisplaybreaks\\
\mathbf m^{(4k+2)}&=
\bigl[
 [(2k+1)D_p+n-a_{[p-1,p+4k]},(2k+2)D_p+n-a_{[p-1,p+4k+2]}],\\
 &\ \ [n-a_1,\ldots,n-a_{p-2},(2k+1)D_p+n-a_{[p-1,\ldots,p+4k]}-a_{p+4k+2}, 
 \\&\qquad
(2k+1)D_p+n-a_{[p-1,p+4k+1]}],\\
 &\quad[0^p,D_p-a_{p-1},\ldots,D_p-a_{p+4k},a_{p+4k+3},a_{p+4k+4},\ldots]
\bigr],\\
%%%%%%%%%%
\mathbf m^{(4k+3)}&=\bigl[
 [(2k+2)D_p+n-a_{[p-1,p+4k+1]}-a_{p+4k+3},
 (2k+2)D_p+n-a_{[p-1,p+4k+2]}],\\
 &\ \ [n-a_1,\ldots,n-a_{p-2},(2k+2)D_p+n-a_{[p-1,p+4k+3]},\\
 &\qquad(2k+1)D_p+n-a_{[p-1,p+4k+1]}],\\
 &\quad[0^p,D_p-a_{p-1},\ldots,D_p-a_{p+4k+1},a_{p+4k+4},a_{p+4k+5},\ldots,a_{\tilde n}]
\bigr]. 
%%%%%%%%%%%%%%%%%%
\allowdisplaybreaks\\
\mathbf m^{(3,1)}&=
\bigl[
 [2D_p+n+a_{p-2}-a_{p-1}-a_{[p-1,p+2]},2D_p+n-a_{[p-1,p+2]}],\\
 &\quad\ \ [n-a_1,\ldots,n-a_{p-3},D_p+n-a_{p-1}-a_{p+1}-a_{p+2},\\
 &\qquad\quad D_p+n-a_{p-1}-a_p-a_{p+2},D_p+n-a_{p-1}-a_p-a_{p+1}]\\
 &\qquad[0^p,D_p-a_{p-1},2D_p+a_{p-2}-a_{[p-1,p+2]},a_{p+3},a_{p+4},\ldots\ldots,a_{\tilde n}]
\bigr].
\allowdisplaybreaks\\
%%%%%%%%%%%%
\mathbf m^{(3,2)}&=\bigl[
 [D_p+n-a_{p-1}-a_p,
 3D_p+n+a_{p-2}-2a_{p-1}-2a_p-a_{p+1}-a_{p+2}-a_{p+3}],\\
 &\quad\ \ [n-a_1,\ldots,n-a_{p-3},D_p+n-a_{p-1}-a_p-a_{p+3},\\
 &\qquad\ \  D_p+n-a_{p-1}-a_p-a_{p+2},D_p+n-a_{p-1}-a_p-a_{p+1}],\\
 &\qquad[0^p,D_p-a_{p-1},D_p-a_p,D_p+a_{p-2}-a_{p-1}-a_p,a_{p+4},a_{p+5},\ldots,a_{\tilde n}]
\bigr].
\end{align*}

Moreover, if
\begin{equation*}
 a_{p+N-2}+a_{p+N}=D_p
\end{equation*}
for a positive integer $N$, then the value of
$\mathcal R\mathcal S\mathbf m$ is unchanged whether
$N_{\mathbf m}$ is taken to be $N-1$ or $N$.
\end{theorem}

\begin{remark}
The last claim of the theorem follows from the fact that $\mc_\sigma(\mathbf m)=\mathbf m$ if 
$d_\sigma(\mathbf m)=0$ (cf.~Remark~\ref{rem:d0}). 
It also follows from the fact that, for example,  by replacing
$a_1,\dots,a_{\tilde n}$ with
$a_1,\dots,a_{N-1},a_N-\epsilon,\dots,a_{\tilde n}-\epsilon,
(\tilde n-N+1)\epsilon$, where 
$0\le\epsilon\ll1$. Note that 
Remark~\ref{remark:epsilon} justifies this argument.
\end{remark}

\begin{remark} Under the notation in Theorem~\ref{thm:RSm}, 
$\mathbf m=(2n-3)3,2^n,1^{2n}\in\mathcal F$ ($n\ge3$) satisfies 
$\mathbf m^{(3,1)}\in\mathcal F$ when $n=3$ and $\mathbf m^{(n)}\in\mathcal F$
when $n>3$.
Moreover, $\mathbf m=1^3,1^3,1^3$ satisfies $\mathbf m^{(0)}\in\mathcal F$ and  
$\mathbf m=9^2,64^3,3^6$ satisfies $\mathbf m^{(3,2)}\in\mathcal F$.

The following table gives the numbers $|\{\mathbf m\in\mathcal F_{50}\mid \Np(\mathbf m)=3,\ 
\mathbf m^{(*)}\in\mathcal F^{E_8}\}|$ for $(*)=(0),\,(1),\ldots$.

\medskip
\quad
\begin{tabular}{|r|r|r|r|rr|r|r|r|r|r|r|}
\multicolumn{11}{c}{\scalebox{0.9}[1]{$\mathbf m\in\mathcal F_{50}$,\ \ $N_p(\mathbf m)=3$}}
\\ \hline
%$\mathcal F^{E_8}_{50}$&
$\mathbf m^{(0)}$ & $\mathbf m^{(1)}$ & $\mathbf m^{(2)}$ & $\mathbf m^{(3)}$ & 
$\mathbf m^{(3,1)}\!$ &$\mathcal F^{E_8}_{50}$
& $\mathbf m^{(3,2)}$ & $\mathbf m^{(4)}$ & $\mathbf m^{(5)}$ & $\mathbf m^{(6)}$ & $\mathbf m^{(7)}$& $\mathbf m^{(8)}$\\ \hline
%&
4772&8870&2617&276&2312\!&\!\!(351)&34&32&9&1&3&2\\ \hline
\end{tabular}
\end{remark}

\medskip
\begin{exmp}
Let $\mathbf m=(2n)21,2^{n+1}1,1^{2n+3}$ with $n\ge1$.
Then
\[
\mathcal S\mathbf m=\bigl[[4n+6,2n+3],[2n+3,2n+3,2n+3],[2n,\overbrace{2,\dots,2}^{n+2},\overbrace{1\dots,1}^{2n+5}]\bigr]
\]
and $\mathcal R\mathcal S\mathbf m$ is obtained by the following sequence
of Katz reductions:
\begin{align*}
\mathcal S\mathbf m&\xrightarrow[1,1,1]{2n}
\bigl[[\underline{2n+6},2n+3],[3,\underline{2n+3},2n+3],
  [0,\overbrace{\underline{2},\dots,2}^{n+2},\overbrace{1\dots,1}^{2n+5}]\bigr]
\allowdisplaybreaks\\
 &\xrightarrow[1,2,2]{2}
\bigl[[\underline{2n+4},2n+3],[3,2n+1,\underline{2n+3}],
  [0,0,\overbrace{\underline{2},\dots,2}^{n+1},\overbrace{1\dots,1}^{2n+5}]\bigr]
\allowdisplaybreaks\\
 &\xrightarrow[1,3,3]{2}
\bigl[[2n+2,\underline{2n+3}],[3,2n+1,\underline{2n+1}],
  [0,0,0,\overbrace{\underline{2},\dots,2}^{n},\overbrace{1\dots,1}^{2n+5}]\bigr]
\allowdisplaybreaks\\
 &\xrightarrow[2,3,4]{1}
\bigl[[2n+2,\underline{2n+2}],[3,\underline{2n+1},2n],
  [0,0,0,1,\overbrace{\underline{2},\dots,2}^{n-1},\overbrace{1\dots,1}^{2n+5}]\bigr]\ \cdots\cdots\\
 &(k=0,1,2,\ldots)\\%[-3mm]
 &\xrightarrow[2,2,5+4k]{1}
\bigl[[\underline{2n+2-2k},2n+1-2k],[3,\underline{2n-2k},2n-2k],
  [0,0,0,\overbrace{1,\cdots,1}^{4k+2},\overbrace{\underline{2},\dots,2}^{n-2-4k},\overbrace{1\dots,1}^{2n+5}]\bigr]
\allowdisplaybreaks\\
 &\xrightarrow[1,2,6+4k]{1}
\bigl[[\underline{2n+1-2k},2n+1-2k],[3,2n-1-2k,\underline{2n-2k}],
  [0,0,0,\overbrace{1,\cdots,1}^{4k+3},\overbrace{\underline{2},\dots,2}^{n-3-4k},\overbrace{1\dots,1}^{2n+5}]\bigr]
\allowdisplaybreaks\\
 &\xrightarrow[1,3,7+4k]{1}
\bigl[[2n-2k,\underline{2n+1-2k}],[3,2n-1-2k,\underline{2n-1-2k}],
  [0,0,0,\overbrace{1,\cdots,1}^{4k+4},\overbrace{\underline{2},\dots,2}^{n-4-4k},\overbrace{1\dots,1}^{2n+5}]\bigr]
\allowdisplaybreaks\\
 &\xrightarrow[2,3,8+4k]{1}
\bigl[[2n-2k,\underline{2n-2k}],[3,\underline{2n-1-2k},2n-2-2k],
  [0,0,0,\overbrace{1,\cdots,1}^{4k+5},\overbrace{\underline{2},\dots,2}^{n-5-4k},\overbrace{1\dots,1}^{2n+5}]\bigr]\\
&\cdots\cdots
\end{align*}
The above reductions terminate when the term $2$ disappears from the
third partition of the spectral type.

Note that this sequence of reductions is the one given in the theorem but it is not unique. 
For example,  
$\xrightarrow[1,3,7+4k]{1}\, \xrightarrow[2,3,8+4k]{1}$ can be replaced by $\xrightarrow[2,3,7+4k]{1}\, \xrightarrow[1,3,8+4k]{1}$ for some $k$.
Moreover, Remark~\ref{remark:epsilon} yields many variations of this
example.
\end{exmp}

\begin{remark}\label{remark:epsilon}
Since $\Hd_\sigma$ and $\mc_\sigma$ define 
linear maps on the root space, 
the arguments and results such as Theorem~\ref{thm:RSm} for a spectral type
$\mathbf m\in\mathcal P$ with $\idx\mathbf m<0$ also apply to $r\mathbf m$ 
for every positive integer $r$.  Here, $\idx r\mathbf m=r^2\idx\mathbf m$.
Hence, we may extend the multiplicities $m_{j,\nu}$ to nonnegative rational numbers, since 
a rational spectral type can be realized by multiplying
$\mathbf m$ by a suitable positive integer.

Set
\[
\mathbf m_\epsilon=\bigl[[n-\epsilon,\epsilon],\,[m_{11},\cdots, m_{1,n_1}],\ldots,
 [m_{p1},\cdots, m_{pn_p}]\bigr]
\]
for a positive rational number $\epsilon$ with $0\le\epsilon\ll1$. 
Then $\mathbf m_\epsilon\in\mathcal F$ and $N_{\mathbf m_\epsilon}$ does not depend on 
$\epsilon$. 

The ``spectral type" 
\[
 \begin{split}
  \tilde{\mathbf m}_{\epsilon}&=\bigl[[m_{11}{-}\epsilon_{11},\ldots,m_{1n_1}{-}\epsilon_{1n_1},
%   \epsilon_{11}+\cdots+\epsilon_{1,n_1},
  \epsilon_{11}{+}\cdots{+}\epsilon_{1n_1}],\ldots,
  [m_{p1}{-}\epsilon_{p1},\ldots,m_{pn_p}{-}\epsilon_{pn_p},
%   \epsilon_{11}+\cdots+\epsilon_{1,n_1},
  \epsilon_{p1}{+}\cdots{+}\epsilon_{pn_p}]\bigr]\\
  &\qquad 
   \text{with \ }0\ll \epsilon_{11}\ll \cdots \ll \epsilon_{1n_1}\ll \epsilon_{21}\ll \cdots \ll \epsilon_{pn_p}
  \ll 1
 \end{split}
\]
is useful to obtain a unique algorithm using maximal Katz reductions as is given 
in Theorem~\ref{thm:RSm}.

Furthermore, applying these deformations to spectral types such as
$3^2,2^3,1^6$ and $(2n)21,2^{n+1}1,1^{2n+3}$, we obtain many further examples.
\end{remark}

To prove the theorem, we redefine \eqref{eq:an} as follows.
\begin{definition}\label{def:linord}
Fix a bijective map
\begin{align*}
 \{1,\dots,\tilde n\}\ni k \mapsto I(k)\in\mathcal I:=\{(j,\nu)\mid j=1,\dots,p,\ \nu=1,\dots,n_j\}\subset\mathbb Z_{>0}\times\mathbb Z_{>0}
\end{align*}
with $\tilde n:=n_1+\cdots+n_p$ so that
\[
 1\le k< \ell\le \tilde n\ \Rightarrow \ m_{I(k)}\ge m_{I(\ell)}.
\]
Moreover, we set
\begin{align*}
 a_k&=\begin{cases}
    m_{I(k)}& (1\le k\le \tilde n),\\
         0 & (k>\tilde n),
      \end{cases}\\
 I_{k_1,\dots,k_m}&:=\{j_{k_i}\mid I(k_i)=(j_{k_i},\nu_{k_i}),\ 1\le i\le m\}\subset\{1,\dots,p\},\\
 I_{[k,\ell]}&:=\begin{cases}
				I_{k,k+1,\ldots,\ell}&(k\le \ell),\\
				\emptyset&(k>\ell).
				\end{cases}
\end{align*}
\end{definition}
%%%%%%%%%%%%%%%%%%%%%  Lemma lem:p-1, ineq:sum1, sum2 %%%%%%%%%%%%%%%%%%%%%%%%%
The following lemma is useful for considering Katz reductions.
\begin{lemma}\label{lem:red}
Let $\mathbf m=m_{11}\cdots m_{1n_1},\,\ldots,\,m_{p1}\cdots m_{p,n_p}\in\mathcal P$ and
$\sigma=(\sigma_1,\dots,\sigma_p)\in\mathbb Z^p_{>0}$.
Put 
\[\sigma'=(\sigma_1,\dots,\sigma_{k-1},\sigma'_k,\sigma_{k+1},\dots,\sigma_p),\]
where
$\sigma'_k\in\mathbb Z_{>0}\setminus\{\sigma_k\}$. Then
\begin{align}
 d_{\sigma}(\mc_\sigma\mathbf m)&=-d_{\sigma}(\mathbf m),\label{eq:inv}\\
 d_{\sigma'}(\mc_\sigma\mathbf m)&=m_{k,\sigma'_k}-m_{k,\sigma_k}.\label{eq:dif1}
\end{align}
In particular, if $m_{k,\sigma'_k}\le m_{k,\sigma_k}$, we have
$d_{\sigma'}(\mc_\sigma\mathbf m)\le 0$. 
\end{lemma}
\begin{proof}
Equation \eqref{eq:inv} follows from \eqref{eq:KatzRed} and \eqref{eq:mc2eq1}. 
Moreover,  by Definition~\ref{def:mc}, we have
\begin{align*}
 d_{\sigma'}(\mc_\sigma\mathbf m)-d_{\sigma}(\mc_\sigma\mathbf m)=
  m_{k,\sigma'_k}-\bigl(m_{k,\sigma_k}-d_\sigma(\mathbf m)\bigr).
\end{align*}
Hence \eqref{eq:dif1} follows. 
\end{proof}

The following lemma will be used in the proof of Theorem~\ref{thm:RSm}.
\begin{lemma}\label{lem:p-1}
{\rm(1)} \ 
Suppose that there exist indices $i_1,\dots,i_p$ such that $i_1<i_2<\cdots<i_p$ and 
\begin{equation}
  a_{i_1}+\cdots+a_{i_p}>(p-2)n.\label{ineq:assum}
\end{equation} 
Then 
\begin{equation*}
  |I_{i_1,\dots,i_p}|=p-1
\end{equation*}
and 
\begin{equation}\label{ineq:sump}
  a_{j_1}+\cdots+a_{j_p}<(p-\tfrac32)n. 
\end{equation}

\noindent
{\rm(2)} \ Suppose $i_{p-1}\ge p$ in {\rm(1)}.
Then  
\begin{equation*}
 |I_{[1,i_p]}|=p-1,\qquad  |I_{[1,p-2]}|=p-2,\qquad  a_{[1,i_p]}\le(p-1)n.
\end{equation*}
Hence, after relabeling, 
we may assume 
\begin{equation}\label{ineq:p-1am}
a_j=m_{j,1}\quad(j=1,\dots,p-2),\qquad I_{p-1,p}\ni p-1,\qquad
I_{[1,i_p]}=\{1,\dots,p-1\}. 
\end{equation}

If there exists $k$ such that 
\[I_k\ne\{p-1\}\qquad\text{and}\qquad p-1\le k\le i_{p-1}, \]
then $I_j=\{p-1\}$ for all $j$ satisfying $p-1\le j\le i_p$ and $j\ne k$.

Moreover, if $m\ge 2$, we have
\begin{equation}\label{ineq:Dp2}
  a_{[p-1,i_p]}\le a_{p-2}+mD_p.
\end{equation}

Let $j_1$ and $j_2$ be positive integers such that  $p-1\le j_1<j_2\le i_p$ and $j_1\le i_{p-1}$.
Put $J=\{p-1,\dots,i_p\}\setminus\{j_1,j_2\}$. 
Then, if $m\ge 1$, we have 
\begin{align}\label{ineq:Dp1}
 \sum_{\nu\in J} a_\nu < a_{p-2}+mD_p.
\end{align}
\end{lemma}

\begin{remark}\label{rem:next3}
{\rm (1)} \ The inequality \eqref{ineq:sump} is best possible because of the example
\[
 \mathbf m=\bigl[[m,m],[m-1,\overbrace{1,\dots,1}^{m+1}],[\overbrace{1,\dots,1}^{2m}]\bigr]\qquad(m\gg1).
\]

\smallskip\noindent
{\rm(2)} \ 
Under the notation in Theorem~\ref{thm:RSm}, 
Lemma~\ref{lem:p-1}~{\rm (2)} is applied to the following cases.

Suppose $(D_p:N)$ holds with an integer $N\ge2$.  Then
\begin{equation*}
a_{[1,p-2]}+a_{p+N-2}+a_{p+N}>(p-2)n. 
\end{equation*}
Moreover, if $\mathbf m^{(N)}_{3,p+N}<\mathbf m^{(N)}_{3,p+N+1}$ 
for $\mathbf m^{(N)}$ in Theorem~\ref{thm:main}, 
then
\begin{equation*}
 a_{[1,p-2]}+a_{p+N-2}+a_{p+N+1}>(p-2)n.
\end{equation*}

\smallskip\noindent
{\rm(3)} \ 
Set $\mathbf m=(16)66,7777,4^7\in\mathcal F$. Then  $p=3$, $D_p=12$,  $a_{p-2}+D_p=28$, 
$a_{p+2}+a_{p+4}=13$, $N_{\mathbf m}=4$ and $a_{p-1}+\cdots+a_{p+2}=a_{p-2}+D_p$.
Hence, the condition $j_1\le i_{p-1}$ is necessary for \eqref{ineq:Dp1}.
\end{remark}

\begin{proof}[Proof of Lemma~\ref{lem:p-1}]
Suppose that \eqref{ineq:assum} holds. 
Since $\mathbf m$ is fundamental, \eqref{ineq:assum} implies  
$|I_{i_1,\dots,i_p}|<p$.
Since $a_{i_1}+\cdots+a_{i_p}\le|I_{i_1,\dots,i_p}|n$, we conclude $|I_{i_1,\dots,i_p}|=p-1$.
Hence, there exist $1\le k<\ell\le p$
satisfying $|I_{i_k,i_\ell}|=1$. 
Then
\begin{align*}
 a_{i_1}+\cdots+a_{i_p}&=(a_{i_1}+\cdots+a_{i_p}-a_{i_\ell})+a_{i_\ell}< (p-2)n+\tfrac n2.
\end{align*}
Thus, \eqref{ineq:sump} follows.

If
\[
  0<k_1<k_2<\cdots<k_p\text{ and } k_\nu\le i_\nu\quad (1\le \nu\le p),
\]
then $a_{k_1}+\cdots+a_{k_p} >(p-2)n$ and $|I_{k_1,\dots,k_p}|=p-1$.

If $|I_{[1,p-2]}|\le p-3$, then $a_{p-2}\le \frac n2$ and
$a_{[1,p]}=a_{[1,p-2]}+a_{p-1}+a_p\le (p-3)n+\tfrac n2+\tfrac n2\le (p-2)n$.
Hence, $|I_{[1,p-2]}|=p-2$ and we may assume $a_j=m_{j,1}$ for $j=1,\dots,p-2$
and $a_{j_0}=m_{p-1,1}$ with a suitable $j_0\in\{p-1,p\}$.

Suppose $j_{p-1}\ge p$. 
Since $|I_{1,\dots,p-2,p-1,p+1}|=|I_{1,\dots,p-2,p,p+1}|=p-1$, 
we have $I_{[1,p+1]}=\{1,\dots,p-1\}$.
Hence, $I_{1,\dots,p-2,j_0,j}=\{1,\dots,p-1\}$ for 
$p+1\le j\le i_p$ and therefore, $I_{[1,i_p]}=\{1,\dots,p-1\}$.
%%%

Moreover, if there exists $k$ such that $I_k\ne\{p-1\}$ and $p-1\le k\le i_{p-1}$, 
then $p-1\in I_{1,\dots,p-2,k,j}$ and $I_j=\{p-1\}$ for $j$ satisfying $p-1\le j\le i_p$ and $j\ne k$.

When $m\ge 2$, 
\begin{align*}
 a_{[p-1,i_p]}-a_{p-2}+mD_p&=a_{[p-1,i_p]}-a_{p-2}+m a_{[1,p-2]}-m(p-2)n\\
             &=a_{[1,i_p]}+\sum_{\nu=1}^{p-2}(m-2)(a_\nu-n)+a_{[1,p-3]}-2(p-2)n\\
             &\le(p-1)n+(p-3)n-2(p-2)n=0. 
\end{align*}

Let $j_1$ and $j_2$ be indices in Lemma~\ref{lem:p-1} (2).
Then $a_{j_1}+a_{j_2}\ge D_p$, namely,
\begin{align*}
 a_{j_1}+a_{j_2}+a_{[1,p-2]}>(p-2)n. 
\end{align*}
Suppose $\sum_{j\in J} a_j\ge a_{p-2}+D_p$, namely,
\[
  \sum_{j\in J}a_j + a_{[1,p-3]}\ge(p-2)n.
\]
Then
\begin{align*}
(2p-4)n &<  \sum_{j\in J}a_j + a_{[1,p-3]} + a_{[1,p-2]}+a_{j_1}+a_{j_2}\\
  &=a_{[1,i_p]}+a_{[1,p-3]}\le (p-1)n+(p-3)n,
\end{align*}
which is a contradiction.  Since $D_p\ge0$, we have the lemma.
\end{proof}

The following lemma justifies the case classification of the reductions 
of $\mathbf m^{(2)}$ given in the theorem. 
\begin{lemma}\label{lem:30-32}
If one of three conditions \eqref{cond:3-1}, \eqref{cond:3-2} and $(D_p:3)$ holds,
then the other two conditions do not hold, whereas  $(D_p:2)$ holds.
\end{lemma}
\begin{proof}
Note first that each of the these conditions implies $(D_p:2)$.

Suppose \eqref{cond:3-2} or $(D_p:3)$ holds. 
Then Lemma~\ref{lem:p-1} implies  $|I_{1,\dots,p+3}|=p-1$. 

If $(D_p:3)$ holds, then
\begin{align*}
 &(a_{[1,p-3]}+a_{p-2}+a_{p+1}+a_{p+3})+
 (a_{[1,p-3]}+a_{p-1}+a_{p+1}+a_{p+2})\\
 &=a_{[1,p+3]}+ a_{[1,p-3]}+(a_{p+1}-a_p)\\
 &\le (p-1)n+(p-3)n=2(p-2)n, 
\end{align*}
which shows that \eqref{cond:3-1} %\eqref{LL:31}
does not hold.

Similarly, if \eqref{cond:3-2} 
holds, then 
neither \eqref{cond:3-1} nor $(D_p:3)$
holds, since
\begin{align*}
 &(a_{[1,p-3]}+a_{p-1}+a_p+a_{p+3})+(a_{[1,p-3]}+a_{p-1}+a_{p+1}+a_{p+2})\\
 &=a_{[1,p+3]}+a_{[1,p-3]}+(a_{p-1}-a_{p-2}) \le 2(p-2)n,\\
 &(a_{[1,p-3]}+a_{p-1}+a_p+a_{p+3})+(a_{[1,p-3]}+a_{p-2}+a_{p+1}+a_{p+3})\\
 &=a_{[1,p+3]}+a_{[1,p-3]}+(a_{p+3}-a_{p+2})\le 2(p-2)n.
\end{align*}
Thus we have the lemma.
\end{proof}
%{\color{gray}Example. $\mathbf m=n^2,1^{2n},1^{2n}$.} 

%%%%%%%%%%%%%%%%%%%%%%%%%%%%%%%%%%%%%%%%%%%%%%%%
\begin{proof}[Proof of Theorem~\ref{thm:RSm}]
Under the notation in the theorem,
\begin{align*}
\mathcal S\mathbf m&=\mathcal S m_{1,1}\cdots m_{1,n_1},\,m_{2,1}\cdots m_{2,n_p},\,\ldots,\,
        m_{p,1}\cdots m_{p,n_p}\\
&=\bigl[
 [(p-1)n,n],[\overbrace{n,\dots,n}^p],[a_1,a_2,\ldots,a_p,a_{p+1},\ldots]
\bigr]\\
&\xrightarrow[1,1,1]{a_1}\,\xrightarrow[1,2,2]{a_2}\,\cdots\,\xrightarrow[1,p,p]{a_p} \\
\mathbf m^{(0)}&=\bigl[[(p-1)n-a_1-\cdots - a_p,n],\\%[-2mm]
&\qquad [n-a_1,n-a_2,\ldots,\underline{n-a_p}],
 [\overbrace{0,\dots,0}^p,\underline{a_{p+1}},a_{p+2},\ldots]
\bigr].
\intertext{Then we have $d_{1,p,p+1}(\mathbf m^{(0)})\le 0$ by Lemma~\ref{lem:red} with 
$a_p\ge a_{p+1}$. Since}
&d_{2,p,p+1}(\mathbf m^{(0)})=a_1+\cdots+a_{p-1}+a_{p+1}-(p-2)n=a_{p-1}+a_{p+1}-D_p, 
\end{align*}
the spectral type $\mathbf m^{(0)}$ is fundamental if $(D_p:1)$ does not hold.  

Hereafter we assume $N_{\mathbf m}\ge 1$.  Then
\begin{align*}
 \mathbf m^{(0)}&\xrightarrow [2,p,p+1]{a_{p-1}+a_{p+1}-D_p}
\allowdisplaybreaks\\
 \mathbf m^{(1)}&=\bigl[
  [D_p+n-a_{p-1}-a_p,\underline{D_p+n-a_{p-1}-a_{p+1}}],\\
 &\qquad
   [n-a_1\ldots,n-a_{p-1},D_p + n- a_{p-1}-a_p-a_{p+1}],\\
 &\qquad\quad
   [0^p,D_p-a_{p-1},a_{p+2},a_{p+3},\cdots]
 \bigr].
\allowdisplaybreaks
\intertext{Lemma~\ref{lem:red} shows $d_{2,p,p+1}(\mathbf m^{(1)})< 0$, 
$d_{2,p-1,p+1}(\mathbf m^{(1)})\le 0$ and $d_{2,p,p+2}(\mathbf m^{(1)})\le 0$. Since} 
% &\ \  d_{2,p-1,p+1}(\mathbf m^{(1)})=a_p-a_{p-1}\le 0,\\
% &\ \  d_{2,p,p+2}(\mathbf m^{(1)})=a_{p+2}-a_{p+1}\le 0,\\
 &\ \  d_{2,p-1,p+2}(\mathbf m^{(1)})=a_p+a_{p+2}-D_p,
\end{align*}
the tuple $\mathbf m^{(1)}$ is fundamental
if $(D_p:2)$ does not hold.

Hereafter we assume $N_{\mathbf m}\ge 2$. Then
\begin{align*}
 \mathbf m^{(1)}&\xrightarrow [2,p-1,p+2]{a_p+a_{p+2}-D_p}\\
 \mathbf m^{(2)}&=\bigl[
  [D_p+n-a_{p-1}-a_p,2D_p+n-a_{p-1}-a_p-a_{p+1}-a_{p+2}],\\
 &\qquad
   [n-a_1\ldots,n-a_{p-2},D_p+n-a_{p-1}-a_p-a_{p+2},D_p+n-a_{p-1}-a_p-a_{p+1}],\\
 &\qquad\quad
   [0^p,D_p-a_{p-1},D_p-a_p,a_{p+3},a_{p+4},a_{p+5},\ldots]
 \bigr].
\intertext{It follows from Lemma~\ref{lem:red} that}
  &d_{2,p-1,p+2}(\mathbf m^{(2)})<0,\ 
  d_{1,p-1,p+2}(\mathbf m^{(2)})\le0,\ 
  d_{2,p-2,p+2}(\mathbf m^{(2)})\le0 \text{ \ and \ }
  d_{2,p-1,p+3}(\mathbf m^{(2)})\le0.
\intertext{Moreover,}
 &d_{1,p-2,p+2}(\mathbf m^{(2)})= a_1+\cdots+a_{p-3}+a_{p-1}+a_{p+1}+a_{p+2}-(p-2)n\\
 &\phantom{d_{1,p-2,p+2}(\mathbf m^{(2)})}=a_{p-1}+a_{p+1}+a_{p+2}-a_{p-2}-D_p
 \qquad(>0 \rightarrow \mathbf m^{(3,1)}).%\\
% &d_{1,p-1,p+2}(\mathbf m^{(2)})= a_{p+1}-a_p\le 0,\\
% &d_{2,p-2,p+2}(\mathbf m^{(2)})= a_{p-1}-a_{p-2}\le 0.
\intertext{In the following, we assume $\mathbf m^{(2)}_{3,p+2}<\mathbf m^{(2)}_{3,p+3}$, 
namely, $a_p+a_{p+3}>D_p$.  Then 
$|I_{1,\dots,p+3}|=p-1$.}
 &d_{1,p-2,p+3}(\mathbf m^{(2)})= 2a_1+\cdots+2a_{p-3}+a_{p-2}+a_{p-1}+\cdots+a_{p+3}-2(p-2)n\\
 &\phantom{d_{1,p-2,p+3}(\mathbf m^{(2)})}\le 
  a_1+\cdots+a_{p+3}+(p-3)n-(2p-4)n\\
 &\phantom{d_{1,p-2,p+3}(\mathbf m^{(2)})}
=a_1+\cdots+a_{p+3}-(p-1)n\le 0,%\quad\text{\color{red}($\because$\ Remark~\ref{rem:next3})}
\\
 &d_{1,p-1,p+3}(\mathbf m^{(2)})= a_{p+1}+a_{p+3}-D_p
 \qquad\qquad\qquad\quad\ (>0\rightarrow \mathbf m^{(3)}),\\
 &d_{2,p-2,p+3}(\mathbf m^{(2)})= a_1+\cdots+a_{p-3}+a_{p-1}+a_p+a_{p+3}-(p-2)n\\
 &\phantom{d_{2,p-2,p+3}(\mathbf m^{(2)})}=a_{p-1}+a_p+a_{p+3}-a_{p-2}-D_p\qquad(>0\rightarrow \mathbf m^{(3,2)}).
%\\
% &d_{2,p-1,p+3}(\mathbf m^{(2)})= a_{p+3}-a_{p+2} \le 0.
\end{align*}

First, suppose that $(D_p:3)$ holds.

We define $\mathbf m^{(\ell)}$ inductively  for $\ell=3,\dots,N_{\mathbf m}$
and show that $\mathbf m^{(N_{\mathbf m})}$ is fundamental if $N_{mathbf m}\ge3$.

We assume that 
\begin{align*}
\mathbf m^{(4k+3)}=&\bigl[
 [\underline{(2k+2)D_p+n-a_{[p-1,p+4k+1]}-a_{p+4k+3}},\\
 &\ (2k+2)D_p+n-a_{[p-1,p+4k+2]}],\\
 &\ \ [n-a_1,\ldots,n-a_{p-2},(2k+2)D_p+n-a_{[p-1,p+4k+3]},\\
 &\qquad\underline{(2k+1)D_p+n-a_{[p-1,p+4k+1]}}],\\
 &\quad[0^p,D_p-a_{p-1},\ldots,D_p-a_{p+4k+1},\underline{a_{p+4k+4}},a_{p+4k+5},\ldots]
\bigr],
\end{align*}
which follows from $\mathbf m^{(3)}=\mc_{(1,p-1,p+3)}\mathbf m^{(2)}$ when $k=0$.
Note that $\mathbf m^{(4k+3)}_{1,1}\ge \mathbf m^{(4k+3)}_{1,2}$.

Suppose first that $\mathbf m^{(4k+3)}_{3,p+4k+3}\ge \mathbf m^{(4k+3)}_{3,p+4k+4}$.
\begin{align*}
d_{1,p-2,p+4k+3}(\mathbf m^{(4k+3)})&=a_{[p-1,p+4k+2]}-a_{p+4k+1}-a_{p-2}-(2k+1)D_p<0&&\Leftarrow\eqref{ineq:Dp1},\\
d_{1,p-1,p+4k+3}(\mathbf m^{(4k+3)})&=-d_{1,p-1,p+4k+3}(\mathbf m^{(4k+2)})
 = D_p-a_{p+4k+1}-a_{p+4k+3}<0&&\Leftarrow\eqref{eq:inv},\\
d_{1,p,p+4k+3}(\mathbf m^{(4k+3)})  &=a_{p+4k+2}-a_{p+4k+1}\le0&&\Leftarrow\eqref{eq:dif1}.
\end{align*}
Here we have used the fact $a_{[p-1,p+4k+2]}=a_{[p-1,p+4k+3]}-a_{p+4k+3}$ 
and the first inequality follows from 
\eqref{ineq:Dp1} in Lemma~\ref{lem:p-1} with $(j_1,j_2)=(p+4k+1,p+4k+3)$. 
Hence $\mathbf m^{(4k+3)}$ is fundamental in this case.

Now suppose that $\mathbf m^{(4k+3)}_{3,p+4k+3}< \mathbf m^{(4k+3)}_{3,p+4k+4}$, that is,
\begin{equation*}
  a_{p+4k+1}+a_{p+4k+4}>D_p.
\end{equation*}
Then
\begin{align*}
d_{1,p-2,p+4k+4}(\mathbf m^{(4k+3)})&=a_{[p-1,p+4k+4]}-a_{p+4k+3}-a_{p-2}-(2k+2)D_p<0&&\Leftarrow\eqref{ineq:Dp2},\\
d_{1,p-1,p+4k+4}(\mathbf m^{(4k+3)})&=a_{p+4k+4}-a_{p+4k+3}\le 0&&\Leftarrow\eqref{eq:dif1},\\
d_{1,p,p+4k+4}(\mathbf m^{(4k+3)})&=a_{p+4k+2}+a_{p+4k+4}-D_p\qquad(>0 \to\mathbf m^{(4k+4)}).
\end{align*}
Thus, $\mathbf m^{(4k+3)}$ is fundamental if  $(D_p:4k+4)$ does not hold.

\bigskip
Suppose now that $(D_p:4k+4)$ holds. 
Then $\mathbf m^{(4k+4)}=\mathrm{mc}_{1,p,p+4k+4}\mathbf m^{(4k+3)}$.

Replacing $k$ by $k+1$, we obtain $\mathbf m^{(4k)}=\mc_{(1,p,p+4k)}\mathbf m^{(4k-1)}$ 
for $k\ge 1$, under the assumption that $(D_p:4k)$ holds, where
\begin{align*}
\mathbf m^{(4k)}&=\bigl[
 [(2k+1)D_p+n-a_{[p-1,p+4k]},\underline{2kD_p+n-a_{[p-1,p+4k-2]}} ],\\
&\ \ [n-a_1,\ldots,n-a_{p-2},2kD_p+n-a_{[p-1,p+4k-1]},
%\\&\ \ 
\underline{2kD_p+n-a_{[p-1,p+4k-2]}-a_{p+4k}} ],\\
 &\quad[0^p,D_p-a_{p-1},\ldots,D_p-a_{p+4k-2},\underline{a_{p+4k+1}},a_{p+4k+2},\ldots]
\bigr]
\end{align*}
By the assumption $a_{p+4k}+a_{p+4k-2}>D_p$, we have the relations:
\begin{align*}
d_{1,p-2,p+4k}(\mathbf m^{(4k)})&=a_{[p-1,p+4k-3]}-a_{p-2}-(2k-1)D_p\le 0
  &&\Leftarrow\eqref{ineq:Dp1},\\
d_{1,p,p+4k}(\mathbf m^{(4k)})&=-d_{1,p,p+4k}(\mathbf m^{(4k-1)})=D_p-a_{p+4k}-a_{p+4k-2}<0
&&\Leftarrow\eqref{eq:inv},\\
d_{2,p-2,p+4k}(\mathbf m^{(4k)})&=a_{[p-1,p+4k]}-a_{p+4k-2}-a_{p-2}-2kD< 0&&\Leftarrow\eqref{ineq:Dp2},\\
d_{2,p,p+4k}(\mathbf m^{(4k)})&=a_{p+4k-1}-a_{p+4k-2}\le0&&\Leftarrow\eqref{eq:dif1},\\
d_{1,p-2,p+4k+1}(\mathbf m^{(4k)})&=a_{[p-1,p+4k+1]}-a_{p+4k}-a_{p+4k-1}-a_{p-2}-2kD_p<0&&\Leftarrow\eqref{ineq:Dp2},\\
d_{1,p,p+4k+1}(\mathbf m^{(4k)})&=a_{p+4k+1}-a_{p+4k}\le 0&&\Leftarrow\eqref{eq:dif1},\\
d_{2,p-2,p+4k+1}(\mathbf m^{(4k)})&=a_{[p-1,p+4k+1]}-a_{p-2}-(2k+2)D_p<0&&\Leftarrow\eqref{ineq:Dp2},\\
d_{2,p,p+4k+1}(\mathbf m^{(4k)})&=a_{p+4k+1}+a_{p+4k-1}-D_p\qquad(>0\to\mathbf m^{(4k+1)}). 
\end{align*}

\bigskip
For $k\ge1$, we have
$\mathbf m^{(4k+1)}=\mc_{(2,p,p+4k+1)}\mathbf m^{(4k)}$, 
where % 6 cases
\begin{align*}
\mathbf m^{(4k+1)}=&\bigl[
 [(2k+1)D_p+n-a_{[p-1,p+4k]},\underline{(2k+1)D_p+n-a_{[p-1,p+4k-1]}-a_{p+4k+1}}],\\
 &\ \ [n-a_1,\ldots,n-a_{p-2},\underline{2kD_p+n-a_{[p-1,p+4k-1]}},%
%\\&\ \ 
(2k+1)D_p+n-a_{[p-1,p+4k+1]}],\\
 &\quad[0^p,D_p-a_{p-1},\ldots,D_p-a_{p+4k-1},\underline{a_{p+4k+2}},a_{p+4k+3},\ldots]
\bigr]
\end{align*}
\begin{align*}
d_{2,p-2,p+4k+1}(\mathbf m^{(4k+1)})&=a_{[p-1,p+4k]}-a_{p+4k-1}-a_{p-2}-2kD_p<0
&&\Leftarrow\eqref{ineq:Dp2},\\
d_{2,p-1,p+4k+1}(\mathbf m^{(4k+1)})&=a_{p+4k}-a_{p+4k-1}\le 0&&\Leftarrow\eqref{eq:dif1},\\
d_{2,p,p+4k+1}(\mathbf m^{(4k+1)})&=-d_{2,p,p+4k+1}(\mathbf m^{(4k)})=D_p-a_{p+4k+1}-a_{p+4k-1}<0
&&\Leftarrow\eqref{eq:inv},\\
d_{2,p-2,p+4k+2}(\mathbf m^{(4k+1)})&=a_{[p-1,p+4k+2]}-a_{p+4k+1}- a_{p-2}-(2k+1)D_p<0
&&\Leftarrow\eqref{ineq:Dp2},\\
d_{2,p-1,p+4k+2}(\mathbf m^{(4k+1)})&=a_{p+4k+2}+a_{p+4k}-D_p\qquad(>0\to\mathbf m^{(4k+2)}),\\
d_{2,p,p+4k+2}(\mathbf m^{(4k+1)})&=a_{p+4k+2}-a_{p+4k+1}\le0&&\Leftarrow\eqref{eq:dif1}.
\end{align*}
\bigskip

For $k\ge1$, we have $\mathbf m^{(4k+2)}=\mc_{(2,p-1,p+4k+2)}\mathbf m^{(4k+1)}$, where
\begin{align*}
 \mathbf m^{(4k+2)}=&\bigl[
 [\underline{(2k+1)D_p+n-a_{[p-1,p+4k]}},(2k+2)D_p+n-a_{[p-1,p+4k+2]}],\\
 &\ \ [n-a_1,\ldots,n-a_{p-2},\underline{(2k+1)D_p+n-a_{[p-1,\ldots,p+4k]}-a_{p+4k+2}}
 \\&\qquad
(2k+1)D_p+n-a_{[p-1,p+4k+1]}],\\
 &\quad[0^p,D_p-a_{p-1},\ldots,D_p-a_{p+4k},\underline{a_{p+4k+3}},a_{p+4k+4},\ldots]
\bigr]. 
\end{align*}
\begin{align*}
d_{1,p-2,p+4k+2}(\mathbf m^{(4k+2)})&=a_{[p-1,p+4k+2]}-a_{p+4k}-a_{p-2}-(2k+1)D_p<0
&&\Leftarrow\eqref{ineq:Dp2},\\
d_{1,p-1,p+4k+2}(\mathbf m^{(4k+2)})&=a_{p+4k+1}-a_{p+4k}\le0 &&\Leftarrow\eqref{eq:dif1},\\
d_{2,p-2,p+4k+2}(\mathbf m^{(4k+2)})&=a_{[p-1,p+4k-1]} -a_{p-2}-2kD_p<0
&&\Leftarrow\eqref{ineq:Dp2},\\
d_{2,p-1,p+4k+2}(\mathbf m^{(4k+2)})&=-d_{2,p-1,p+4k+2}(\mathbf m^{(4k+1)})=D_p-a_{p+4k+2}-a_{p+4k}<0&&\Leftarrow\eqref{eq:inv},\\
d_{1,p-2,p+4k+3}(\mathbf m^{(4k+2)})&=a_{[p-1,p+4k+3]}-a_{p-2}-(2k+2)D_p<0
&&\Leftarrow\eqref{ineq:Dp2},\\
d_{1,p-1,p+4k+3}(\mathbf m^{(4k+2)})&=a_{p+4k+3}+a_{p+4k+2}-D_p\quad(>0\to\mathbf m^{(4k+3)}),\\
d_{2,p-1,p+4k+3}(\mathbf m^{(4k+2)})&=a_{[p-1,p+4k+3]}-a_{p+4k+2}-a_{p+4k+1}-a_{p-2}-(2k+1)D_p<0
&&\Leftarrow\eqref{ineq:Dp2},\\
d_{2,p-2,p+4k+3}(\mathbf m^{(4k+2)})&=a_{p+4k+3}-a_{p+4k+2}\le 0&&\Leftarrow\eqref{eq:dif1}.
\end{align*}

If $a_{p+4k+3}+a_{p+4k+2}-D_p\le 0$, $\mathbf m^{(4k+2)}$ is fundamental. 

If $a_{p+4k+3}+a_{p+4k+2}-D_p>0$, we have 
$\mathbf m^{(4k+3)}=\mc_{(1,p-1,p+4k+3)}\mathbf m^{(4k+2)}$.

Thus, by induction, $\mathbf m^{(N_{\mathbf m})}$ is fundamental except in
the cases \eqref{cond:3-1} and \eqref{cond:3-2}.

\bigskip
%%%%%%%%%%%%%%%%%%%%%%%%%%%%%%%%  (3.1)  %%%%%%%%%%%%%%%%%%%%%%%%%%%%%%%%
Suppose that \eqref{cond:3-1} holds.  
Then $\mathbf m^{(3,1)}=\mc_{(1,p-2,p+2)}\mathbf m^{(2)}$, where
\begin{align*}
\mathbf m^{(3,1)}&=\bigl[
 [\underline{2D_p+n+a_{p-2}-2a_{p-1}-a_p-a_{p+1}-a_{p+2}},\\
 &\qquad\quad 2D_p+n-a_{p-1}-a_p-a_{p+1}-a_{p+2}],\\
 &\quad\ \ [n-a_1,\ldots,n-a_{p-3},\underline{D_p+n-a_{p-1}-a_{p+1}-a_{p+2}},\\
 &\qquad\quad D_p+n-a_{p-1}-a_p-a_{p+2},D_p+n-a_{p-1}-a_p-a_{p+1}]\\
 &\qquad[0^p,D_p-a_{p-1},2D_p+a_{p-2}-a_{p-1}-a_p-a_{p+1}-a_{p+2},a_{p+3},a_{p+4},\ldots]
\bigr].
\end{align*}
Since $a_{[1,p-3]}+a_{p-1}+a_{p+1}+a_{p+2}>(p-2)n$, 
Lemma~\ref{lem:p-1} shows that 
we may assume \eqref{ineq:p-1am} with $i_p=p+2$. 
Since $I_{[1,p+2]}=I_{1,\dots,p-3,p-1,p,p+1}=\{1,\dots,p-1\}$,
we have $a_k=m_{p-2,2}$ with a suitable $k\in\{p-1,p,p+1\}$. 
Moreover, Lemma~\ref{lem:p-1} shows $I_{j}=\{p-1\}$ for 
$j\in\{p-1,p,p+1,p+2\}\setminus\{k\}$ and 
$I_{1,\dots,p-3,p-1,p+1,p+2}=I_{1,\dots,p-3,p-1,p,p+2}
=\{1,\dots,p-1\}$.  Then we must have $k=p-1$.  Hence
\begin{equation}
 a_j=m_{j,1}\quad(j=1,\dots,p-2),\qquad a_{p-1}=m_{p-2,2},\qquad
 a_{p+\nu}=m_{p-1,\nu+1}\quad(\nu=0,1,2).
\end{equation}
In particular, $|I_{1,\dots,p-4,p-2,\dots,p+3}|=p-2$ when $p>3$.

Here, $E_8^{(1)}=33,222,111111$ and $(6-\epsilon)\epsilon,33,222,111111$
are examples (cf.~Remark~\ref{remark:epsilon}).

Note that 
\begin{equation}
   a_p+a_{p+3}<D_p\label{C:311}.
\end{equation}
Indeed, if $a_p+a_{p+3}\ge D_p$, then $|I_{[1,p+3]}|=p-1$, and hence
\begin{align*}
  0&<(a_p+a_{p+3}-D_p)+(a_{p-1}+a_{p+1}+a_{p+2}-a_{p-2}-D_p)\\
      &= a_{[1,p+3]}+a_{[1,p-3]}-2(p-2)n\\
      &\le (p-1)n+(p-3)n-2(p-2)n= 0,
\end{align*}
a contradiction.

We have the following relations: 
\begin{align*}
 d_{1,p-2,p+1}(\mathbf m^{(3,1)})&=a_p-a_{p-1}\le0&&\Leftarrow\eqref{eq:dif1},\\
 d_{1,p-2,p+2}(\mathbf m^{(3,1)})&=D_p+a_{p-2}-a_{p-1}-a_p-a_{p+3}<0&&
  \Leftarrow\eqref{eq:inv},\\
 d_{1,p-2,p+3}(\mathbf m^{(3,1)})&=a_p+a_{p+3}-D_p< 0&&\Leftarrow\eqref{C:311}%\\
\end{align*}
Moreover, if $p>3$, we have
\begin{align*}
 d_{1,p-3,p+1}(\mathbf m^{(3,1)})&=a_{[1,p+2]}-a_{p-3}-a_{p-1}-(p-2)n < 0
\quad{\Leftarrow \ 
 |I_{1,\dots,p-4,p-2,p,p+1,p+2}|=p-2},\\
 d_{1,p-3,p+2}(\mathbf m^{(3,1)})&=a_{p-2}-a_{p-3}\le0
  \qquad\qquad\qquad\qquad\quad\ \ \Leftarrow\eqref{eq:dif1}\\
 d_{1,p-3,p+3}(\mathbf m^{(3,1)})
      &=a_{[p-1,p+3]}-a_{p-3}-2D_p\\
      &= (a_{[p-1,p+2]}-a_p-a_{p-3}- D_p)+(a_p+a_{p+3}-D_p)\\ 
      &< a_{[1,p+2]}-a_{p-3}-a_p-(p-2)n\le 0
\quad\ \ {\Leftarrow \ 
 |I_{1,\dots,p-4,p-2,p-1,p+1,p+2}|=p-2}.
\end{align*}
Hence, $\mathbf m^{(3.1)}\in\mathcal F$.

%%%%%%%%%%%%%%%%%%%%%%%%%%%%%%%%%%%%% (3.2) %%%%%%%%%%%%%%%%%%%%%%%
\bigskip
Suppose that \eqref{cond:3-2} holds. 
Then, $\mathbf m^{(3,2)}=\mc_{(2,p-2,p+3)}\mathbf m^{(2)}$, where
\begin{align*}
\mathbf m^{(3,2)}&=\bigl[
 [D_p+n-a_{p-1}-a_p,
 3D_p+n+a_{p-2}-2a_{p-1}-2a_p-a_{p+1}-a_{p+2}-a_{p+3}],\\
 &\quad\ \ [n-a_1,\ldots,n-a_{p-3},\underline{D_p+n-a_{p-1}-a_p-a_{p+3}},\\
 &\qquad\ \  D_p+n-a_{p-1}-a_p-a_{p+2},D_p+n-a_{p-1}-a_p-a_{p+1}],\\
 &\qquad[0^p,D_p-a_{p-1},D_p-a_p,D_p+a_{p-2}-a_{p-1}-a_p,a_{p+4},a_{p+5},\ldots]. 
\bigr]
\end{align*}
We may assume \eqref{ineq:p-1am} with $i_p=p+3$. 
There are two cases.

If $I_{[p-1,p]}=\{p-1\}$, the condition $|I_{1,\dots,p-3,p-1,p,j}|=p-1$
for $j\in\{p+1, p+2, p+3\}$ implies that $I_{p+1,p+2,p+3}=\{p-2\}$ and hence
\[
  a_j=m_{j,1}\quad(j=1,\dots,p-2),\quad
  a_{p-1}=m_{p-1,1},\quad a_{p}=m_{p-1,2},\quad a_{p+\nu}=m_{p-2,\nu+1}\quad(\nu=1,2,3).
\]

If there exists $k\in\{p-1,p\}$ such that $I_k \ne\{p-1\}$, 
then $I_{p+1,p+2,p+3}=\{p-1\}$ and the condition
$|I_{1,\dots,p-3,k,p+1,p+2}|=p-1$ implies $I_k=\{p-2\}$.  Therefore, 
\[
 a_j=m_{j,1}\quad(j=1,\dots,p-2),\quad
 \{a_{p-1},a_p\}=\{m_{p-2,2},m_{p-1,1}\},\quad a_{p+\nu}=m_{p-1,\nu+1}\quad(\nu=1,2,3).
\]
Here, the spectral types $9333,882,1^{18}$ and $99,6444,333333$ are corresponding examples.

We have the following relations:
\begin{align*}
 d_{1,p-2,p+3}(\mathbf m^{(3,2)})&=a_{p+1}+a_{p+2}-D_p&&\Leftarrow\eqref{eq:dif1}\\
 &<(a_{p+1}+a_{p+2}-D_p)+(a_{p-1}+a_p+a_{p+3}-a_{p-2}-D_p)&&\Leftarrow\eqref{cond:3-2}\\
 &=a_{[1,p+3]}+a_{[1,p-3]}-2(p-2)n\\
 &\le (p-1)n+(p-3)n - 2(p-2)n\le 0
,\allowdisplaybreaks\\
 d_{2,p-2,p+3}(\mathbf m^{(3,2)})&=D_p+a_{p-2}-a_{p-1}-a_p-a_{p+3}<0&&\Leftarrow \eqref{eq:inv}
,\allowdisplaybreaks\\
 d_{1,p-2,p+4}(\mathbf m^{(3,2)})&=a_{[p-1,p+4]}-a_{p+3}-a_{p-2}-2D_p\\
                                 &\le a_{[p-1,p+3]}-a_{p-2}-2D_p\le0&&\Leftarrow\eqref{ineq:Dp2}
,\allowdisplaybreaks\\
 d_{2,p-2,p+4}(\mathbf m^{(3,2)})&=a_{p+4}-a_{p+3}\le0&&\Leftarrow\eqref{eq:dif1}
.
\end{align*}
Moreover, if $p>3$, we have
\begin{align*}
 d_{1,p-3,p+3}(\mathbf m^{(3,2)})&=a_{[p-1,p+3]}-a_{p-3}-2D_p\le0 \qquad\qquad\qquad\quad
 \Leftarrow\eqref{ineq:Dp2}
%\color{gray}
.\\
 d_{2,p-3,p+3}(\mathbf m^{(3,2)})&=a_{p-2}-a_{p-3}\le0\qquad\qquad\qquad\qquad\qquad\qquad\Leftarrow\eqref{eq:dif1}
,\allowdisplaybreaks\\
 d_{1,p-3,p+4}(\mathbf m^{(3,2)})
 &=a_{[p-1,p+4]}+a_{p-1}+a_p-a_{p-3}-a_{p-2}-3D_p
 \\&\le a_{p+4}+a_{p-1}+a_p-a_{p-3}-D_p
 \quad\qquad\qquad\ \ \,\Leftarrow\eqref{ineq:Dp2}
 \\&\le a_{[1,p-4]}+a_{[p-2,p]} -(p-2)n
\le 0\qquad\qquad\quad\Leftarrow
 |I_{1,\dots,p-4,p-2,p-1,p}|=p-2,\\
 d_{2,p-3,p+4}(\mathbf m^{(3,2)})
 &= a_{p-1}+a_p+a_{p+4}-a_{p-3}-D_p\\
 &\le a_{[1,p+1]}-a_{p-3}-(p-2)n \le 0
 \qquad\qquad\qquad\Leftarrow |I_{1,\dots,p-4,p-2,\dots,p+1}|=p-2.
\end{align*}
Thus, $\mathbf m^{(3,2)}\in\mathcal F$, completing the proof of Theorem~\ref{thm:RSm}.
\end{proof}
%%%
We show the following proposition, which corresponds to Lemma~\ref{lem:main}~(4) and 
completes the proof of
Theorem~\ref{thm:main}.
\begin{proposition}\label{cor:RSm}
We have $\ord\mathcal R\mathcal S\mathbf m>\ord\mathbf m$
for\/ $\mathbf m\in\mathcal F\setminus\mathcal F^{E_8}$.
\end{proposition}
\begin{proof}
Since $\ord\mathbf m^{(0)}=pn-(a_1+\cdots+a_n)$, the inequality \eqref{ineq:sump} gives
$\ord\mathbf  m^{(0)}>\frac 32n$.

Suppose that $N_{\mathbf m}\ge 1$.
Then $\ord\mathbf m^{(0)}-\ord\mathbf m^{(1)}=a_1+\cdots+a_{p-1}+a_{p+1}-(p-2)n<\tfrac n2$ and 
hence $\ord\mathbf m^{(1)}>n$.

Thus, we may assume that $N:=N_{\mathbf m}\ge 2$. Then Lemma~\ref{lem:p-1} 
gives $|I_{[1,p+N]}|=p-1$, and we may further assume $a_j=m_{j,1}$ for $j=1,\dots,p-2$ and
\begin{equation}\label{eq:cor}
 \{a_{p+N+1},a_{p+N+2},\cdots,a_{\tilde n}\}\supset\{m_{p1},m_{p2},\dots,m_{pn_p}\}.
\end{equation}
Note that $m_{p1}+\cdots+m_{pn_p}=n$ and that the third partition of $\mathbf m^{(N)}$ equals
\[
 [0^p,D_p-a_{p-1},\ldots,D_p-a_{p+N-2},a_{p+N+1},a_{p+N+2},\cdots,a_{\tilde n}]. 
\]
Therefore, if  
\[
 D_p-a_{p-1}=(n-m_{11})+\cdots+(n-m_{p-2,1})-a_{p-1}
\]
is positive, then $\ord \mathbf m^{(N)}>n$ 

If $|I_{[1,p-1]}|=p-1$, then $m_{11}+\cdots+m_{p-2,1}+a_{p-1}<m_{11}+\cdots+m_{p1}\le (p-2)n$,
which implies $D_p-a_{p-1}>0$ and hence $\ord \mathbf m^{(N)}>n$.
Thus, we assume that $|I_{1,\dots,p-1}|=p-2$. 
Then, there exists $j\le p-2$ such that $a_{p-1}=m_{j2}$.
Hence, if $D_p=a_{p-1}$, then $j=1$ and 
\begin{equation}\label{eq:prop}
  p=3,\ (a_1,a_2)=(m_{11},m_{12}) \text{ \ and  \ }m_{11}+m_{12}=n.
\end{equation}
If $D_p=a_{p-1}=a_p$, then $m_{11}+a_3=n$, 
which is impossible since $m_{11}+m_{21}+m_{31}\le n$. 
Thus $\ord \mathbf m^{(N)}>n$.

Suppose that $\mathcal R\mathcal S\mathbf m=\mathbf m^{(3,2)}$. Since 
$a_{[1,p-3]}+a_{p-1}+a_{p+1}+a_{p+3}>(p-2)n$, we have \eqref{eq:cor} with $N=3$,
and and the same argument gives $\ord\mathbf m^{(3,2)}>n$.

Finally, suppose that $\mathcal R\mathcal S\mathbf m=\mathbf m^{(3,1)}$ 
and $\ord \mathbf m^{(3,1)}\le n$.
Then $N=2$ and the two sides of \eqref{eq:cor} are equal, and
$\ord\mathbf m^{(3,1)}=n$.
Moreover,  $D_p=a_p$,  and we have \eqref{eq:prop}.   Furthermore
$p+N=n_1+n_2$, which means $n_2=3$.  Hence, $\mathbf m$ is of $(2,3,*)$-type.
\end{proof}
%%%
\begin{remark}\label{remark:Risa}
To obtain the results in this paper, we developed several programs
for the computer algebra system \texttt{Risa/Asir}.
They compute examples and display the calculations in a suitable format.
Some of these programs are available in \verb|os_muldif.rr|
\cite{Or}. For example, it contains the following functions.

\medskip\noindent
\texttt{chkspt()} tests irreducibly realizability and computes $\idx\mathbf m$, $\mathcal R\mathbf m$, etc.
\quad($\mathbf m$ : spectral type)

\medskip\noindent
\texttt{spgen()} computes $\mathcal F_\ell$ for $\ell\le0$ and, more generally,

computes all $\mathbf m\in\mathcal P$ such that
$\mathbf m\underset{\mc}\sim\mathbf n$ and $\ord\mathbf m\le n$,
\quad
($\mathbf n\in\mathcal P$, $n\ge0$)

\medskip\noindent
\texttt{refinements()} computes all refinements of $\mathbf m$
\quad($\mathbf m\in\mathcal P$).

\medskip\noindent
\texttt{spfundclass()} computes $\mathcal F_\ell/!\sim$
\quad($\ell\le0$).

\medskip\noindent
\texttt{spfundauto()} computes the automorphism group of the Fuchsian
systems with a given spectral type $\mathbf m$.

\medskip\noindent
\texttt{spmcv()} calculates $\mc_\sigma(\mathbf m)$\quad($\mathbf m\in\mathcal P,\ \sigma\in\mathbb Z^p_{\ge 0}$).

\medskip\noindent
\texttt{spharnad()} calculates $\Hd_\sigma \mathbf m$\quad($\mathbf m$ is a spectral type with refinement structure,\ $\sigma\in\mathbb Z^p_{\ge 0}$).

\medskip\noindent
\texttt{spfundroot()}  calculates $\mathcal U\mathbf m$\quad ($\mathbf m\in\mathcal P$).

\medskip\noindent
\texttt{spmcr()} calculates $\mathcal R\mathbf m$ with showing steps of the Katz reductions $\mathbf m\to\mathcal R\mathbf m
\quad(\mathbf m\in\mathcal P)$, 

\quad where the multiplicities $m_{j,\nu}$ may contain infinitesimal positive numbers
(cf.~Remark~\ref{remark:epsilon}).

\medskip\noindent
\texttt{spDEfromE8()} computes all $\mathbf m\in\mathcal F^T$
such that $\mathbf m\sim\mathbf n$ for $T=D_4,E_6,E_7$\quad($\mathbf n\in\mathcal F^{E_8}$).

\medskip\noindent
\texttt{linprog()} performs linear programming that may involve
infinitesimal positive numbers.

\medskip\noindent
\texttt{spharnadcv()} 
computes all $\mathcal R\Hd_\sigma\mathbf n$
such that
$
\mathbf n\underset{\mc}\sim\mathbf m$ and 
$
\ord\mathbf n\le\ord\mathbf m+\ell
$\quad
($\mathbf m\in\mathcal F$ and $\ell\in\mathbb Z_{\ge0}$).

\medskip\noindent
\texttt{sp2tab()} generates a \TeX\ table for a given list,
such as $\mathcal F_\ell/\!\sim$.

\medskip\noindent
\texttt{spfunddiag()} generates a \TeX\ figure using \texttt{TikZ}
that displays the relations $\sim$ for a given set of spectral types.

\medskip
Here the arguments of the functions are indicated in  "(\ )", for example,
($\mathbf m\in\mathcal P$). 

The examples in \S\ref{sec:example} were obtained by using \texttt{spfundclass()}, \texttt{sp2harnadcv()},
\texttt{sp2tab()} and \texttt{spfunddiag()}.
The lists $\mathcal F_\ell/\!\sim$ with $\ell\ge -100$ obtained by \texttt{spfundclass()}
are available (cf.~\cite{Ospect10}). 
\end{remark}

\begin{remark}
We assume that the differential equations have no ramified irregular singularities.
It is expected that the results of this paper can be extended to the ramified case
by studying of such singularities, including their confluences, unfoldings and
degenerations of HTL canonical forms (cf.~\cite{Kawakami2025}).
An extension to linear difference equations is also interesting.
\end{remark}
%%%%%%%%%%%%%%%%%%%%%%%%%%%%%%%%%%%%%%%%%%%%%%%%%%%%%%%%
%%%%%%%%%%%%%%%%%%%%%%   4.2  %%%%%%%%%%%%%%%%%%%%%%%%%%
\subsection{$D_4$-fundamental spectral types}
\label{sec:D4}
%%%%%%%%%%%%%%%%%%%%%%%%%%%%%%%%%%%%%%%%%%%%%%%%%%%%%%%%
In this subsection, we examine $\mathcal F_{[\mathbf n]}^{D_4}$ for $\mathbf n\in\mathcal F^{E_8}$.
Here, we define as follows.
\begin{definition}
For $\mathbf n\in\mathcal P$ and $T=E_6,\,E_7,\,E_8$ or $D_4$, we set
\begin{align*}
 \mathcal F_{[\mathbf n]}^T:=\mathcal F_{[\mathbf n]}\cap\mathcal F^T.
\end{align*}
\end{definition}

For an ordered $D_4$-fundamental spectral type
\begin{equation}\label{eq:D4}
\mathbf m
  =m_{11}m_{12},\,m_{21}m_{22},\,m_{31}m_{32},\,m_{41}m_{42}\cdots m_{4,n_4},   
\end{equation}
we calculate $\mathcal U\mathbf m$ by Theorem~\ref{thm:RSm}. 
We may assume $m_{11}\ge m_{21}\ge m_{31}$.
Then
\begin{align*}
 \begin{split}
  m_{11}&\ge m_{12}>0,\ m_{21}\ge m_{22}>0,\ m_{31}\ge m_{32},\\ 
  m_{41}&\ge m_{42}\ge\cdots \ge m_{4,n_4}>0,\ 
  m_{11}\ge m_{21}\ge m_{31},\\
  n&:=\ord\mathbf m=m_{11}+m_{12}=m_{21}+m_{22}=m_{31}+m_{32}=m_{41}+\cdots+m_{4,n_4},\\
 d_{1,1,1,1}(\mathbf m)&=m_{11}+m_{21}+m_{31}+m_{41}-2n=m_{31}+m_{41}-m_{12}-m_{22}\le 0.
  \end{split}
\end{align*}
Since $m_{41}\le m_{12}+m_{22}-m_{31}\le m_{12}$, we have 
\begin{equation}\label{eq:D4ineq}
 m_{11}\ge m_{21}\ge m_{31}\ge m_{32}\ge m_{22}\ge m_{12}\ge m_{41}\ge \cdots\ge m_{4n_4}>0
\end{equation}
and $\mathbf m$ is expressed in normal form.
Hence, under the notation \eqref{eq:an}, 
\begin{align*}
a_1&=m_{11},\ a_2=m_{21},\ a_3=m_{31},\ a_4=m_{32},\ a_5=m_{22},\ a_6=m_{12},\ a_{6+k}=m_{4,k}
\quad(1\le k\le n_4),\\
D_4&=2n-a_1-a_2=m_{12}+m_{22}=a_5+a_6,\ a_4+a_6\ge D_4\ge a_5+a_5
\end{align*}
and therefore we may put $N_{\mathbf m}=2$. Then, Theorem~\ref{thm:RSm}
with the relations
\begin{align*}
 (a_2+D_4)-(a_6+a_5+a_3)&=(a_2+a_5+a_6)-(a_3+a_5+a_6)=a_2-a_3\ge0,\\
 (a_2+D_4)-(a_7+a_4+a_3)&=(a_2+a_5+a_6)-(a_3+a_4+a_7)=a_6-a_7\ge0
\end{align*}
proves
\begin{align*}
 \mathcal R\mathcal S\mathbf m
 &=\mathbf m^{(2)}\\
 &=\bigl[[ D_4+n-a_3-a_4,2D_4+n-a_3-a_4-a_5-a_6],\\
 &\quad\ \ ([n-a_1,n-a_2,D_4+n-a_3-a_4-a_6,D_4+n-a_3-a_4-a_5],\\
 &\qquad
  [D_4-a_3,D_4-a_4,a_7,a_8,\ldots]\bigr]\\
 &=\bigl[[m_{12}+m_{22},m_{12}+m_{22}],[m_{12},m_{22},m_{22},m_{12}],\\
 &\qquad
   [m_{12}+m_{22}-m_{32},m_{12}+m_{22}-m_{31},m_{41},\ldots,m_{4,n_4}]\bigr].
\end{align*}
Here we note that
\[
 m_{12}+m_{22}\ge m_{22}\ge m_{12}\ge m_{12}+m_{22}-m_{32}\ge m_{12}+m_{22}-m_{31}\ge m_{41}
 \ge\cdots\ge m_{4,n_4}.
\]

Set
\[
\begin{split}
 a_1&=a_2=m_{12}+m_{22},\ a_3=a_4=m_{22},\ a_5=a_6=m_{12},\\
 a_7&=m_{12}+m_{22}-m_{32},\ a_8=m_{12}+m_{22}-m_{31},\ a_{8+k}=m_{4k}\quad(1\le k\le n_4),\\
 n&=a_1+a_2=2(m_{12}+m_{22}),\ D_3=n-a_1=m_{12}+m_{22},\\
 \mathbf m&=\bigl[[a_1,a_2],[a_3,a_4,a_5,a_6],[a_7,a_8,\cdots,a_{8+n_4}]\bigr].
\end{split}
\]
Since $a_3+a_5=D_3$, Theorem~\ref{thm:RSm} proves
\begin{align*}
 \mathcal R\mathcal S\mathbf m&=\mathbf m^{(1)}\\
 &=\bigl[[D_3+n-a_2-a_3,D_3+n-a_2-a_4],[n-a_1,n-a_2,D_3+n-a_2-a_3-a_4]\\
 &\qquad
   [D_3-a_2,a_5,a_6,\dots,a_{8+n_4}]\bigr]\\
 &=\bigl[[2m_{12}+m_{22},2m_{12}+m_{22}], [m_{12}+m_{22},m_{12}+m_{22},2m_{12}],\\
 &\qquad [0,m_{12},m_{12},m_{12}+m_{22}-m_{32},m_{12}+m_{22}-m_{31},m_{41},\dots,m_{4,n_5}]\bigr].
\end{align*}
Here we note that
\[
 \begin{split}
 &2m_{12}+m_{22}\ge m_{12}+m_{22}\ge 2m_{12}\ge m_{12}\ge m_{12}+m_{22}-m_{32}
  \ge m_{12}+m_{22}-m_{31}\\
 &\ge m_{41}\ge\cdots\ge m_{4,n_4}.
 \end{split}
\]
%%%%%%%%%%%%%%%%%
\begin{theorem}\label{thm:D4} {\rm (1)} \ 
Let  $\mathbf m=m_{11}m_{12},m_{21}m_{22},m_{31}m_{32},m_{41}\cdots m_{4,n_4}$
be an ordered fundamental spectral type satisfying $m_{11}\ge m_{21}\ge m_{31}$.  
Then
\begin{align}
\begin{split}
 \mathcal R\mathcal S\mathbf m
  &=\bigl[[m_{12}+m_{22},m_{12}+m_{22}],[m_{12},m_{12},m_{22},m_{22}],\\
  &\qquad [m_{12}+m_{22}-m_{32},m_{12}+m_{22}-m_{31},m_{41},\ldots,m_{4,n_4}]\bigr]
\end{split}\notag%\label{eq:D4E7}
\intertext{and}
 \begin{split}
  \mathcal U\mathbf m
   &=\bigl[[2m_{12}+m_{22},2m_{12}+m_{22}],[m_{12}+m_{22},m_{12}+m_{22},2m_{12}]\\
   &\qquad[m_{12},m_{12},m_{12}+m_{22}-m_{32},m_{12}+m_{22}-m_{31},m_{41},\ldots,m_{4,n_4}]\bigr].
 \end{split}\label{eq:D4E80}
\end{align}
All the expressions above are ordered.

\noindent
{\rm (2)} \ 
For any $\mathbf n\in\mathcal F^{E_8}$, we have 
\[ |\mathcal F_{[\mathbf n]}^{D_4}|\le 1.\]  
Moreover, $|\mathcal F_{[\mathbf n]}^{D_4}|=1$ if and only if
\begin{equation}\label{eq:D4E81}
 \bar m_{11}=\bar m_{12}=\bar m_{21}+\bar m_{31},\ 
 \bar m_{21}=\bar m_{22} \text{ \ and \ }\bar m_{31}=\bar m_{32}
\end{equation}
under the ordered expression
\[
\mathbf n=\bar m_{11}\bar m_{12},\,\bar m_{21}\bar m_{22}\bar m_{23},\,
   \bar m_{31}\bar m_{32}\bar m_{33}\bar m_{34}\bar m_{35}\cdots\bar m_{3,n_3},
\]
namely, $\mathbf n$ has the ordered expression
\begin{align}\label{eq:D4E82}
 \mathbf n&=\bigl[[\bar m_{21}+\bar m_{31},\bar m_{21}+\bar m_{31}],
    [\bar m_{21},\bar m_{21}, 2\bar m_{31}], 
 [\bar m_{31},\bar m_{31},\bar m_{33},\bar m_{34},\bar m_{35},\ldots, \bar m_{3,n_3}]\bigr].
\end{align}
Then the element $\mathbf m\in\mathcal F_{[\mathbf n]}^{D_4}$ whose
expression is given by \eqref{eq:D4} is
\begin{align*}
 &\bigl[[\ord\mathbf m-\bar m_{31},\bar m_{31}],[\ord\mathbf m-\bar m_{21}+\bar m_{31},\bar m_{21}-\bar m_{31}],\\
 &\quad[\ord\mathbf m-\bar m_{21}+\bar m_{33},\bar m_{21}-\bar n_{33}],[\bar m_{35},\bar m_{36},
\ldots,\bar m_{3,n_3}]\bigr],
\end{align*}
where
\begin{align*}
 \ord\mathbf m=2\bar m_{21}-\bar m_{33}-\bar m_{34}.
\end{align*}
\end{theorem}
%%%%
\begin{proof}
The first part of the theorem has already been proved.

To prove the second part, suppose that $\mathbf n=\mathcal U\mathbf m$ for some 
$\mathbf m$ as in (1).
Note that \eqref{eq:D4E81} follows from the ordered expression \eqref{eq:D4E80}.
Then, $\bar m_{23}=(\bar m_{11}+\bar m_{12})-(\bar m_{21}+\bar m_{22})=2\bar m_{31}$
and hence we obtain \eqref{eq:D4E82}.
Comparing \eqref{eq:D4E80} and \eqref{eq:D4E82}, we obtain
\begin{align*}
 m_{12}&=\bar m_{31},& m_{22}&=\bar m_{21}-\bar m_{31},\\
 m_{32}&=m_{12}+m_{22}-\bar m_{33}=\bar m_{21}-\bar m_{33},& m_{31}&=m_{12}+m_{22}-\bar m_{34}=\bar m_{21}-\bar m_{34},\\
 m_{4,\nu}&=\bar m_{3,\nu+4}\quad(1\le\nu\le n_4-4),&
 \ord\mathbf m&=m_{31}+m_{32}=2\bar m_{21}-\bar m_{33}-\bar m_{34},\\
 m_{11}&=\ord\mathbf m-m_{12},& m_{21}&=\ord\mathbf m-m_{22}.
\end{align*}
Hence the map $\mathcal U:\mathcal F^{D_4}\to\mathcal F^{E_8}$ is injective.

Moreover, since the expressing \eqref{eq:D4E82} of $\mathbf n$ is ordered and 
fundamental, we have
\begin{align*}
 m_{31}-m_{32}&=\bar m_{33}-\bar m_{34}\ge0,\\
 m_{32}-m_{22}&=\bar m_{21}-\bar m_{33}\ge0,\\
 m_{22}-m_{12}&=\bar m_{21}-2\bar m_{31}\ge0,\\
 m_{12}-m_{41}&=\bar m_{31}-\bar m_{35}\ge0,\\
 m_{11}+m_{21}+m_{31}+m_{41}-2\ord\mathbf m&=
 m_{31}+m_{41}-(m_{12}+m_{22})=\bar m_{35}-\bar m_{34}\le 0,
\end{align*}
where the $m_{j,\nu}$ are defined from $\mathbf n$ by the above relations.
This means that  $\{m_{j,\nu}\}$ is an ordered $D_4$-fundamental spectral type in normal form. 
Thus, ${m_{j,\nu}}$ is an ordered $D_4$-fundamental spectral type in
\end{proof}
%%%%%%%%%%%%%%%%%%%%%%%%%%%%%%%%%%%%%%%%%%%%%%%%%%%%%%%
%%%%%%%%%%%%%%%%%%%%%%   4.3  %%%%%%%%%%%%%%%%%%%%%%%%%%
\subsection{$E_7$-fundamental spectral types}
\label{sec:E7}
%%%%%%%%%%%%%%%%%%%%%%%%%%%%%%%%%%%%%%%%%%%%%%%%%%%%%%%
Let
\begin{equation*}
 \mathbf m=\bigl[[m_{11},m_{12}],[m_{21},m_{22},m_{23},m_{24}],
  [m_{31},m_{32},\ldots,m_{3n_3}]\bigr]
\end{equation*}
be an ordered fundamental $E_7$-tuple, and put $n=\ord\mathbf m$. 
Then
\begin{align}\begin{split}
%0<n-m_{11}\le m_{11},\\ 
&m_{11}\ge m_{12}>0,\ m_{21}\ge m_{22}\ge m_{23}\ge m_{24}>0,\ 
 m_{31}%\ge m_{32}
 \ge \cdots \ge m_{3{n_3}}>0\\
&n=m_{11}+m_{12}=m_{21}+m_{22}+m_{23}+m_{24}=m_{31}+\cdots+m_{3n_3},\\ 
&d_{1,1,1}(\mathbf m)=m_{11}+m_{21}+m_{31}-n=m_{21}+m_{31}-m_{12}\le 0. 
\end{split}\label{eq:E7cond}\end{align}
It follows that
\begin{align*}
 m_{31}&\le m_{12}-m_{21}\le \tfrac n2-\tfrac12(m_{21}+m_{22})
  =\tfrac12(m_{23}+m_{24})\le m_{23},\\
 m_{11}&\ge m_{12}\ge m_{21}\ge m_{22}\ge m_{23}\ge m_{31}\ge \cdots\ge m_{3n_3},
 \ m_{23}\ge m_{24}.
\end{align*}
Under the notation  \eqref{eq:an}, 
\begin{align}\begin{split}
 &a_1=m_{11},\ a_2=m_{12},\ a_3=m_{21},\ a_4=m_{22},\ a_5=m_{23},\ 
 a_6=m_{24}\text{ \ or \ }m_{31},\\
 &a_3+a_5\ge D_3=n-m_{11}=m_{12}\ge a_5+a_7.
\end{split}\label{eq:E7ma}\end{align}
Here the last inequality follows from the fact that Lemma~\ref{lem:p-1}: 
if $a_5+a_7>D_3$, then $|I_{[1,7]}|=2$.  
Similarly, if $a_1+a_3+a_6>n$, then $a_6=m_{24}$.
Thus, Theorem~\ref{thm:RSm} shows that $\mathcal U\mathbf m$
is one of $\mathbf m^{(2)}$,  $\mathbf m^{(3)}$,  $\mathbf m^{(3,1)}$
and $\mathbf m^{(3,2)}$.  
More precisely, we have the following cases.

If $m_{11}-m_{21}-m_{23}=a_1+a_4+m_{24}-n>0$, 
then $a_6=m_{24}$ and 
\begin{align*}\mathcal U\mathbf m&=\mathbf m^{(3)}=
\bigl[[2D_3+n-a_{[2,4]}-a_6,D_3+n-a_{[2,5]}],\\
&\qquad
  [n-a_1,2D_3+n-a_{[2,6]},D_3+n-a_{[2,4]}],
    [D_3-a_2,D_3-a_3,D_3-a_4,a_7,a_8,\ldots]\bigr]
\end{align*}

If $m_{22}+m_{23}-m_{11}=a_2+a_4+a_5-n>0$, 
\begin{align*}\mathcal U\mathbf m&=\mathbf m^{(3,1)}=
\bigl[[2D_3+n+a_1-a_2-a_{[2,5]},2D_3+n-a_{[2,5]}],\\
&\quad
  [D_3+n-a_2-a_4-a_5,D_3+n-a_2-a_3-a_5,D_3+n-a_2-a_3-a_4],\\
&\qquad
    [D_3-a_2,2D_3+a_1-a_2-a_3-a_4-a_5,a_6,a_7,a_8,\ldots]\bigr].
\end{align*}

If $m_{21}+m_{24}-m_{11}=a_2+a_3+m_{24}-n>0$,
we have $m_{24}=a_6$ and  
\begin{align*}\mathcal U\mathbf m&=\mathbf m^{(3,2)}=
\bigl[[D_3+n-a_2-a_3,3D_3+n+a_1-2a_2-2a_3-a_4-a_5-a_6],\\
&\quad
  [D_3+n-a_2-a_3-a_6,D_3+n-a_2-a_3-a_5,D_3+n-a_2-a_3-a_4],\\
&\qquad
    [D_3-a_2,D_3-a_3,D_3+a_1-a_2-a_3,a_6,a_7,a_8,\ldots]\bigr].
\end{align*}

In the remaining case,
\begin{align*}\mathcal U\mathbf m&=\mathbf m^{(2)}=
\bigl[[D_3+n-a_2-a_3,2D_3+n-a_{[2,5]}],\\
&\qquad
  [n-a_1,D_3+n-a_2-a_3-a_5,D_3+n-a_{[2,4]}],
    [D_3-a_2,D_3-a_3,a_6,a_7,a_8,\ldots]\bigr].
\end{align*}

Applying \eqref{eq:E7ma} to the above cases, we obtain the following theorem.
\begin{theorem}\label{thm:E7} $\Hr\mathbf m$ equals
\begin{align}
\label{E7-126}
\begin{split}
\mathbf m^{(3)}=
&\bigl[[m_{12}+m_{23},m_{12}+m_{24}],[m_{12},m_{12},m_{23}+m_{24}],\\
&\ \ 
  [m_{12}-m_{22},m_{12}-m_{21},m_{31},\ldots,m_{3n_3}]\bigr]
  \quad\text{if \ }m_{11}>m_{21}+m_{23},
\end{split}\\
%&\qquad\text{if \ }m_{11}>m_{21}+m_{23},\notag\\
%%
\label{E7-115}
\begin{split}
\mathbf m^{(3,1)}=
 &\bigl[[m_{11}+m_{24},m_{12}+m_{24}],[m_{21}+m_{24},m_{22}+m_{24},m_{23}+m_{24}],\\
 &\ \  [m_{24},m_{24},m_{31},\ldots,m_{3n_3}\bigr]
 \qquad\qquad\qquad\ \ \,\text{if \ }m_{11}<m_{22}+m_{23},
\end{split}\\
% &\qquad\text{if \ }m_{11}<m_{22}+m_{23},\notag\\
\label{E7-216}
\begin{split}
\mathbf m^{(3,2)}=
 &\bigl[[m_{22}+m_{23}+m_{24},m_{22}+m_{23}+m_{24}],[m_{22}+m_{23},m_{22}+m_{24},m_{23}+m_{24}],\\
 &\ \ [m_{11}-m_{21},m_{12}-m_{21}, m_{31},\ldots,m_{3n_3}\bigr]
 \quad\text{if \ }m_{11}<m_{21}+m_{24},
\end{split}\\
%&\qquad\text{if \ }m_{11}<m_{21}+m_{24},\notag\\
%%
\label{E7-0}
\begin{split}
\mathbf m^{(2)}= 
&\bigl[[m_{22}+m_{23}+m_{24},m_{12}+m_{24}],[m_{12},m_{22}+m_{24},m_{23}+m_{24}],\hspace{5mm}\\
 &\ \  [m_{24},m_{12}-m_{21},m_{31},\ldots,m_{3n_3}\bigr]
\end{split}\\
&\quad\text{if \ }m_{11}\ge m_{22}+m_{23},\ m_{11}\le m_{21}+m_{23}\text{ \ and \ } m_{11}\ge m_{21}+m_{24}.
\notag
\end{align}
Although the tuples \eqref{E7-216} and \eqref{E7-0} are ordered, 
$m_{24}<m_{31}$ may occur in \eqref{E7-115}, 
and $m_{12}<m_{23}+m_{24}$ may happen in \eqref{E7-126}.
The tuple \eqref{E7-0} satisfies $d_{1,1,1}(\Hr\mathbf m)=0$. 
\end{theorem}

\begin{remark}\label{rem:E8to7}
Using Theorem~\ref{thm:E7}, we can determine $\mathcal F_{[\mathbf n]}^{E_7}$ for 
for any $\mathbf n\in\mathcal F^{E_8}$ as follows.
  
Since every $\mathbf m\in\mathcal F_{[\mathbf n]}^{E_7}$ 
is obtained from  one of the four cases in the theorem, 
$\mathbf n$ satisfies at least one of the following four conditions:
\begin{align}
&\bar m_{21}=\bar m_{22}\text{ \ or \ }\bar m_{22}=\bar m_{23},\label{eq:E7C1}\\
&\text{there exists $\nu$ satisfying $\bar m_{3,\nu}=\bar m_{3,\nu+1}$,}\label{eq:E7C0}\\
&\bar m_{11}=\bar m_{12},\label{eq:E7C2}\\
&\bar m_{12}=\bar m_{21}+\bar m_{31}.\label{eq:E7C3}
\end{align}
Here, $\mathbf n$ is written in the ordered form
\begin{equation}\label{eq:n}
 \mathbf n=\bar m_{11}\bar m_{12},\,\bar m_{21}\bar m_{22}\bar m_{23},\,
   \bar m_{31}\bar m_{32}\bar m_{33}\cdots\bar m_{3,\bar n_3}. 
\end{equation} 
For example, if \eqref{eq:E7C3} holds, $\mathbf m$ may be given by \eqref{E7-0}, which implies
\begin{align*}
 m_{21}&=\bar m_{21}-\bar m_{32},\ 
 m_{22}=\bar m_{22}-\bar m_{31},\ 
 m_{23}=\bar m_{23}-\bar m_{31},\ 
 m_{24}=\bar m_{31},\\
 \ord\mathbf m&=\bar m_{11}+\bar m_{21}-\bar m_{32},\  m_{12}=\bar m_{21},\ 
 m_{11}=\bar m _{11}-\bar m_{32}. 
\end{align*}
Checking inequalities \eqref{eq:E7cond} and those in \eqref{E7-0}, we determine
whether $\mathbf m$ is indeed given by \eqref{E7-0}.

In the same way, we can determine the elements $\mathbf m\in\mathcal F_{[\mathbf n]}^{E_7}$ 
corresponding to each case in Theorem~\ref{thm:E7}.  
Here we note that there may be several $\mathbf m$
corresponding to the case \eqref{E7-115} with $\nu$ satisfying \eqref{eq:E7C0},
whereas  there is at most one $\mathbf m$ 
corresponding to each of the other cases. 
%.  But there exists
%at most one $\mathbf m$ corresponding to any one of the other cases. 
\end{remark}
%%%
\begin{exmp} (1) \ Put 
\begin{equation*}
 \mathbf n=\texttt{ee,a99,443333221111}\in\mathcal F^{E_8}.
\end{equation*}
Then $\idx\mathbf n=-50$, $\#\mathcal F_{[\mathbf n]}^{E_7}=5$
and $\mathcal F_{[\mathbf n]}^{E_7}$ is obtained as is explained in Remark~\ref{rem:E8to7}:

Case \eqref{E7-126} : % [2]
\ \texttt{b9,5555,3333221111}

Case \eqref{E7-115} %[3]
with 
$m_{24}=1,\ 2,\ 3$ : 

\quad \texttt{dd,9881,4433332211}\qquad \texttt{cc,8772,4433331111}\qquad \texttt{bb,7663,4433221111}

Case \eqref{E7-0} : %[0]
\ \texttt{aa,6554,3333221111}\\
Here $\texttt{a}=10,\,\texttt{b}=11,\,\texttt{c}=12,\,\texttt{d}=13,\,\texttt{e}=14,\ldots$

\smallskip\noindent
(2) \ 
For any $\mathbf n\in \mathcal F^{E_8}$ satisfying $\idx \mathbf n\ge -100$, 
we have calculated $\mathcal F_{[\mathbf n]}$  (cf.~\cite{Ospect10}) and found 
$\bar{\mathbf n}\in\mathcal F_{[\mathbf n]}^{E_7}$ 
such that $\ord \bar{\mathbf n}\ge \ord\mathbf m$
for any $\mathbf m\in\mathcal F_{[\mathbf n]}\setminus\{\mathbf n\}$.
However, in particular, 
\begin{equation}\label{eq:generic}
 \mathbf n=\bigl[[50,49],[34,33,32],[14,13,12,11,10,9,8,7,6,5,4]\bigr]
\end{equation}
with $\idx\mathbf n=-630$ does not satisfy any of the conditions 
in Remark~\ref{rem:E8to7}.   
Therefore, \[\mathcal F_{[\mathbf n]}^{E_7}=\emptyset.\]
\end{exmp}

%%%%%%%%%%%%%%%%%%%%%%%%%%%%%%%%%%%%%%%%%%%%%%%%
%%%%%%%%%%%%%%%%%%%%%%   4.4  %%%%%%%%%%%%%%%%%%
\subsection{$E_6$-fundamental spectral types}
\label{sec:E6}
%%%%%%%%%%%%%%%%%%%%%%%%%%%%%%%%%%%%%%%%%%%%%%%%
In this subsection, we calculate $\mathcal U\mathbf m$ for an $E_6$-fundamental spectral 
type $\mathbf m$. The result is as follows.
\begin{theorem}  Let
\begin{align*}
\mathbf m=\bigl[[m_{11},m_{12},m_{13}],[m_{21},m_{22},m_{23}],[m_{31},\ldots,m_{3n_3}]\bigr]
\end{align*}
be an ordered $E_6$-fundamental spectral type. 
Then $\Hr\mathbf m$ equals
\begin{align}
\label{E6-212}
 \begin{split}
  &\bigl[[m_{11}+2m_{13},m_{12}+2m_{13}],
   [m_{13}+m_{21},m_{13}+m_{22},m_{13}+m_{23}]\\
  &\quad [m_{13},m_{13},m_{13},m_{31},\ldots]\bigr]
 \end{split}\\
 &\qquad\text{if \ }
  m_{21}<m_{12},
 \notag
 \allowdisplaybreaks\\
%%
%216
\label{E6-126}
 \begin{split}
  &\bigl[[m_{13}+m_{22}+m_{23},m_{12}+2m_{13}],
   [m_{12}+m_{13},m_{13}+m_{22},m_{13}+m_{23}], \\
  &\quad
    [m_{13},m_{13},m_{12}+m_{13}-m_{21},m_{31},\ldots]\bigr]
 \end{split}\\
 &\qquad\text{if \ }
   m_{22}\le m_{12}\le m_{21},
 \notag
 \allowdisplaybreaks\\
\label{E6-223}
 \begin{split}
  &\bigl[[m_{12}+m_{13}+m_{23},m_{12}+m_{13}+m_{23}],[m_{12}+m_{13},m_{12}+m_{23},m_{13}+m_{23}],\\
  &\quad [m_{23},m_{12}+m_{13}-m_{22},m_{12}+m_{13}-m_{21},m_{31},\ldots]\bigr]
 \end{split}\\
 &\qquad\text{if \ }
  m_{12}<m_{22}\text{ \ and \ }m_{13}\ge m_{23},
 \notag
 \allowdisplaybreaks\\
%226
 \begin{split}
 & \bigl[[m_{12}+m_{13}+m_{23},m_{12}+2m_{13}],
   [m_{12}+m_{13},m_{12}+m_{13},m_{13}+m_{23}],\\
 &\quad
  [m_{13},m_{12}+m_{13}-m_{22},m_{12}+m_{13}-m_{21},m_{31},\ldots]\bigr]
 \end{split}\\
 &\qquad\text{if \ }
 m_{23}\le m_{12} < m_{22}\text{ \ and \ }m_{13}< m_{23},
 \notag
 \allowdisplaybreaks\\
%133
 \begin{split}
  &\bigl[[2m_{12}+m_{13},m_{12}+2m_{13}],\ 
  [m_{12}+m_{13},m_{12}+m_{13},m_{12}+m_{13}]\\
  &\quad [n-m_{11}-m_{23},n-m_{11}-m_{22},n-m_{11}-m_{21},m_{31},\ldots]\bigr]
 \end{split}\\
 &\qquad\text{if \ }
  m_{12} < m_{23}\text{ \ and \ }m_{13}< m_{23}.
 \notag
\end{align}
In the case \eqref{E6-212}, it may happen that $m_{13}<m_{31}$.
In the other cases, the above tuples in $\mathcal F^{E_8}$ are ordered.
\end{theorem}
%%%
\begin{proof} Note that
\begin{equation*}\begin{split}
&m_{11}\ge m_{12}\ge m_{13}>0,\ m_{21}\ge m_{22}\ge m_{23}>0,\ m_{31}\ge \cdots\ge m_{3n_3}>0,\\
&n=m_{11}+m_{12}+m_{13}=m_{21}+m_{22}+m_{23}=m_{31}+\cdots+m_{3n_3},\\
&m_{11}\ge m_{21}\ge m_{31},\\
&m_{12}+m_{13}\ge m_{21}+m_{31}.
\end{split}\end{equation*}
We calculate $\mathcal R\mathcal T m$ as the proof of Theorem~\ref{thm:RSm}. 
Hence, consider the following transformations.
\begin{align*}
&\mathcal T\mathbf m\simeq\bigl[[2n,n],[n,n,n],[m_{11},m_{21},m_{22},m_{12},m_{23},m_{13},m_{31},\ldots]\bigr]\\
&\xrightarrow[1,1,1]{m_{11}}
\xrightarrow[1,2,2]{m_{21}}
\xrightarrow[1,3,3]{m_{22}}
\\
&\bigl[[2n-m_{11}-m_{21}-m_{22},\underline{n}],
  [n-m_{11},n-m_{21},\underline{n-m_{22}}],[0,0,0,\underline{m_{12}},m_{23},m_{13},m_{31},\ldots]\bigr]\\
&\xrightarrow[2,3,4]{m_{21}-m_{13}}
\quad((n-m_{22})+m_{21}-(2n-m_{11}-m_{21}-m_{22})=m_{21}-m_{13},\ \ m_{21}\ge\tfrac n3\ge m_{13})\\
&\bigl[[m_{12}+m_{13}+m_{23},m_{13}+m_{22}+m_{23}],[m_{12}+m_{13},m_{22}+m_{23},m_{13}+m_{23}],\\
&\quad
 [0,0,0,n-m_{11}-m_{21},m_{23},m_{13},m_{31},\ldots]\bigr].
\end{align*}
Here we note the relations
\begin{align*}
&(m_{12}+m_{13}+m_{23}) - (m_{13}+m_{22}+m_{23})=m_{11}-m_{22},\\
&(m_{22}+m_{23})-(m_{12}+m_{13})=m_{11}-m_{12}\ge 0, \\
&(m_{22}+m_{23})-(m_{13}+m_{23})=m_{22}-m_{13}.
\end{align*}

\medskip
\noindent
\underline{Case \ %[1,2,6] 
$m_{12}\ge m_{22}\ (\Rightarrow m_{13}\le m_{23})$} : 
\begin{align*}
&\bigl[[\underline{m_{12}+m_{13}+m_{23}},m_{13}+m_{22}+m_{23}],
   [m_{12}+m_{13},\underline{m_{22}+m_{23}},m_{13}+m_{23}],\\
&\quad
 [0,0,0,n-m_{11}-m_{21},\underline{m_{23}},m_{13},m_{31},\ldots]\bigr]\\
&\xrightarrow[1,2,5]{m_{23}-m_{13}}\\
&\bigl[[m_{12}+2m_{13},\underline{m_{13}+m_{22}+m_{23}}],
  [\underline{m_{12}+m_{13}},m_{13}+m_{22},m_{13}+m_{23}], \\
&\quad
    [0,0,0,\underline{m_{12}+m_{13}-m_{21}},m_{13},m_{13},m_{31},\ldots]\bigr].
\end{align*}
Here, we remark
\begin{align*}
&\quad(m_{13}+m_{22}+m_{23})-(m_{12}+2m_{13})=m_{11}-m_{21}\ge 0,\\
&\quad(m_{12}+m_{13})-(m_{13}+m_{22})=(m_{12}+m_{13}-m_{22})-m_{13}=m_{12}-m_{22}\ge0,\\
%&\quad(m_{12}+2m_{13})-(m_{13}+m_{22})-m_{13}=m_{12}-m_{22}\ge 0,\\
&\quad m_{13}-(m_{12}+m_{13}-m_{21})=m_{21}-m_{12},\\
&\quad m_{12}+m_{13}-m_{21}=n-m_{11}-m_{21}\ge m_{31},\\
&\quad(m_{12}+2m_{13})-(m_{12}+m_{13})-m_{13}=0.
\end{align*}
Hence, if $m_{21}\ge m_{12}$, we have the theorem.

We assume $m_{12}>m_{21}$ and apply $\mc_{(2,1,4)}$ to the last tuple.  Then
\begin{align*}
\xrightarrow[2,1,4]{m_{12}-m_{21}} \ 
&\bigl[[m_{12}+2m_{13},m_{11}+2m_{13}],
[m_{13}+m_{21},m_{13}+m_{22},m_{13}+m_{23}]\\
&\quad [0,0,0,m_{13},m_{13},m_{13},m_{31},\ldots]\bigr],
\end{align*}
which is $E_8$-fundamental since
\begin{align*}(m_{12}+2m_{13})-(m_{13}+m_{21})-m_{13}&=m_{12}-m_{21}>0,\\
(m_{12}+2m_{13})-(m_{13}+m_{21})-m_{31}
&=m_{12}+m_{13}-m_{31}\\
&=n-m_{11}-m_{21}-m_{31}\ge0.
\end{align*}
When 
\begin{align*}
 \mathbf m=972,6^3,3^6 \text{\ \ or \ \ }\mathbf m=972,6^3,1^{18},
\end{align*}
$m_{13}<m_{31}$ or $m_{13}>m_{31}$, respectively.

\medskip
\noindent
\underline{Case \ %[2,2.3] : 
$m_{12}<m_{22}$ \ and \ $m_{13}\ge m_{23}$} : 
\begin{align*}
&\bigl[[m_{12}+m_{13}+m_{23},\underline{m_{13}+m_{22}+m_{23}}],[m_{12}+m_{13},\underline{m_{22}+m_{23}},m_{13}+m_{23}],\\
&\quad
 [0,0,0,n-m_{11}-m_{21},m_{23},\underline{m_{13}},m_{31},\ldots]\bigr]
\allowdisplaybreaks\\
&\xrightarrow[2,2,6]{m_{22}-m_{12}}\\  % P53
&\bigl[[m_{12}+m_{13}+m_{23},m_{12}+m_{13}+m_{23}],[\underline{m_{12}+m_{13}},m_{12}+m_{23},m_{13}+m_{23}],\\
&\qquad [0,0,0,m_{12}+m_{13}-m_{21},\underline{m_{23}},m_{12}+m_{13}-m_{22},m_{31},\ldots]\bigr].
\intertext{Here}
%%%
&\quad m_{23}-(m_{12}+m_{13}-m_{22})=m_{22}+m_{23}-m_{12}-m_{13}=m_{11}-m_{21}\ge0,\\
&\quad m_{12}+m_{13}-m_{22}\ge m_{12}+m_{13}-m_{21}=n-m_{11}-m_{21}\ge m_{31},\\
&\quad(m_{12}+m_{13}+m_{23})-(m_{12}+m_{13})-m_{23}=0.
\end{align*}
Hence we have the ordered $E_8$-fundamental tuple \eqref{E6-223}.

\medskip
\noindent
\underline{Case \ %[2,2,6] : 
$m_{12} < m_{22}$ \ and \ $m_{13}< m_{23}$} : 
\begin{align*}
&\bigl[[m_{12}+m_{13}+m_{23},\underline{m_{13}+m_{22}+m_{23}}],[m_{12}+m_{13},\underline{m_{22}+m_{23}},
  m_{13}+m_{23}\bigr],\\
&\quad
 [0,0,0,n-m_{11}-m_{21},\underline{m_{23}},m_{13},m_{31},\ldots]]
\allowdisplaybreaks\\
&\xrightarrow[2,2,5]{m_{11}-m_{21}}\\
&\mathbf n=\bigl[[\underline{m_{12}+m_{13}+m_{23}},m_{12}+2m_{13}],
 [m_{12}+m_{13},m_{12}+m_{13},m_{13}+m_{23}],\\
&\quad\qquad
  [0,0,0,m_{12}+m_{13}-m_{21},m_{12}+m_{13}-m_{22},\underline{m_{13}},m_{31},\ldots]\bigr].
\intertext{Here}
&\quad (m_{12}+m_{13}+m_{23})-(m_{12}+2m_{13})=m_{12}-m_{13}\ge 0,\\
&\quad(m_{12}+m_{13})-(m_{13}+m_{23})=m_{12}-m_{23},\\
&\quad m_{13}-(m_{12}+m_{13}-m_{22})=m_{22}-m_{12}>0,\\
&\quad m_{12}+m_{13}-m_{22}\ge m_{12}+m_{13}-m_{21}=n-m_{11}-m_{21}\ge m_{31},\\
% &\quad (m_{12}+2m_{23})-(m_{13}+m_{23})-m_{13}=(m_{12}-m_{13})+2(m_{23}-m_{13})>0.
&\quad(m_{12}+m_{13}+m_{23})-(m_{12}+m_{13})-m_{13}=m_{23}-m_{13}>0.\\
\intertext{Then, if $m_{12}\ge m_{23}$, $\mathbf n$ is  $E_8$-fundamental.  
We assume $m_{12}< m_{23}$. Then}
&\mathbf n\xrightarrow[1,3,6]{m_{23}-m_{12}}
\bigl[[\underline{2m_{12}+m_{13}},m_{12}+2m_{13}],\ 
 [\underline{m_{12}+m_{13}},m_{12}+m_{13},m_{12}+m_{13}],\\
&\qquad\qquad [0,0,0,n-m_{11}-m_{21},n-m_{11}-m_{22},\underline{n-m_{11}-m_{23}},m_{31},\ldots]\bigr].
\intertext{Since}
&\quad (m_{12}+2m_{13})-(m_{12}+m_{13})-(n-m_{11}-m_{23})=m_{23}-m_{12}>0,
\end{align*}
this spectral type equals  $\Hr\mathbf m$.
\end{proof}
\begin{remark}
For $\mathbf n\in\mathcal F^{E_8}$ with $\idx\mathbf \ge-100$, we have 
$\mathcal F_{[\mathbf n]}^{E_6}\ne\emptyset$ (cf.~Remark~\ref{remark:Risa}).
Note that 
$\mathcal F_{[\mathbf n]}^T=\emptyset$ for $\mathbf n$ given by 
\eqref{eq:generic} and $T=D_4,\,E_6$ or $E_7$. 
\end{remark}
%%%

%\bibliographystyle{abbrv}
%\bibliography{references}
\end{document}